\documentclass[10.3pt]{amsart}
\usepackage[english]{babel}

\usepackage{biblatex}
\bibstyle{amsplain}
\usepackage[mathlines,switch]{lineno} 

\usepackage{amsmath,amssymb,enumitem,verbatim,stmaryrd,xcolor,microtype,graphicx,mathscinet,mathtools}
\usepackage[T1]{fontenc}
\usepackage[utf8]{inputenc}
\usepackage[english]{babel}
\usepackage[letterpaper, top=3.2cm,bottom=3.2cm,left=2.5cm,right=2.5cm]{geometry}
\usepackage[bookmarksdepth=2,linktoc=page,colorlinks,linkcolor={red!80!black},citecolor={red!80!black},urlcolor={blue!80!black}]{hyperref}
\usepackage{tikz}\usetikzlibrary{matrix,arrows,decorations.markings,topaths,calc}
\usepackage{tikz-cd}
\tikzstyle{edge} = [fill opacity=.5,line cap=round, line join=round, line width=20pt]
\tikzstyle{thickedge} = [fill opacity=.5,line cap=round, line join=round, line width=30pt]
\usepackage{setspace}
\usepackage{amsthm,amsfonts,amscd}
\usepackage{amsxtra}
\usepackage{eucal}
\usepackage{mathrsfs}
\usepackage{color}
\usepackage{xcolor}
\usepackage{blkarray}
\usepackage{multirow}
\usepackage[center]{caption}

\usepackage{tcolorbox}
\usepackage{graphicx}
\usepackage{pdfpages}

\usepackage{etoolbox}
\makeatletter
\patchcmd{\@startsection}{\@afterindenttrue}{\@afterindentfalse}{}{}             
\patchcmd{\part}{\bfseries}{\bfseries\LARGE}{}{}
\patchcmd{\section}{\scshape}{\bfseries}{}{}\renewcommand{\@secnumfont}{\bfseries} 
\patchcmd{\@settitle}{\uppercasenonmath\@title}{\large}{}{}
\patchcmd{\@setauthors}{\MakeUppercase}{}{}{}
\addto{\captionsenglish}{} 
\addto{\captionsenglish}{} 
\addto{\captionsenglish}{} 
\makeatother

\usepackage{fancyhdr}

\DeclareFontFamily{OT1}{pzc}{}                                
\DeclareFontShape{OT1}{pzc}{m}{it}{<-> s * [1.10] pzcmi7t}{}
\DeclareMathAlphabet{\mathpzc}{OT1}{pzc}{m}{it}
\DeclareSymbolFont{sfoperators}{OT1}{bch}{m}{n} 
\DeclareSymbolFontAlphabet{\mathsf}{sfoperators} 
\DeclareSymbolFont{cmletters}{OML}{cmm}{m}{it}              
\DeclareSymbolFont{cmsymbols}{OMS}{cmsy}{m}{n}
\DeclareSymbolFont{cmlargesymbols}{OMX}{cmex}{m}{n}
\DeclareMathSymbol{\myjmath}{\mathord}{cmletters}{"7C}     \let\jmath\myjmath 
\DeclareMathSymbol{\myamalg}{\mathbin}{cmsymbols}{"71}     
\DeclareMathSymbol{\mycoprod}{\mathop}{cmlargesymbols}{"60}
\DeclareMathSymbol{\myalpha}{\mathord}{cmletters}{"0B}     \let\alpha\myalpha 
\DeclareMathSymbol{\mybeta}{\mathord}{cmletters}{"0C}      \let\beta\mybeta
\DeclareMathSymbol{\mygamma}{\mathord}{cmletters}{"0D}     \let\gamma\mygamma
\DeclareMathSymbol{\mydelta}{\mathord}{cmletters}{"0E}     \let\delta\mydelta
\DeclareMathSymbol{\myepsilon}{\mathord}{cmletters}{"0F}   \let\epsilon\myepsilon
\DeclareMathSymbol{\myzeta}{\mathord}{cmletters}{"10}      \let\zeta\myzeta
\DeclareMathSymbol{\myeta}{\mathord}{cmletters}{"11}       \let\eta\myeta
\DeclareMathSymbol{\mytheta}{\mathord}{cmletters}{"12}     \let\theta\mytheta
\DeclareMathSymbol{\myiota}{\mathord}{cmletters}{"13}      \let\iota\myiota
\DeclareMathSymbol{\mykappa}{\mathord}{cmletters}{"14}     \let\kappa\mykappa
\DeclareMathSymbol{\mylambda}{\mathord}{cmletters}{"15}    \let\lambda\mylambda
\DeclareMathSymbol{\mymu}{\mathord}{cmletters}{"16}        \let\mu\mymu
\DeclareMathSymbol{\mynu}{\mathord}{cmletters}{"17}        \let\nu\mynu
\DeclareMathSymbol{\myxi}{\mathord}{cmletters}{"18}        \let\xi\myxi
\DeclareMathSymbol{\mypi}{\mathord}{cmletters}{"19}        \let\pi\mypi
\DeclareMathSymbol{\myrho}{\mathord}{cmletters}{"1A}       \let\rho\myrho
\DeclareMathSymbol{\mysigma}{\mathord}{cmletters}{"1B}     \let\sigma\mysigma
\DeclareMathSymbol{\mytau}{\mathord}{cmletters}{"1C}       \let\tau\mytau
\DeclareMathSymbol{\myupsilon}{\mathord}{cmletters}{"1D}   \let\upsilon\myupsilon
\DeclareMathSymbol{\myphi}{\mathord}{cmletters}{"1E}       \let\phi\myphi
\DeclareMathSymbol{\mychi}{\mathord}{cmletters}{"1F}       \let\chi\mychi
\DeclareMathSymbol{\mypsi}{\mathord}{cmletters}{"20}       \let\psi\mypsi
\DeclareMathSymbol{\myomega}{\mathord}{cmletters}{"21}     \let\omega\myomega
\DeclareMathSymbol{\myvarepsilon}{\mathord}{cmletters}{"22}\let\varepsilon\myvarepsilon
\DeclareMathSymbol{\myvartheta}{\mathord}{cmletters}{"23}  \let\vartheta\myvartheta
\DeclareMathSymbol{\myvarpi}{\mathord}{cmletters}{"24}     \let\varpi\myvarpi
\DeclareMathSymbol{\myvarrho}{\mathord}{cmletters}{"25}    \let\varrho\myvarrho
\DeclareMathSymbol{\myvarsigma}{\mathord}{cmletters}{"26}  \let\varsigma\myvarsigma
\DeclareMathSymbol{\myvarphi}{\mathord}{cmletters}{"27}    \let\varphi\myvarphi

\newcommand\symdiff{\triangle}
\renewcommand{\emph}[1]{\textbf{#1}}

\newcommand{\skewpair}[1]{\{ #1 , #1 ^* \}}
\newcommand{\nostar}[1]{\overline{#1}}
\newcommand{\ocinterval}[2]{( #1 , #2 ]}

\newcommand{\lift}[1]{\mathrm{lift}(#1)}

\newcommand{\Min}{\mathrm{Min}}
\newcommand{\conj}[1]{\tilde{#1}}

\newcommand{\elem}{\mathrm{Elem}}

\newcommand{\ul}[1]{\underline{#1}}

\newcommand{\sgn}{\mathrm{sgn}}
\newcommand{\charac}{\mathrm{char}}

\newcommand{\supp}{\mathrm{Supp}}

\newcommand\C{{\mathbb C}}

\newcommand\F{{\mathbb F}}
\newcommand\K{{\mathbb K}}
\newcommand\N{{\mathbb N}}
\newcommand\bP{{\mathbb P}}
\newcommand\Q{{\mathbb Q}}
\newcommand\R{{\mathbb R}}
\newcommand\bS{{\mathbb S}}
\newcommand\T{{\mathbb T}}
\newcommand\U{{\mathbb U}}
\newcommand\Z{{\mathbb Z}}

\newcommand\cA{{\mathcal A}}
\newcommand\cB{{\mathcal B}}
\newcommand\cC{{\mathcal C}}
\newcommand\cD{{\mathcal D}}

\newcommand\cT{{\mathcal T}}
\newcommand\cV{{\mathcal V}}
\newcommand\cW{{\mathcal W}}

\renewcommand{\leq}{\leqslant}
\renewcommand{\geq}{\geqslant}

\newtheorem{theorem}{Theorem}[section]
\newtheorem{proposition}[theorem]{Proposition}
\newtheorem{lemma}[theorem]{Lemma}

\newtheorem{corollary}[theorem]{Corollary}

\theoremstyle{definition}
\newtheorem{definition}[theorem]{Definition}
\newtheorem{example}[theorem]{Example}

\theoremstyle{remark}
\newtheorem{remark}[theorem]{Remark}

\title{Orthogonal matroids over tracts}

\author{Tong Jin}
\email{tongjin@gatech.edu}
\address{School of Mathematics, Georgia Institute of Technology, Atlanta, USA}

\author{Donggyu Kim}
\email{donggyu@kaist.ac.kr}
\address{Department of Mathematical Sciences, KAIST, Daejeon, South Korea \and Discrete Mathematics Group, Institute for Basic Science (IBS), Daejeon, South Korea}

\date{\today}

\begin{document}


\begin{abstract}
We generalize Baker-Bowler's theory of matroids over tracts to orthogonal matroids, define orthogonal matroids with coefficients in tracts in terms of Wick functions, orthogonal signatures, circuit sets, and orthogonal vector sets, and establish basic properties on functoriality, duality, and minors. Our cryptomorphic definitions of orthogonal matroids over tracts provide proofs of several representation theorems for orthogonal matroids. In particular, we give a new proof that an orthogonal matroid is regular if and only if it is representable over $\F_2$ and $\F_3$, which was originally shown by Geelen~\cite{Geelen1996thesis}, and we prove that an orthogonal matroid is representable over the sixth-root-of-unity partial field if and only if it is representable over $\F_3$ and $\F_4$. 
\end{abstract}

\thanks{{\bf Acknowledgements.} The first author was partially supported by a Simons Foundation Travel Grant and NSF Research Grant DMS-2154224. The second author was supported by the Institute for Basic Science (IBS-R029-C1). We thank Matt Baker for suggesting this project and for helpful comments on an early draft of this paper, and thank Nathan Bowler for discussions on Theorem~\ref{thm:main-weak1}. }



\maketitle

\section{Introduction}\label{section:intro}

Let $F$ be a field, and let $V = F^{2n}$ be a $2n$-dimensional vector space over $F$ endowed with a symmetric non-degenerate bilinear form $Q$. We say a subspace $W \subseteq V$ is {\em isotropic} if $Q(W,W) = 0$, and {\em maximal isotropic} or {\em Lagrangian} if it is isotropic and of dimension $n$. Given a maximal isotropic subspace $W$ of $V$, one can associate to $W$ a point $w$ of $\bP^N(F)$ with coordinates $w_I$ indexed by the subsets $I \subseteq \{1, \dots, n\}$, where $N = 2^n - 1$. Just as the usual Grassmannian $G(r, n)$ parameterizes all $r$-dimensional subspaces of an $n$-dimensional vector space, all maximal isotropic subspaces of $V$ can be parameterized by the {\em Lagrangian orthogonal Grassmannian} $OG(n, 2n) \subseteq \bP^N(F)$. The Lagrangian orthogonal Grassmannian is a projective variety cut out by homogeneous quadratic polynomials known to physicists as the {\em Wick equations}~\cite{Mu10}. 

The combinatorial counterpart of Lagrangian orthogonal Grassmannians is the notion of a {\em Lagrangian orthogonal matroid}. For simplicity, we omit the adjective `Lagrangian' and call them {\em orthogonal matroids}. 

Let $E = [n] \cup [n]^\ast = \{1, \dots, n\} \cup \{1^\ast, \dots, n^\ast\}$ with the obvious involution $\ast: E \to E$ that induces an involution on the power set $\mathcal{P}(E)$, and denote by $X \symdiff Y$ the symmetric difference of two sets $X$ and $Y$. A subset $A \subseteq E$ is said to be {\em admissible} or a {\em subtransversal} if $A \cap A^* = \emptyset$. An $n$-element admissible subset is a {\em transversal}. 
We call $\{x, x^\ast\} \subseteq E$ with $x\in E$ a {\em divergence}. 

One of the simplest way to define orthogonal matroids is via the {\em symmetric exchange axiom}.

\begin{definition}
An {\em orthogonal matroid} on $E$ is a pair $M = (E, \cB)$, where the nonempty collection $\cB$ of transversals of $E$ satisfies the following axiom: if $B_1, B_2 \in \cB$, then for every divergence $\{x_1, x_1^\ast\} \subseteq B_1 \symdiff B_2$, there exists $\{x_2, x_2^\ast\} \subseteq B_1 \symdiff B_2$ with $\{x_2, x_2^\ast\} \ne \{x_1, x_1^\ast\}$ such that $B_1 \symdiff\{x_1, x_1^\ast, x_2, x_2^\ast\}  \in \cB$. 
\end{definition}

The finite set $E(M) := E$ is called the {\em ground set} of the orthogonal matroid, and $\cB(M) := \cB$ is the collection of {\em bases}. 

Orthogonal matroids were studied by various researchers from different perspectives. An equivalent definition of orthogonal matroids was firstly introduced by Kung in~\cite{Kung1978} in 1978 under the name of {\em Pfaffian structures}; see also~\cite{Kung1997}. Bouchet studied basic properties of orthogonal matroids, initially under the name of {\em symmetric matroids} and later {\em even $\symdiff$-matroids}, including their bases, independent sets, circuits, the rank function, a greedy algorithm, minors, and representation theory of orthogonal matroids over fields~\cite{Bouchet1987sym,Bouchet1988repre,BouchetMM1,BouchetMM2}. One can associate an orthogonal matroid to a graph embedded on an orientable surface~\cite{Bouchet1989map,Chun2019ribbon}. Orthogonal matroids also coincide with the class of {\em Coxeter matroids of type $D_n$} in the sense of \cite{Borovik2003}.

Tracts were introduced by Baker and Bowler in~\cite{Baker2019}, as an algebraic framework to represent matroids that simultaneously generalizes the notion of linear subspaces, matroids, valuated matroids, oriented matroids, and regular matroids. This framework provides short and conceptual proofs for many matroid representation theorems~\cite{BO21, BO23}. Recently in~\cite{JL2022}, Jarra and Lorscheid extended Baker-Bowler's theory to flag matroids, which also lie in the class of Coxeter matroids of type $A_n$ as ordinary matroids. An introduction to tracts will be given in Section~\ref{sec:tracts}. 

We generalize the theory of matroids over tracts in~\cite{Anderson2019} and~\cite{Baker2019} to orthogonal matroids, and show that there are (at least) three natural notions of orthogonal matroids over a tract~$F$, which we call {\em weak orthogonal $F$-matroids}, {\em moderately weak orthogonal $F$-matroids}, and {\em strong orthogonal $F$-matroids} in order of increasing strength. We give axiom systems for these in terms of Wick functions, orthogonal signatures, circuit sets, and vector sets, and prove the cryptomorphism for strong orthogonal $F$-matroids. 

\begin{theorem}
  Let $E = [n] \cup [n]^*$ and let $F$ be a tract. Then there are natural bijections between: 
  \begin{enumerate}
      \item Strong orthogonal $F$-matroids on $E$. 
      \item Strong orthogonal $F$-signatures on $E$.
      \item Strong $F$-circuit sets of orthogonal matroids on $E$.
      \item Orthogonal $F$-vector sets on $E$.
  \end{enumerate}
\end{theorem}

We also prove natural bijections between weaker notions.
\begin{theorem}
  There is a natural bijection between:
  \begin{enumerate}
    \item Weak orthogonal $F$-matroids on $E$.
    \item Weak $F$-circuit sets of orthogonal matroid on $E$.
  \end{enumerate}
\end{theorem}
\begin{theorem}
  There is a natural bijection between:
  \begin{enumerate}
    \item Moderately weak orthogonal $F$-matroids on $E$.
    \item Weak orthogonal $F$-signatures on $E$.
  \end{enumerate}
\end{theorem}

Our definitions show compatibility with various existing definitions in the following ways; see Section~\ref{sec:equiv of diff defs}. 

\begin{enumerate}
    \item If the support of a strong or weak orthogonal matroid on $E$ over $F$ is the lift of an ordinary matroid on $[n]$, then an orthogonal matroid on $E$ over $F$ is the same thing as a strong or weak matroid on $[n]$ over $F$ in the sense of~\cite{Baker2019}. 
    
    \item A strong or weak orthogonal matroid over the Krasner hyperfield $\mathbb{K}$ is the same thing as an ordinary orthogonal matroid. 
    
    \item A strong or weak orthogonal matroid over a field $K$ is the same thing as a projective solution to the Wick equations in $\bP^N(K)$, or an orthogonal matroid representable over $K$ in the sense of~\cite{Bouchet1988repre}. 

    \item A strong or weak orthogonal matroid over the regular partial field $\mathbb{U}_0$ is the same thing as a regular orthogonal matroid in the sense of~\cite{Geelen1996thesis}.
    
    \item A strong or weak orthogonal matroid over the tropical hyperfield $\T$ is the same thing as a valuated orthogonal matroid in the sense of~\cite{Dress1991,Wenzel1993,Wenzel1996}, or a tropical Wick vector in the sense of~\cite{Rincon2012}. 

    \item A strong orthogonal matroid over the sign hyperfield $\bS$ is the same thing as an oriented orthogonal matroid in the sense of~\cite{Wenzel1993,Wenzel1996}. 
\end{enumerate}


Together with several properties of tracts, we are able to prove representation theorems for orthogonal matroids. For instance, we give new proofs of the following characterizations of regular orthogonal matroids. 

\begin{theorem}[Geelen,~Theorem~4.13~of~\cite{Geelen1996thesis}]
  Let $M$ be an orthogonal matroid.
  Then the following are equivalent:
  \begin{enumerate}[label=\rm(\roman*)]
      \item $M$ is representable over $\F_2$ and $\F_3$.
      \item $M$ is representable over the regular partial field $\mathbb{U}_0$.
      \item $M$ is representable over all fields. 
  \end{enumerate}
\end{theorem}

We say that an orthogonal matroid is {\em regular} if it satisfies one of the three equivalent conditions in the above theorem. We also give two more characterizations of regular orthognal matroids without a specific minor $\ul{M_4}$ on $[4]\cup[4]^*$ (see Section~\ref{sec:applications} for a precise description of $\ul{M_4}$). 

\begin{theorem}
    Let $M$ be an orthogonal matroid with no minor isomorphic to $\ul{M_4}$ and
    let $(K,\prec)$ be an ordered field.
    Then the following are equivalent:
    \begin{enumerate}[label=\rm(\roman*)]
        \item $M$ is regular.
        \item $M$ is representable over $\F_{2}$ and $K$.
        \item $M$ is representable over $\F_{2}$ and the sign hyperfield $\bS$.
    \end{enumerate}
\end{theorem}

We then extend Whittle's theorem~\cite[Theorem~1.2]{Wh97} that a matroid is representable over both $\F_3$ and $\F_4$ if and only if it is representable over the sixth-root-of-unity partial field $R_6$ to orthogonal matroids. 

\begin{theorem}
    Let $M$ be an orthogonal matroid.
    Then the following are equivalent:
    \begin{enumerate}[label=\rm(\roman*)]
        \item $M$ is representable over the sixth-root-of-unity partial field $R_6$.
        \item $M$ is representable over $\F_3$ and $\F_4$.
        \item $M$ is representable over $\F_3$, $\F_{p^2}$ for all primes $p$, and $\F_q$ for all primes $q$ with $q \equiv 1 \pmod{3}$.   
    \end{enumerate}
\end{theorem}

\medskip

\textbf{Structure of the paper.}
In the remaining part of Section~\ref{section:intro}, we recall the classical theory of orthogonal matroids, mainly following~\cite{Borovik2003}, and describe matroids as orthogonal matroids. In Section~\ref{section:F-matroids} we survey some of the main results from the theory of matroids over tracts~\cite{Baker2019,Anderson2019}, including the definition of tracts. In Section~\ref{section:F-orthogonalmatroids} we define three notions of orthogonal matroids over tracts, namely the {\em weak}, {\em moderately weak}, and {\em strong orthogonal matroids over tracts}, using Wick functions, orthogonal $F$-signatures, $F$-circuit sets of orthogonal matroids, and orthogonal $F$-vector sets. These four axiom systems turn out to be cryptomorphic for strong orthogonal matroids over tracts, and the proofs are given in Section~\ref{sec:OMoverTracts}. Section~\ref{sec:OMoverTracts} also includes equivalences between the weaker notions, as well as several examples and counterexamples. In Section~\ref{sec:applications}, we discuss applications to representation theorems for orthogonal matroids.

\subsection{Orthogonal matroids}
\label{sec:course in orthogonal matroids}

Let $E = [n] \cup [n]^*$. The symmetric exchange axiom for orthogonal matroids on $E$ turns out to be equivalent to the {{\em strong symmetric exchange axiom}}~\cite[Theorem~4.2.4]{Borovik2003}. 

\begin{proposition}[Strong Symmetric Exchange]
If $M = (E,\cB)$ is an orthogonal matroid, then for every $B_1, B_2 \in \cB$ and divergence $\{x_1, x_1^\ast\} \subseteq B_1 \symdiff B_2$, there exists $\{x_2, x_2^\ast\} \subseteq \cB$ with $\{x_2, x_2^\ast\} \ne \{x_1, x_1^\ast\}$ such that both $B_1 \symdiff \{x_1, x_1^\ast, x_2, x_2^\ast\}$ and $B_2 \symdiff \{x_1, x_1^\ast, x_2, x_2^\ast\}$ belong to $\cB$. 
\end{proposition}

\begin{example}
Let $M$ be a matroid on $[n]$. Then the pair $\lift{M} := ([n] \cup [n]^\ast, \mathcal{B})$, where $\mathcal{B} := \{B \cup ([n] \setminus B)^\ast : B \text{ is a basis of } M\}$, is an orthogonal matroid. This is called the {\em lift} of the matroid $M$. Notice that an orthogonal matroid $N$ on $E = [n] \cup [n]^*$ is the lift of a matroid if and only if $B \cap [n]$ have the same cardinality for all bases $B$ of $N$. 
\end{example}

\begin{example}
Let $M = (E, \cB)$ be an orthogonal matroid and let $A \subseteq E$ be a subset such that $A=A^*$. Then $M \symdiff A := (E, \cB \symdiff A)$ is an orthogonal matroid, where $\cB \symdiff A := \{B \symdiff A : B \in \cB\}$. This is an example of a general operation on orthogonal matroids called {\em twisting}.
\end{example}

\begin{definition}
Two orthogonal matroids $M_1$ and $M_2$ are {\em isomorphic} if there exists a bijection $f : E(M_1) \to E(M_2)$ that respects the involutions on $E(M_1)$ and $E(M_2)$, and a transversal $T \subseteq E_1(M)$ is a basis of $M_1$ if and only if $f(T)$ is a basis of $M_2$. 
\end{definition}

\begin{definition}
Every subset of a basis is called an {\em independent} set. Admissible subsets of $E$ that are not independent are called {\em dependent}. A {\em circuit} is a minimal dependent set with respect to inclusion. 
\end{definition}

Let $\mathcal{C}(M)$ denote the family of circuits of an orthogonal matroid $M$. There is a characterization of orthogonal matroids in terms of circuits. 

\begin{proposition}[Theorem~4.2.5~of~\cite{Borovik2003}]
Let $\mathcal{C}$ be a set of admissible subsets of $E = [n] \cup [n]^\ast$. Then $\mathcal{C}$ is the family of circuits of an orthogonal matroid if and only if $\mathcal{C}$ satisfies the following five axioms:
  \begin{enumerate}[label=\rm(C\arabic*)]
    \item\label{item:C1} $\emptyset \not \in \mathcal{C}$.
    \item\label{item:C2} If $C_1, C_2 \in \cC$ with $C_1 \subseteq C_2$, then $C_1 = C_2$. 
    \item\label{item:C3} If $C_1 \ne C_2 \in \mathcal{C}$, $x \in C_1\cap C_2$, and $C_1 \cup C_2$ is admissible, then there exists $C_3 \in \mathcal{C}$ such that $C_3 \subseteq (C_1 \cup C_2) \setminus \{x\}$.
    \item\label{item:C4} If $C_1,C_2 \in \mathcal{C}$ and $C_1 \cup C_2$ is not admissible, then $C_1 \cup C_2$ contains at least two divergences.
    \item\label{item:C5} If $T$ is a transversal and $x \not\in T$, then $T \cup \{x\}$ contains an element in $\mathcal{C}$. 
  \end{enumerate}
\end{proposition}


Recall that for a matroid $M$, there is a unique circuit contained in $B \cup \{e\}$ for every basis $B$ and an element $e \not\in B$, called the {\em fundamental circuit} with respect to $B$ and $e$. The next proposition gives an analogous notion of fundamental circuits for orthogonal matroids. 

\begin{proposition}[Theorem~4.2.1~of~\cite{Borovik2003}]\label{prop:fundamental-circuit}
Let $M = (E, \cB)$ be an orthogonal matroid. Take $B \in \cB$ and $x \not\in B$. Then there exists a unique circuit $C_M(B,x)$ of $M$ such that $C_M(B,x) \subseteq B \cup \{x\}$. Furthermore, $C_M(B,x)$ is given by $C_M(B,x) = \{x\} \cup \{b \in B \setminus \{x^*\} : B \symdiff \{b, b^\ast, x, x^\ast\} \in \cB\}$. 
\end{proposition}

We call $C_M(B,x)$ the {\em fundamental circuit} with respect to $B$ and $x$, and we often write it as $C(B,x)$ if $M$ is clear from the context.

We will use the following lemma frequently in Section~\ref{sec:OMoverTracts}.

\begin{lemma}\label{lem:extendcircuit}
Let $C$ be a circuit of an orthogonal matroid $M$. Then there exists a transversal $T$ containing $C$ such that for every $x \in C$, $T \symdiff \skewpair{x}$ is a basis of $M$.
\end{lemma}

\begin{proof}
  We choose an arbitrary $y \in C$ and take a basis $B$ containing $C \setminus \{y\}$. Then $T = B \symdiff \skewpair{y}$ satisfies the desired property.
\end{proof}

Orthogonal matroids admit duals. Let $M = (E, \cB)$ be an orthogonal matroid, then the collection of bases $\cB^\ast$ of the {\em dual orthogonal matroid} $M^\ast$ is defined as 
\[
\cB^\ast := \{B^\ast : B \in \cB\}. 
\]
Circuits of $M^\ast$ are called {\em cocircuits} of $M$, and must be of the form $C^\ast$ for some circuit $C$ of~$M$. 

We finally discuss minors of orthogonal matroids in the sense of~\cite{BouchetMM2}. 

Let $M = (E, \cB)$ be an orthogonal matroid. An element $x \in E$ is {\em singular} if $M$ has no basis containing $x$, or equivalently, $\{x\}$ is a circuit of $M$. Otherwise, we call the element $x$ {\em nonsingular}. By~\ref{item:C4}, if an element $x$ is singular, then $x^\ast$ is nonsingular.

Let $M$ be an orthogonal matroid on $E$ and let $x \in E$. If $x$ is nonsingular, then
\begin{align*}
  \{B \setminus \{x\} : x \in B \in \mathcal{B}(M)\}
\end{align*}
is the set of bases of an orthogonal matroid on $E \setminus \{x, x^\ast\}$. We denote this orthogonal matroid by $M|x$. If $x$ is singular, then we define $M|x := M|x^\ast$. We call $M|x$ an {\em elementary minor} of $M$. In particular, if $x \in [n]$ (resp. $x \in [n]^\ast$) then it corresponds to the contraction (resp. deletion) by $x$ in the sense of~\cite{Chun2019ribbon}. 

An orthogonal matroid $N$ is a {\em minor} of another orthogonal matroid $M$ if $N$ can be obtained from $M$ by taking elementary minors sequentially. Note that $M|x|y = M|y|x$, and thus we write $M|x_1|x_2|\dots|x_k$ as $M|S$ where $S = \{x_1,\dots,x_k\}$.

For a collection $\cC$ of subsets of $E$, let $\Min(\cC)$ denote the set of minimal elements of $\cC$ with respect to inclusion. The following proposition characterizes circuits of minors of orthogonal matroids. 

\begin{proposition}\label{prop:circuits of a minor}
For an orthogonal matroid $M$ and an element $x \in E$, we have 
\begin{align*}
  \cC(M|x) = \Min \left( \{C \setminus \{x\} : x^\ast \not\in C \in \cC(M)
  \text{ and } C\neq \{x\}\}.  
  \right)
\end{align*}
\end{proposition}

\begin{proof}
    By the definition of $M|x$, if $x$ is nonsingular, then $\{x\}$ is not a circuit of $M$ and
  \begin{align*}
    \cC(M|x)
    & = \Min \{ C \in \cA : C \not\subseteq B \text{ for all bases $B$ of $M$ with $x \in B$} \} \\
    & = \Min \{ C \in \cA : C\cup\{x\} \text{ is dependent in $M$} \} \\
    & = \Min \{ C \in \cA : C \text{ or } C\cup\{x\} \text{ is a circuit of $M$} \} \\
    & = \Min \{ C \setminus\{x\} : x^\ast \not\in C \in \cC(M)\},
  \end{align*}
where $\cA$ is the set of all subtransversals in $E \setminus \skewpair{x}$.
Now we assume that $x$ is singular. Then $\{x\}$ is the only circuit of $M$ containing $x$, and $M$ has no circuit containing $x^*$. Since $M|x = M|x^*$, by the previous result, we have 
  \begin{align*}
      \cC(M|x) = \cC(M|x^*)
      &=
      \Min \{ C \setminus\{x^*\} : x \not\in C \in \cC(M) \}\\
      &=
      \Min \{ C \setminus\{x\} : x^* \not\in C \in \cC(M) \text{ and } C \neq \{x\} \}.
      \qedhere
  \end{align*}
\end{proof}

\subsection{Matroids and orthogonal matroids}\label{sec:matroid and orthogonal matroid}

Recall that if $M$ is a matroid on $[n]$, then $\lift{M}$ is an orthogonal matroid whose set of bases $\cB(\lift{M})$ is given by $\{B \cup ([n]\setminus B)^* : B \in \cB(M)\}$. There is a similar result for circuits. 

\begin{proposition}[Bouchet,~Proposition~4.1~of~\cite{BouchetMM2}]\label{eq:circuits of lift}
Let $M$ be a matroid. Then the set of circuits of the orthogonal matroid $\lift{M}$ is 
\begin{align*}
  \cC(\lift{M}) = \{C : \text{$C$ is a circuit of $M$}\} \cup \{D^* : \text{$D$ is a cocircuit of $M$}\}.
\end{align*}
\end{proposition}

Furthermore, the lift of the dual matroid $M^*$ can be obtained by taking the involution $*$ for all bases and circuits of the lift of the original matroid $M$. In other words, $\mathcal{B}(\lift{M^*}) = (\mathcal{B}(\lift{M}))^*$ and $\mathcal{C}(\lift{M^*}) = (\mathcal{C}(\lift{M}))^*$.
An element $x \in E$ is singular in $\lift{M}$ if and only if either $x \in [n]$ and $x$ is a loop in $M$, or $x \in [n]^*$ and $x^*$ is a coloop in $M$. Finally, minors of the lift of a matroid $M$ can be expressed as lifts of minors of $M$.

\begin{proposition}[Bouchet,~Corollary~5.3~of~\cite{BouchetMM2}]
Let $M$ be a matroid on $[n]$ and let $x \in [n]$. Then $\lift{M}|x = \lift{M / x}$ and $\lift{M}|x^* = \lift{M \setminus x}$. As a consequence, we have $\mathcal{C}(\lift{M} | x)
    = \mathcal{C}(M / x) \cup ( \mathcal{C}^*(M / x) )^*
    = \mathcal{C}(M / x) \cup ( \mathcal{C}(M^* \setminus x) )^*$,
where $\mathcal{C}^*(M/x)$ denotes the set of cocircuits of $M/x$.
\end{proposition}


\section{Matroids over Tracts}\label{section:F-matroids}

Let $0 \leq r \leq n$ be nonnegative integers and consider the finite set $E = [n] = \{1, \dots, n\}$. Denote by $\binom{E}{r}$ the family of all $r$-element subsets of $E$. In this section, we review the study of matroids over tracts in~\cite{Anderson2019} and~\cite{Baker2019}. 

\subsection{Tracts}\label{sec:tracts}

A {\em tract} $F = (G, N_F)$ is an abelian group $G$ (written multiplicatively), together with an additive relation structure $N_F$, which is a subset of the group semiring $\mathbb{N}[G]$ satisfying:

\begin{enumerate}[label=\rm(T\arabic*)]
    \item The zero element $0$ of $\mathbb{N}[G]$ belongs to $N_F$. 
    \item The identity element $1$ of $G$ does not belong to $N_F$.
    \item There is a unique element $\epsilon$ of $G$ such that $1 + \epsilon \in N_F$.
    \item If $g \in G$ and $a \in N_F$, then $ga \in N_F$. 
\end{enumerate}

We think of $N_F$ as linear combinations of elements of $G$ which `sum to zero', and call it  the {\em null set} of the tract $F$. 

A useful lemma from~\cite{Baker2019} about tracts is as follows. 

\begin{lemma}[Lemma~1.1~of~\cite{Baker2019}]
Let $F = (G,N_F)$ be a tract. Then we have the following:
\begin{enumerate}[label=\rm(\roman*)]
\item If $x,y \in G$ with $x+y \in N_F$, then $y = \epsilon x$.

\item $\epsilon^2 = 1$.

\item $G \cap N_F = \emptyset$.
\end{enumerate}
\end{lemma} 

Because of this lemma, we also write $F$ for the set $G \cup \{0\}$, and write $-1$ instead of $\epsilon$. We will sometimes use $G$ and $F^\times$ interchangeably. 

A {\em tract homomorphism} $\varphi: F_1 \to F_2$ is a group homomorphism $\varphi: F_1^\times \to F_2^\times$ such that the induced semiring homomorphism $\mathbb{N}[F_1^\times] \to \mathbb{N}[F_2^\times]$ maps $N_{F_1}$ to $N_{F_2}$. All tracts together with tract homomorphisms between them form a category. An {\em involution} $\tau$ of a tract $F$ is a tract homomorphism $\tau: F \to F$ such that $\tau^2$ is the identity map.

\begin{example}
The {\em initial tract} is $\mathbb{I} = (\{1, -1\}, \{0, 1 + (-1)\})$, where the multiplication on $\mathbb{I}^\times$ is the usual one. 
\end{example}

\begin{example}
Let $K$ be a field and let $G$ be a subgroup of $K^\times$. The multiplicative monoid $F = K/G = (K^\times/G) \cup \{0\}$ can be endowed with a natural tract structure by setting $N_F := \{\sum_{i = 1}^k x_i \in \mathbb{N}[K^\times/G] : 0 \in \sum_{i = 1}^k x_i\}$. We call tracts of this form {\em quotient hyperfields}. 
Especially, whenever $G = \{1\}$, one can view a field as a tract.
\end{example}

\begin{example}
The {\em Krasner hyperfield} is $\mathbb{K} = K/K^\times = (\{1\}, \N[1] \setminus \{1\})$ for an arbitrary field $K$ with more than two elements. This is the terminal object in the category of tracts.
The {\em hyperfield of signs} is $\mathbb{S} = \mathbb{R} / \mathbb{R}_{>0} = (\{\pm 1\}, N_{\mathbb{S}})$, where an element $\sum x_i \in \N[\{\pm 1\}]$ is in $N_{\mathbb{S}}$ if and only if there is at least one $x_i = 1$ and at least one $x_j = -1$, or all $x_i$ are zero.  
\end{example}

\begin{example}
The {\em tropical hyperfield} $\mathbb{T} = (\R \cup \{+ \infty\}, N_{\mathbb{T}})$, where $+\infty$ serves as the zero element in the tract. The multiplication on $\mathbb{T}^\times = \R$ is the usual addition, and the `addition' on $\mathbb{T}$ is defined as $\sum x_i \in N_{\mathbb{T}}$ if and only if the minimum of $x_i$'s is achieved at least twice. 
\end{example}

\begin{example}
A {\em{partial field}} $P$ is a pair $(G, R)$ of a commutative ring $R$ with $1$ and a subgroup $G$ of the group of units of $R$ such that $-1$ belongs to $G$ and $G$ generates the ring $R$. We can associate a tract structure on any partial field $P$ by setting the null set to be the set of all formal sums $\sum_{i = 1}^k x_i \in \mathbb{N}[G]$ such that $\sum_{i = 1}^k x_i = 0 \in R$. Notice that a partial field with $G = R \setminus \{0\}$ is the same thing as a field. 
\end{example}

\begin{example}
The {\em regular partial field} is $\U_0 = (\{1, -1\}, \Z)$.
The {\em sixth-root-of-unity partial field} is $R_6 := (\langle \zeta \rangle, \Z[\zeta])$ where $\zeta \in \C^\times$ is a root of $x^2-x+1 = 0$.
\end{example}

We list some examples of tracts and tract homomorphisms in Figure~\ref{fig:tracts}.

\begin{figure}
  \begin{tikzpicture}
    \node [align=center] at (0.3,0.5) (I) {$\mathbb{I}$};
    \node [align=center] at (2,0.5) (U0) {$\mathbb{U}_0$};

    \node [align=center] at (4,2) (F2) {$\mathbb{F}_2$};
    \node [align=center] at (6,2) (F4) {$\mathbb{F}_4$};

    \node [align=center] at (4,1) (R6) {$R_6$};
    \node [align=center] at (6,1) (F3) {$\mathbb{F}_3$};


    \node [align=center] at (3.3,0) (Q) {$\mathbb{Q}$};
    \node [align=center] at (4.7,0) (R) {$\mathbb{R}$};
    \node [align=center] at (6,0) (S) {$\mathbb{S}$};

    \node [align=center] at (6,-1) (T) {$\mathbb{T}$};

    \node [align=center] at (8,0.5) (K) {$\mathbb{K}$};

    \draw [->] (I) -- (U0);
    \draw [->] (U0) -- (F2);
    \draw [->] (U0) -- (R6);
    \draw [->] (U0) -- (Q);

    \draw [->] (F2) -- (F4);

    \draw [->] (R6) -- (F4);
    \draw [->] (R6) -- (F3);


    \draw [->] (Q) -- (R);
    \draw [->] (R) -- (S);
    \draw [->] (Q) -- (T);

    \draw [->] (F4) -- (K);
    \draw [->] (F3) -- (K);
    \draw [->] (S) -- (K);
    \draw [->] (T) -- (K);
  \end{tikzpicture}
  \caption{Examples of tracts and tract homomorphisms.}
  \label{fig:tracts}
\end{figure}

The category of tracts admits products, which will be useful for studying representations of matroids and orthogonal matroids in Section~\ref{sec:applications}. Let $F_1, F_2$ be tracts. The (categorical) {\em product} $F_1 \times F_2$ can be constructed explicitly as follows. As a set, $F_1 \times F_2$ is $(F_1^\times \oplus F_2^\times) \cup \{0\}$, endowed with the coordinate-wise multiplication on $F_1^\times \oplus F_2^\times$, and the rule $0 \cdot (x_1, x_2) = (x_1, x_2) \cdot 0 = 0$. The null set of $F_1 \times F_2$ is 
\[
N_{F_1 \times F_2} := \left\{ \sum_{i=1}^k (x_i,y_i) \in \N[F_1^\times \oplus F_2^\times] : \sum_{i = 1}^k x_i \in N_{F_1} \text{ and } \sum_{i = 1}^k y_k \in N_{F_2}\right\}. 
\]

\subsection{Grassmann-Pl\"{u}cker functions}
The easiest way of defining matroids over tracts is via the Grassmann-Pl\"{u}cker functions. This is also the tract analogue of the basis exchange axiom for ordinary matroids. 

\begin{definition}
Let $F$ be a tract. A {\em strong Grassmann-Pl\"{u}cker function of rank $r$ on $E$ with coefficients in $F$} is a function $\varphi: E^r \to F$ satisfying \ref{GP1}--\ref{GP3}:
\begin{enumerate}[label=\rm(GP\arabic*)]
    \item\label{GP1} $\varphi$ is not identically zero.  
    
    \item\label{GP2} $\varphi$ is alternating, i.e.,
    for all $x_1, \dots x_r \in E$, $\varphi(x_1,\dots,x_r) = 0$ if $x_i = x_j$ for some $i\neq j$, and $\varphi(x_1, \dots, x_i, \dots, x_j, \dots, x_r) = -\varphi(x_1, \dots, x_j, \dots, x_i, \dots, x_r)$.
    
    \item\label{GP3} For any two subsets $\{x_1, \dots, x_{r+1}\}$ and $\{y_1, \dots, y_{r-1}\}$ of $E$, we have the {\em Grassmann-Pl\"{u}cker relations}: 
    \[
    \sum_{k = 1}^{r+1}(-1)^k \varphi(x_1, \dots, \hat{x}_k, \dots, x_{r+1}) \varphi(x_k, y_1, \dots, y_{r-1}) \in N_F. 
    \]
\end{enumerate}
\end{definition}

\begin{definition}
Let $F$ be a tract. A {\em weak Grassmann-Pl\"{u}cker function of rank $r$ on $E$ with coefficients in $F$} is a function $\varphi: E^r \to F$ such that the support $\{\{x_1,\dots,x_r\} \in \binom{E}{r} : \varphi(x_1,\dots,x_r) \neq 0\}$ of $\varphi$ is the set of bases of a rank $r$ matroid on $E$, and $\varphi$ satisfies~\ref{GP1}, \ref{GP2}, and the next weaker replacement of~\ref{GP3}.
\begin{enumerate}[label=\rm(GP3)$'$]
    \item\label{GP3'} For any two subsets $J_1 = \{x_1, \dots, x_{r+1}\}$ and $J_2 = \{y_1, \dots, y_{r-1}\}$ of $E$ with $|J_1|=r+1$, $|J_2|=r-1$, and $|J_1 \setminus J_2| = 3$, we have the {\em $3$-term Grassmann-Pl\"{u}cker relations}: 
    \[
    \sum_{k = 1}^{r+1}(-1)^k \varphi(x_1, \dots, {\hat{x}}_k, \dots, x_{r+1}) \varphi(x_k, y_1, \dots, y_{r-1}) \in N_F. 
    \]
\end{enumerate}
\end{definition}

Two strong (resp. weak) Grassmann-Pl\"{u}cker functions $\varphi_1$ and $\varphi_2$ are {\em equivalent} if $\varphi_1 = c \cdot \varphi_2$ for some $c \in F^\times$, and we call an equivalence class $M_\varphi := [\varphi]$ of strong (resp. weak) Grassmann-Pl\"{u}cker functions a {\em strong \textup{(resp.} weak\textup{)} matroid over the tract $F$}, or simply a {\em strong \textup{(resp.} weak\textup{)} $F$-matroid}. It can be shown that every strong $F$-matroid is a weak $F$-matroid. We denote by $\underline{M}_\varphi$ the underling ordinary matroid of a strong or weak $F$-matroid $M_\varphi$ whose set of bases is $\mathcal{B}(\underline{M}_\varphi) = \{\{x_1,\dots,x_r\} \in \binom{E}{r}: \varphi(x_1,\dots,x_r) \ne 0\}$. 

\subsection{$F$-circuits and dual pairs} 
We now give two cryptomorphic definitions of matroids over a tract $F$ in terms of $F$-circuits and dual pairs of $F$-signatures. 

Denote by $F^E$ the set of all functions from $E$ to $F$. The {\em support} of $X \in F^E$ is the set of elements $e$ in $E$ such that $X(e) \ne 0$, and is denoted by $\underline{X}$ or $\supp{X}$. Given two functions $X = (x_1, \dots, x_n)$ and $Y = (y_1, \dots, y_n) \in F^E$, where $F$ is endowed with an involution $x \mapsto \overline{x}$, the {\em inner product} of $X$ and $Y$ is $X \cdot Y := \sum_{k = 1}^n x_k \overline{y_k}$. We say that two functions $X$ and $Y$ are {\em orthogonal}, denoted by $X \perp Y$, if $X \cdot Y \in N_F$.

When $F$ is the field $\mathbb{C}$ of complex numbers or the sixth-root-of-unity partial field $R_6$, the involution $x \mapsto \overline{x}$ should be taken to be the complex conjugation. For $F \in \{\mathbb{K}, \mathbb{S}, \mathbb{T}\}$, the involution should be taken to be the identity map. 

The {\em linear span} of $X_1, \dots, X_k \in F^E$ is defined to be the set of all functions $X \in F^E$ such that 
\[
c_1X_1 + c_2X_2 + \dots + c_kX_k - X \in (N_F)^E
\]
for some $c_1,\dots,c_k \in F$.

\begin{definition}\label{def:matroid signature}
Let $\ul{M}$ be an ordinary matroid on $E$. An {\em $F$-signature} of $\ul{M}$ is a subset $\mathcal{C} \subseteq F^E$ such that the following hold:
 \begin{enumerate}[label=\rm(\roman*)]
        \item The support $\underline{\mathcal{C}} := \{\underline{X} : X \in \mathcal{C} \}$ of $\mathcal{C}$ is the set of circuits of $\ul{M}$. 
        \item For all $X \in \mathcal{C}$ and $\alpha \in F^\times$, we have $\alpha X \in \mathcal{C}$. 
        \item {If $X,Y \in \mathcal{C}$ and $\underline{X} \subseteq \underline{Y}$, then $X = \alpha Y$ for some $\alpha \in F^\times$.}
    \end{enumerate}
\end{definition}

For an ordinary matroid $\ul{M}$ on $E$, we denote by $C_{\ul{M}}(B,e)$ the fundamental circuit of $M$ with respect to $B\in \cB(\ul{M})$ and $e\in E\setminus B$. 
The subscript will be omitted if no confusion arises.

\begin{definition}\label{def:F-circuit sets of matroids}
    Let $F$ be a tract and let $\ul{M}$ be an ordinary matroid on $E$.
    A subset $\mathcal{C}$ of $F^{E}$ is called a {\em strong $F$-circuit set of $\ul{M}$} if it satisfies the following axioms:
    \begin{enumerate}[label=\rm(CS\arabic*)]
        \item\label{CS1}
        $\mathcal{C}$ is an $F$-signature of $\ul{M}$.
        \item\label{CS2}\label{item:CS2}
        For every basis $B$ of $\ul{M}$ and for each 
        $X \in \cC$, $X$ is in the linear span of $\{X_{e}\}_{e\in E\setminus B}$, where $X_{e} \in \cC$ has support $C(B,e)$.
    \end{enumerate}
    We call $\mathcal{C}$ a {\em weak $F$-circuit set of $\ul{M}$} if it satisfies \ref{CS1} and the following replacement: 
    \begin{enumerate}[label=\rm(CS\arabic*)$'$]
    \setcounter{enumi}{1}
        \item\label{item:CS2'}
        For every basis $B$ of $\ul{M}$ and distinct elements $e_1,e_2 \in E\setminus B$, if 
        $X_1$ and $X_2$ in $\cC$ have supports $C(B,e_1)$ and $C(B,e_2)$, respectively, and $f$ is a common element of the two supports, then there exists 
        $Y\in \cC$ such that $Y(f) = 0$ and $Y$ is in the linear span of $X_1$ and $X_2$.
    \end{enumerate}
\end{definition}

\begin{definition}\label{def:dual pairs of matroids}
    Let $F$ be a tract and let $\ul{M}$ be an ordinary matroid on $E$.
    A pair $(\mathcal{C},\mathcal{D})$ of subsets of $F^E$ is called a {\em strong dual pair of $F$-signatures of $\ul{M}$} if
    \begin{enumerate}[label=\rm(DP\arabic*)]
        \item\label{DP1} $\mathcal{C}$ is an $F$-signature of $\ul{M}$. 
        \item\label{DP2} $\mathcal{D}$ is an $F$-signature of the dual matroid $\ul{M}^*$. 
        \item\label{DP3} $X\perp Y$ for all $X \in \mathcal{C}$ and $Y \in \mathcal{D}$.
    \end{enumerate}
A pair $(\mathcal{C},\mathcal{D})$ is called a {\em weak dual pair of $F$-signatures of $\ul{M}$} if it satisfies \ref{DP1}, \ref{DP2}, and the following weakening of~\ref{DP1}:
    \begin{enumerate}[label=\rm(DP\arabic*)$'$]
        \setcounter{enumi}{2}
        \item\label{Dp3'} $X\perp Y$ for all $X \in \mathcal{C}$ and $Y \in \mathcal{D}$ with $|\underline{X} \cap \underline{Y}| \leq 3$.
    \end{enumerate}
\end{definition}

\subsection{$F$-vectors}
One can naturally ask for an axiomatization of linear spaces over a tract $F$. Anderson answered this question and gave another cryptomorphic definition of strong $F$-matroids in terms of $F$-vectors in~\cite{Anderson2019}.


For a subset $\mathcal{W} \subseteq F^E$, a {\em support basis} for $\mathcal{W}$ is a minimal subset of $E$ meeting every element of $\supp(\mathcal{W} \setminus \{0\})$. A {\em reduced row-echelon form} of $\mathcal{W}$ with respect to a support basis $B$ is a subset $\Phi_B = \{w_i^B\}_{i \in B} \subseteq \mathcal{W}$ such that $w_i^B(j) = \delta_{ij}$ for each $i,j\in B$, and every $w \in \cW$ is in the linear span of $\Phi_B$. It is not difficult to see that if $\Phi_B$ exists, then it is unique. We say a collection $\Phi = \{\Phi_B\}$ of reduced row-echelon forms is {\em tight} if $\mathcal{W}$ is precisely the set of elements of $F^E$ which are in the linear span of $\Phi_B$ for all $\Phi_B \in \Phi$. 

\begin{definition}
A subset $\mathcal{W}$ of $F^E$ is an {\em $F$-vector set on $E$} if the following hold:
\begin{enumerate}[label=\rm(V\arabic*)]
\item Every support basis has a reduced row-echelon form. 
\item The collection of all such reduced row-echelon forms is tight. 
    \end{enumerate}
\end{definition}

\subsection{Crypotomorphisms} The main results of~\cite{Baker2019,Anderson2019} are the following theorems. 

\begin{theorem}[Theorem~4.17~of~\cite{Baker2019} and Theorem~2.18~of~\cite{Anderson2019}]\label{thm:strong matroid bijection}
  Let $E$ be a finite set and let $F$ be a tract endowed with an involution $x \mapsto \overline{x}$. Then there are natural bijections between: 
  \begin{enumerate}
      \item Strong $F$-matroids on $E$.
      \item Strong $F$-circuit sets of matroids on $E$. 
      \item Ordinary matroids on $E$ endowed with a strong dual pair of $F$-signatures. 
      \item $F$-vector sets on $E$.
  \end{enumerate}
\end{theorem}


\begin{theorem}[Theorem~4.18~of~\cite{Baker2019}]\label{thm:weak matroid bijection}
  Let $E$ be a finite set and let $F$ be a tract endowed with an involution $x \mapsto \overline{x}$. Then there are natural bijections between: 
  \begin{enumerate}
      \item Weak $F$-matroids on $E$.
      \item Weak $F$-circuit sets of matroids on $E$. 
      \item Ordinary matroids on $E$ endowed with a weak dual pair of $F$-signatures. 
  \end{enumerate}
\end{theorem}

\subsection{Functoriality, duality, and minors. }

Let $F$ be a tract with an involution $x \mapsto \overline{x}$. The theory of functoriality, duality, and minors for matroids over tracts generalizes the corresponding classical theory for matroids.  
For simplicity, here we only give the descriptions via the strong Grassmann-Pl\"{u}cker functions. 

Given a strong Grassmann-Pl\"{u}cker function $\varphi: E^r \to F$ and a tract homomorphism $f: F \to F'$, we define the {\em pushforword} $f_\ast\varphi: E^r \to F'$ as  
\[
(f_\ast\varphi)(x_1, \dots, x_r) = f(\varphi(x_1, \dots, x_r)). 
\]
It is not hard to see that $f_\ast\varphi$ is a strong Grassmann-Pl\"{u}cker function. Notice that pushforwards are functorial: if $F_1 \xrightarrow{f} F_2 \xrightarrow{g} F_3$ are tract homomorphisms, then $(g \circ f)_\ast = g_\ast \circ f_\ast$. 

\medskip

The {\em dual Grassmann-Pl\"{u}cker function} $\varphi^\ast : E^{n-r} \to F$ of $\varphi$ is determined by~\ref{GP2} and
\[
\varphi^\ast(x_1, \dots, x_{n-r}) = \sgn(x_1, \dots, x_{n-r}, x_1', \dots, x_r') \cdot \overline{\varphi(x_1', \dots, x_r')}, 
\] 
where $x_1', \dots, x_r'$ is any ordering of $E \setminus \{x_1, \dots, x_{n-r}\}$, and $\sgn(x_1, \dots, x_{n-r}, x_1', \dots, x_r') \in \{\pm 1\}$ is the permutation sign taken as an element of $F$.
This notion of dual Grassmann-Pl\"{u}cker functions satisfies $\varphi^{\ast\ast} = \varphi$, and the underlying matroid of $\varphi^\ast$ is the dual matroid of the underlying matroid of $\varphi$. 

\medskip

Let $\varphi: E^r \to F$ be a strong Grassmann-Pl\"{u}cker function with the underlying matroid $\ul{M}_{\varphi}$ and let $A \subseteq E$.
We denote by $\ell$ and $k$ the ranks of $A$ and $E\setminus A$ in $\ul{M}_\varphi$, respectively.

Let $\{a_1, \dots, a_\ell\}$ be a maximal independent subset of $A$ in $\ul{M}_{\varphi}$. The {\em contraction} $\varphi / A: (E \setminus A)^{r - \ell} \to F$ is defined by 
\[
(\varphi / A)(x_1, \dots, x_{r-\ell}) = \varphi(x_1, \dots, x_{r - \ell}, a_1, \dots, a_r). 
\]

Choose $\{b_1, \dots, b_{r-k}\}$ such that $\{b_1, \dots, b_{r-k}\}$ is a basis of $\ul{M}_{\varphi} / (E\setminus A)$. Then the {\em deletion} $\varphi \setminus A: (E \setminus A)^k \to F$ is defined by
\[
(\varphi \setminus A)(x_1, \dots, x_k) = \varphi(x_1, \dots, x_k, b_1, \dots, b_{r-k}). 
\]

The following lemma shows that the contractions and deletions are well-defined. 

\begin{lemma}[Lemma~4.4~of~\cite{Baker2019}]
  The following hold.
\begin{enumerate}
\item Both $\varphi / A$ and $\varphi \setminus A$ are strong Grassmann-Pl\"{u}cker functions, and they are independent of all choices up to a global multiplication by a nonzero element of $F$. 

\item $\ul{M}_{\varphi / A} = \ul{M}_\varphi / A$ and $\ul{M}_{\varphi \setminus A} = \ul{M}_\varphi \setminus A$. 

\item $(\varphi \setminus A)^\ast = \varphi^\ast / A$. 
\end{enumerate}
\end{lemma}


\section{Orthogonal Matroids over Tracts}\label{section:F-orthogonalmatroids}

Let $E = [n] \cup [n]^\ast$ be a finite set and let $F$ be a tract endowed with an involution $x \mapsto \overline{x}$. In Section~\ref{sec:Wick functions}, we define strong, moderately weak, and weak orthogonal matroids on $E$ over $F$ in terms of Wick functions. We then establish other cryptomorphic definitions, including orthogonal $F$-signatures and $F$-circuit sets of orthogonal matroids in Section~\ref{sec:OrthogonalSignature}, and orthogonal $F$-vector sets in Section~\ref{sec:VectorSet}.
We then summarize equivalences and implications of various notions in Section~\ref{sec:main results}.
In Sections~\ref{sec:OM functionality}--\ref{sec:OM minors}, we introduce functoriality, duality, and minors, and in Section~\ref{sec:equiv of diff defs}, we explain how orthogonal $F$-matroids generalize historical works on orthogonal matroids by specifying $F$.



\subsection{Wick functions}\label{sec:Wick functions} We describe the first cryptomorphic characterization of strong, moderately weak, and weak orthogonal matroids over tracts in terms of Wick functions. We denote by $\mathcal{T}_n$ the family of all transversals of $E$.

\begin{definition}
A {\em strong Wick function on $E$ with coefficients in $F$} is $\varphi: \mathcal{T}_n \to F$ such that:
\begin{enumerate}[label=\rm(W\arabic*)]
    \item\label{item:W1} $\varphi$ is not identically zero. 
    
    \item\label{item:W2} For all $T_1$, $T_2 \in \mathcal{T}_n$, we have
    \[
      \sum_{k = 1}^m (-1)^k \varphi(T_1 \symdiff \skewpair{x_k}) \varphi(T_2 \symdiff \skewpair{x_k}) \in N_F, 
    \]
    where $(T_1 \symdiff T_2) \cap [n] = \{x_1 < \cdots < x_m\}$.
\end{enumerate}
\end{definition}

\begin{proposition}\label{prop:strong-is-weak}
  The support $\supp(\varphi) := \{T \in \mathcal{T}_n : \varphi(T) \ne 0\}$ of a strong Wick function $\varphi$ is the set of bases of an ordinary orthogonal matroid.
\end{proposition}

\begin{proof}
  Clearly $\supp(\varphi) \ne \emptyset$ by~\ref{item:W1}. Let $B_1, B_2$ be in $\supp(\varphi)$ with $\{x, x^\ast\} \subseteq B_1 \symdiff B_2$. Let $T_1 = B_1 \symdiff \{x, x^\ast\}$ and $T_2 = B_2 \symdiff \{x, x^\ast\}$, and we write $(B_1 \symdiff B_2) \cap [n] = (T_1 \symdiff T_2) \cap [n] = \{x_1 < \cdots < x_m\}$. Take $i \in [m]$ such that $\skewpair{x_i} = \skewpair{x}$. Then we have $\varphi(T_1 \symdiff \skewpair{x_i}) \varphi(T_2 \symdiff \skewpair{x_i}) = \varphi(B_1) \varphi(B_2) \ne 0$. By~\ref{item:W2}, there exists $y \in \{x_1,\ldots,x_m\} \setminus \{x_i\}$ such that  $\varphi(T_1 \symdiff \skewpair{y}) \varphi(T_2 \symdiff \skewpair{y}) \neq 0$, implying that $B_j \symdiff \skewpair{x} \symdiff \skewpair{y} = T_j \symdiff \skewpair{y} \in \supp(\varphi)$ for both $j\in \{1, 2\}$. 
\end{proof}

\begin{definition}
Let $\varphi: \mathcal{T}_n \to F$ be a map such that the support of $\varphi$ is the set of bases of an orthogonal matroid.
We say that $\varphi$ is
a {\em moderately weak Wick function on $E$ with coefficients in $F$} if $\varphi$ satisfies~\ref{item:W1} and the following weakened version of~\ref{item:W2}:
\begin{enumerate}[label=\rm(W2)$'$]
    \item\label{item:W2'} For all $T_1$, $T_2 \in \mathcal{T}_n$, if $(T_1 \symdiff T_2) \cap [n] = \{x_1 < \cdots < x_m\}$, and at most four of $\varphi(T_1 \symdiff \skewpair{x_k}) \varphi(T_2 \symdiff \skewpair{x_k})$'s are nonzero, then we have
    \[
      \sum_{k = 1}^m (-1)^k \varphi(T_1 \symdiff \skewpair{x_k}) \varphi(T_2 \symdiff \skewpair{x_k}) \in N_F. 
    \]
\end{enumerate}
We say that $\varphi$ is
a {\em weak Wick function on $E$ with coefficients in $F$} if $\varphi$ satisfies~\ref{item:W1} and: 
\begin{enumerate}[label=\rm(W2)$''$]
    \item\label{item:W2''} For all $T_1$, $T_2 \in \mathcal{T}_n$, if $(T_1 \symdiff T_2) \cap [n] = \{x_1 < x_2 < x_3 < x_4\}$, then we have
    \[
      \sum_{k = 1}^4 (-1)^k \varphi(T_1 \symdiff \skewpair{x_k}) \varphi(T_2 \symdiff \skewpair{x_k}) \in N_F. 
    \]
\end{enumerate}
\end{definition}


Two strong 
Wick functions $\varphi$ and $\psi$ with coefficients in $F$ are {\em equivalent} if $\varphi = c \psi$ for some nonzero $c \in F$, and we call an equivalence class $M_\varphi = [\varphi]$ of strong 
Wick functions a {\em strong 
orthogonal matroid over the tract $F$}, or simply a {\em strong 
orthogonal $F$-matroid}.
We similarly define {\em \textup{(}moderately\textup{)} weak orthogonal $F$-matroid}.
It is trivial that every moderately weak orthogonal $F$-matroid is weak.
Proposition~\ref{prop:strong-is-weak} shows that every strong orthogonal $F$-matroid is a moderately weak orthogonal $F$-matroid.
Three notions of orthogonal $F$-matroids are the same when $F$ is a partial field~\cite{BJ2022}, the tropical hyperfield~$\T$~\cite{Rincon2012}, or the Krasner hyperfield $\K$.
We denote by $\underline{M}_\varphi$ the underlying orthogonal matroid of the orthogonal $F$-matroid $M_\varphi$ whose set of bases is $\supp(\varphi)$. 




\begin{proposition}\label{prop:W and GP}
There is a natural bijection between the set of all strong $F$-matroids on $[n]$ and the set of all strong orthogonal $F$-matroids $M_{\psi}$ on $[n] \cup [n]^\ast$ 
such that the intersections of bases of $\ul{M}_\psi$ and $[n]$ have the same cardinality.
\end{proposition}

\begin{proof}
Let $\varphi: [n]^r \to F$ be a strong Grassmann-Pl\"{u}cker function. Define $\psi: \cT_n \to F$ to be $\psi(T) := \varphi(a_1, \dots, a_r)$ if $T = B \cup ([n] \setminus B)^\ast$ for $B = \{a_1 < \cdots < a_r\}$, and $\psi(T) = 0$ otherwise. It is obvious that $\psi$ is not identically zero, and we claim that $\psi$ satisfies~\ref{item:W2}. To prove~\ref{item:W2}, we take $T_1, T_2 \in \cT_n$ with $(T_1 \symdiff T_2) \cap [n] = \{x_1 < \cdots < x_m\}$. Suppose without loss of generality that $T_1 \cap [n] = \{b_1 < \dots < b_{r+1}\}$ and $T_2 \cap [n] = \{c_1 < \dots < c_{r-1}\}$. If $x_k \in (T_2 \setminus T_1) \cap [n]$, then $\psi(T_1 \symdiff \skewpair{x_k}) = \psi(T_2 \symdiff \skewpair{x_k}) = 0$. If $x_k = b_j \in (T_1 \setminus T_2) \cap [n]$, then since $|T_1 \cap [x_k]| = j$ and $|T_2 \cap [x_k]| = k - j + 2|T_1 \cap T_2 \cap [x_k]|$, we have 
\begin{align*}
    \psi(T_1 \symdiff \skewpair{x_k}) = \varphi(b_1, \dots, \hat{b_j}, \dots, b_{r+1})
    \,\text{ and }\,
    \psi(T_2 \symdiff \skewpair{x_k}) = (-1)^{k-j} \varphi(b_j, c_1, \dots, c_{r-1}). 
\end{align*}
It follows that 
\begin{align*}
 \sum_{k = 1}^m (-1)^k \psi(T_1 \symdiff \skewpair{x_k}) \psi(T_2 \symdiff \skewpair{x_k})
  =
  \sum_{j = 1}^{r+1} (-1)^{j} \varphi(b_1, \dots, \hat{b_j}, \dots, b_{r+1})\varphi(b_j, c_1, \dots, c_{r-1}) \in N_F.
\end{align*}
Therefore, $\psi$ is a strong Wick function whose support forms the bases of $\lift{\underline{M}_\varphi}$. 

Conversely, let $\psi$ be a strong Wick function on $E = [n] \cup [n]^\ast$ such that all elements of $\{B \cap [n] : B \in \supp(\psi)\}$ have the same cardinality $r$. Let $\varphi: [n]^r \to F$ be the (unique) function satisfying~\ref{GP1} and \ref{GP2} defined by $\varphi(a_1, \dots, a_r) := \psi(T)$ where $T = \{a_1, \dots, a_r\} \cup ([n] \setminus \{a_1, \dots, a_r\})^\ast$ for all $\{a_1 < \cdots < a_r\} \subseteq [n]$. Take $J_1 = \{b_1 < \cdots < b_{r+1}\}, J_2 = \{c_1, \dots, c_{r-1}\} \subseteq [n]$, and write $J_1' = J_1 \cup ([n] \setminus J_1)^\ast$ and $J_2' = J_2 \cup ([n] \setminus J_2)^\ast$. Then 
\begin{align*}
 \sum_{j = 1}^{r+1} (-1)^j \varphi(b_1, \dots, \hat{b_j}, \dots, b_{r+1}) \varphi(b_j, c_1, \dots, c_{r-1}) 
= \sum_{j = 1}^{r+1} (-1)^j \cdot \psi(J_1' \symdiff \skewpair{b_j}) \cdot (-1)^{m_k}\psi(J_2' \symdiff \skewpair{b_j}), 
\end{align*}
where $m_j$ is the number of elements in $J_2$ that are less than $b_j$. Write $(J_1' \symdiff J_2') \cap [n] = \{x_1, \dots, x_m\}$. If $e \in J_2$, then since all elements of $\{B \cap [n] : B \in \supp(\psi)\}$ have cardinality $r$, we have $\psi(J_2' \symdiff \skewpair{e}) = 0$. Therefore, we have 
\begin{align*}
 \sum_{j = 1}^{r+1} (-1)^j \varphi(b_1, \dots, \hat{b_j}, \dots, b_{r+1}) \varphi(b_j, c_1, \dots, c_{r-1}) 
= \sum_{k = 1}^m (-1)^k \psi(J_1' \symdiff \skewpair{x_k})\psi(J_2' \symdiff \skewpair{x_k}) \in N_F. 
\end{align*}

It's not hard to see that the two constructions are inverses of each other. 
\end{proof}

\begin{remark}
The variant of Proposition~\ref{prop:W and GP} 
for weak $F$-matroids and weak orthogonal $F$-matroids
holds and the proof is similar. 
\end{remark}

\subsection{Orthogonal signatures and circuit sets}\label{sec:OrthogonalSignature}

Let $\ul{M}$ be an ordinary orthogonal matroid on $E$. As in Section~\ref{section:F-matroids}, we denote the {\em support} of $X \in F^E$ by $\underline{X} = \{i \in E: X(i) \neq 0\}$.
If $X \in F^E$, we write $X^\ast \in F^E$ for the function defined by $X^*(i) := X(i^*)$. Notice that this induces an obvious involution $\ast$ on the subsets of $F^E$. 

\begin{definition}
A subset $\cC \subseteq F^{E}$ is an {\em $F$-signature of $\ul{M}$} if the following hold:
\begin{enumerate}[label=\rm(\roman*)]
\item The support $\underline{\cC} = \{\underline{X} : X \in \cC\}$ of $\cC$ is the set of circuits of $\ul{M}$. 
\item If $X \in \cC$ and $\alpha \in F^\times$, then $\alpha X \in \cC$.
\end{enumerate}
We call $\underline{M}_{\cC} := \ul{M}$ the underlying orthogonal matroid of $\cC$ and call each element of $\mathcal{C}$ an {\em $F$-circuit}.
\end{definition}

The {\em inner product} $\langle \cdot, \cdot \rangle$ on $F^E$ with respect to the involution $x \mapsto \overline{x}$ is defined to be
  \begin{align*}
    \langle X, Y \rangle = \sum_{i\in [n]} ( X(i) \overline{Y(i)} + \overline{X(i^\ast)} Y(i^\ast)).
  \end{align*}
Note that $\langle Y, X \rangle = \overline{\langle X, Y \rangle}$. Let $\conj{\cdot}: F^E \to F^E$ be such that $\conj{X}(i) = X(i)$ if $i\in[n]$ and $\conj{X}(i) = \overline{X(i)}$ otherwise.
Then $\langle X, Y^\ast \rangle = \sum_{i \in E} \conj{X}(i) \conj{Y}(i^\ast)$.

We say that an $F$-signature $\cC$ of $\ul{M}$ satisfies the {\em $2$-term orthogonality} if the following holds:
\begin{enumerate}[label=\rm(O$_2$)]
  \item\label{item:Ot2} $\langle X, Y^* \rangle \in N_F$ for all $X,Y \in \cC$ with $|\underline{X} \cap \underline{Y}^*| = 2$,
\end{enumerate}

\begin{lemma}\label{lem:F-circuit unique up to}
  Let $\cC$ be an $F$-signature of $\ul{M}$ satisfying the $2$-term orthogonality~\ref{item:Ot2}.
  If $X,X' \in \cC$ and $\underline{X} = \underline{X'}$, then $X = \alpha X'$ for some $\alpha \in F^\times$.
\end{lemma}
\begin{proof}
   Consider two $F$-circuits $X$ and $X'$ in $\cC$ with $\underline{X} = \underline{X'} = C$.
   Suppose for contradiction that there exist distinct elements $e,f\in C$ with $X(e)/X(f) \neq X'(e)/X'(f)$. 
   Let $B$ be a basis of $M$ containing $C \symdiff \skewpair{e}$, and let $D$ be the fundamental circuit $C(B,f^*)$. 
   Then $C \cap D^\ast = \{e, f\}$.
   Let $Y$ be an $F$-circuit in $\cC$ such that $\underline{Y} = D$.
   Then $\langle X, Y^\ast \rangle = \conj{X}(e)\conj{Y}(e^*) + \conj{X}(f)\conj{Y}(f^*) \in N_F$ by~\ref{item:Ot2} and thus $\conj{X}(e)/\conj{X}(f) = \conj{Y}(f^*)/\conj{Y}(e^*)$.
   We also have the same result for $X'$, a contradiction.
\end{proof}

\begin{definition}
We call an $F$-signature $\cC$ of $\ul{M}$ a {\em strong orthogonal $F$-signature of $\ul{M}$} if 
\begin{enumerate}[label=\rm(O)]
\item\label{item:O} $\langle X, Y^\ast \rangle \in N_F$ for all $X, Y \in \cC$.
\end{enumerate}
We call an $F$-signature $\cC$ of $\ul{M}$ a {\em weak orthogonal $F$-signature of $\ul{M}$} if 
\begin{enumerate}[label=\rm(O)$'$]
\item\label{item:O'} $\langle X, Y^\ast \rangle \in N_F$ for all $X, Y \in \cC$ with $|\ul{X} \cap \ul{Y}^\ast| \leq 4$.
\end{enumerate}
\end{definition}

\begin{remark}\label{rmk:OS and DP}
Let $(\cC,\cD)$ be a dual pair of $F$-signatures of a matroid $\ul{N}$ on $[n]$, and let $\cC_1$ and $\cD_1$ be the obvious embeddings of $\cC$ and $\cD$ in $F^E = F^{[n] \cup [n]^*}$. By Proposition~\ref{eq:circuits of lift}, $\cC_1 \cup \cD_1^\ast$ is an $F$-signature of $\lift{\ul{N}}$. It is readily seen from definitions that $(\cC,\cD)$ is a strong dual pair of $F$-signatures of~$\ul{N}$ if and only if $\cC_1 \cup \cD_1^*$ is a strong orthogonal $F$-signature of $\lift{\ul{N}}$.
In addition, $(\cC,\cD)$ is a weak dual pair of $F$-signatures of~$\ul{N}$ if and only if $\cC_1 \cup \cD_1^*$ is an $F$-signature of $\lift{\ul{N}}$ which satisfies the following:
\begin{enumerate}[label=\rm(O$_3$)]
  \item\label{item:Ot3} $\langle X, Y^\ast \rangle \in N_F$ for all $X, Y \in \cC$ with $|\ul{X} \cap \ul{Y}^\ast| \leq 3$.
\end{enumerate}
We will show later in Example~\ref{eg:why not choose O3 and Lii} that for some field $K$, there exists a $K$-signature of an orthogonal matroid which satisfies~\ref{item:Ot3} but not~\ref{item:O'}.
\end{remark}


\begin{definition}
  A {\em strong $F$-circuit set} of $\ul{M}$ is an $F$-signature $\cC$ of $\ul{M}$ satisfying~\ref{item:Ot2} and the following property:
  \begin{enumerate}[label=\rm(L)]
    \item\label{item:L} For every $F$-circuit $X \in \cC$ and a basis $B$ of $\ul{M}$, the vector $\tilde{X}$ is in the linear span of $\{\tilde{X_e}\}_{e\in B^*}$, where $X_e$ is an $F$-circuit in $\cC$ with support $C(B,e)$.
  \end{enumerate}
  A {\em weak $F$-circuit set} of $\ul{M}$ is an $F$-signature $\cC$ of $\ul{M}$ satisfying~\ref{item:Ot2} and the next weakened version  of~\ref{item:L}:
  \begin{enumerate}[label=\rm(L-\roman*)$'$]
    \item\label{item:L1} Let $B$ be an arbitrary basis of $\ul{M}$, and let $e_1, e_2 \in B^*$ be distinct. Let $X_i \in \cC$ be an $F$-circuit with support $\underline{X_i} = C(B,e_i)$ for $i = 1,2$. If $\underline{X_1} \cup \underline{X_2}$ is admissible and if $f \in \underline{X_1} \cap \underline{X_2}$, then there exists an $F$-circuit $Y \in \cC$ such that $Y(f) = 0$ and $\tilde{Y}$ is in the linear span of $\tilde{X_1}$ and $\tilde{X_2}$.
    \item\label{item:L2}Let $B$ be an arbitrary basis of $\ul{M}$, and let $e_1, e_2, e_3 \in B^*$ be distinct.  Let $X_i \in \cC$ be an $F$-circuit with support $\underline{X_i} = C(B,e_i)$ for $i = 1,2,3$. If none of $\underline{X_i} \cup \underline{X_j}$ with $1\le i< j \le 3$ is admissible, then there exists an $F$-circuit $Y \in \cC$ such that $Y(e_1^*) = Y(e_2^*) = Y(e_3^*) =0$ and $\tilde{Y}$ is in the linear span of $\tilde{X}_1$, $\tilde{X}_2$, and $\tilde{X}_3$.
  \end{enumerate}
\end{definition}

\begin{remark}\label{rmk:F-circuit sets of matroids and orthogonal matroids}
  Let $\cC$ be a weak $F$-circuit set of a matroid $\ul{N}$ on $[n]$. By~\cite{Baker2019}, its dual $\cD$ is the $F$-signature of the dual matroid $\ul{N}^*$ such that $X \perp Y$ for all $X \in \cC$ and $Y \in \cD$ with $|\underline{X}\cap\underline{Y}|=2$.
  Let $\cC_1$ and $\cD_1$ be natural embeddings of $\cC$ and $\cD$ into $F^{[n]\cup[n]^*}$.
  Then $\cC_1 \cup \cD_1^*$ is an $F$-signature of $\lift{\ul{N}}$ that satisfies~\ref{item:Ot2} and~\ref{item:L1} by definition, and
  $\cC_1 \cup \cD_1^*$ vacuously satisfies \ref{item:L2}.
  Therefore, $\cC_1 \cup \cD_1^*$ is an weak $F$-circuit set of $\lift{\ul{N}}$.
  If $\cC$ is a strong $F$-circuit set of $\ul{N}$, then $\cC_1 \cup \cD_1^*$ is a strong $F$-circuit set of $\lift{\ul{N}}$.

  Indeed, denote by $\pi:F^{[n] \cup [n]^*} \to F^{[n]}$ the canonical projection map. Then an $F$-signature~$\cC$ of $\lift{\ul{N}}$ is a weak (resp. strong) $F$-circuit set if and only if $\{\pi(X): X\in \cC \text{ with } \underline{X} \subseteq [n]\}$ and $\{\pi(X^*): X\in \cC \text{ with } \underline{X}^* \subseteq [n]\}$ are weak (resp. strong) $F$-circuit sets of $\ul{N}$ and $\ul{N}^*$ respectively, and those two $F$-circuit sets are the dual of each other. 
\end{remark}

\begin{example}\label{eg:why not choose O3 and Lii}
  Let $K$ be a field with $|K^\times| > 1$ and $\charac(K) \neq 3$ and let $x \in K \setminus \{0,-3\}$. We assume the trivial involution on $K$. Let $\cC$ be a subset of $K^{[4]\cup[4]^*}$ consisting of the following eight vectors and their scalar multiples by nonzero elements:
  \begin{align*}
    &
    \begin{pmatrix} 0 & 1 & 1 & 1 \\ 1 & 0 & 0 & 0 \end{pmatrix},
    \qquad
    \begin{pmatrix}-1 & 0 & 1 &-1 \\ 0 & 1 & 0 & 0 \end{pmatrix},
    \qquad
    \begin{pmatrix}-1 &-1 & 0 & 1 \\ 0 & 0 & 1 & 0 \end{pmatrix},
    \qquad
    \begin{pmatrix}-1 & 1 &-1 & 0 \\ 0 & 0 & 0 & 1 \end{pmatrix}, \\
    &
    \begin{pmatrix} x & 0 & 0 & 0 \\ 0 & 1 & 1 & 1 \end{pmatrix},
    \qquad
    \begin{pmatrix} 0 & x & 0 & 0 \\-1 & 0 & 1 &-1 \end{pmatrix},
    \qquad
    \begin{pmatrix} 0 & 0 &-x & 0 \\ 1 & 1 & 0 &-1 \end{pmatrix},
    \qquad
    \begin{pmatrix} 0 & 0 & 0 &-x \\ 1 &-1 & 1 & 0 \end{pmatrix}, 
  \end{align*}
  where $\begin{pmatrix} a_1&a_2&a_3&a_4 \\ b_1&b_2&b_3&b_4 \end{pmatrix}$ means $X \in K^{[4]\cup[4]^*}$ such that $X(i)=a_i$ and $X(i^*)=b_i$ with $i\in[4]$.
  Then~$\cC$ is a $K$-signature of the orthogonal matroid whose set of bases is $\{[4], [4]^\ast\} \cup \{ijk^\ast l^\ast: ijkl = [4]\}$. Notice that $\cC$ satisfies~\ref{item:Ot3} and~\ref{item:L1}, but neither~\ref{item:O'} nor~\ref{item:L2}.
\end{example}

We prove the following results in Section~\ref{sec:OMoverTracts}.

\begin{theorem}\label{thm:strong orthogonal signatures and circuit sets}
  An $F$-signature of an orthogonal matroid is a strong orthogonal $F$-signature if and only if it is a strong $F$-circuit set.
\end{theorem}

\begin{theorem}\label{thm:weak orthogonal signatures and circuit sets}
  Every weak $F$-circuit set of an orthogonal matroid is a weak orthogonal $F$-signature.
\end{theorem}

The converse of Theorem~\ref{thm:weak orthogonal signatures and circuit sets} is not true; see Example~\ref{eg:nonidyll}.

\subsection{Orthogonal $F$-vector sets}\label{sec:VectorSet}

Let $\cV$ be a subset of $F^E$. A vector $X \in \cV$ is {\em elementary} in $\cV$ if (i) it is nonzero, and (ii) it has a minimal support in $\cV \setminus \{\mathbf{0}\}$, and (iii) its support $\underline{X}$ is admissible. A transversal $T \in \mathcal{T}_n$ is a {\em support basis} of $\cV$ if there is no $X \in \cV \setminus \{\mathbf{0}\}$ such that $\underline{X} \subseteq T$. A {\em fundamental circuit form} for $\cV$ with respect to a support basis $B$ is $\{X^{\cV}_{B,e}: e\in B^*\}$ where $X^{\cV}_{B,e} \in \cV$ is such that $\supp(X^{\cV}_{B,e}) \subseteq B \symdiff \skewpair{e}$ and $X^{\cV}_{B,e}(e) = 1$.
We simply write $X^{\cV}_{B,e}$ as $X_e$ if it is clear from the context.

\begin{definition}
We call $\mathcal{V} \subseteq F^E$ an {\em orthogonal $F$-vector set} if the following hold:
\begin{enumerate}[label=\rm(V\arabic*)]
\item\label{item:V1} For all elementary vectors $X, Y \in \mathcal{V}$, if $|\underline{X} \cap \underline{Y}^*| \leq 2$, then $\langle X, Y^\ast \rangle \in N_F$. 

\item\label{item:V2} Support bases exist, and for each support basis $B$, there exists a corresponding fundamental circuit form. 

\item\label{item:V3} $\mathcal{V}$ is exactly the set of vectors $X \in F^E$ such that for every support basis $B$ of $\cV$, $\conj{X}$ belongs to the linear span of $\{\conj{X}_e:e\in B^*\}$.
    \end{enumerate}    
\end{definition}

The axiom~\ref{item:V3} implies the uniqueness of the fundamental circuit form for each support basis of an orthogonal $F$-vector set $\cV$, and that every fundamental circuit form of an orthogonal $F$-vector set $\cV$ consists of elementary vectors of $\cV$.
When $F$ is a field, a subset $\cW \subseteq F^{[n]}$ is an $F$-vector set if and only if it is a linear subspace~\cite{Anderson2019}. We give an analogue of this for orthogonal $F$-vector sets.

\begin{theorem}\label{thm:lagrangian subspace}
Let $F$ be a field and $\cV$ be a subset of $F^E$. 
\begin{enumerate}[label=\rm(\roman*)]
\item If $\cV$ is an orthogonal $F$-vector set, then it is a Lagrangian subspace with respect to the inner product $\langle \cdot, (\cdot)^* \rangle$.

\item Whenever $\mathrm{char}(F) \neq 2$, the converse of {\rm (i)} holds.
\end{enumerate}
\end{theorem}

We delay the proof of Theorem~\ref{thm:lagrangian subspace}(i) to Section~\ref{sec:OV}.
Theorem~\ref{thm:lagrangian subspace}(ii) can be deduced from the results of~\cite{Oum2012wqo}.
The condition that $\mathrm{char}(F) \neq 2$ in (ii) is crucial, since otherwise $\mathcal{V} = \{(x,x): x \in F\}$ is a Lagrangian subspace of $F^{[1]\cup[1]^*}$ but not an orthogonal $F$-vector set. 

\begin{lemma}[Oum, Propositions~4.2 and~4.3(i) of~\cite{Oum2012wqo}]\label{lem:Lagrangian C5}
  Let $F$ be a field and let $\cV \subseteq F^E$  be a Lagrangian subspace with respect to $\langle \cdot, (\cdot)^* \rangle$.
  \begin{enumerate}[label=\rm(\roman*)]
    \item There is a support basis of $\cV$.
    \item If $\charac(F) \neq 2$, then for each support basis $B$ of $\cV$ and $x \in B^*$, there exists a unique vector $X \in \cV$ such that $\underline{X} \subseteq B \symdiff \skewpair{x}$ and $X(x) = 1$.
  \end{enumerate}
\end{lemma}




\begin{proof}[Proof of Theorem~\ref{thm:lagrangian subspace}(ii)]
    Since $\cV$ is isotropic, it satisfies~\ref{item:V1}.
    By Lemma~\ref{lem:Lagrangian C5}, \ref{item:V2} holds.
    Since the $n$ vectors in each fundamental circuit form are independent, $\cV$ satisfies~\ref{item:V3}. 
    Therefore, $\cV$ is an orthogonal $F$-vector set.
\end{proof}

\subsection{Main theorems}\label{sec:main results}

We prove the equivalence of four notions of strong orthogonal matroids over tracts.

\begin{theorem}\label{thm:main}
  Let $E = [n] \cup [n]^*$ and let $F$ be a tract endowed with an involution $x \mapsto \overline{x}$. Then there are natural bijections between: 
  \begin{enumerate}[label=\rm(\arabic*)]
      \item Strong orthogonal $F$-matroids on $E$.
      \item Strong orthogonal $F$-signatures on $E$.
      \item Strong $F$-circuit sets of orthogonal matroids on $E$.
      \item Orthogonal $F$-vector sets on $E$.
  \end{enumerate}
\end{theorem}

Similarly, we have the next two equivalences for weaker notions.

\begin{theorem}\label{thm:main-weak1}
  Let $E = [n] \cup [n]^*$ and let $F$ be a tract endowed with an involution $x \mapsto \overline{x}$. Then there is a natural bijection between: 
  \begin{enumerate}[label=\rm(\arabic*)]
      \item Moderately weak orthogonal $F$-matroids on $E$.
      \item Weak orthogonal $F$-signatures on $E$. 
  \end{enumerate}
\end{theorem}

\begin{theorem}\label{thm:main-weak2}
  Let $E = [n] \cup [n]^*$ and let $F$ be a tract endowed with an involution $x \mapsto \overline{x}$. Then there is a natural bijection between: 
  \begin{enumerate}[label=\rm(\arabic*)]
      \item Weak orthogonal $F$-matroids on $E$.
      \item Weak $F$-circuit sets of orthogonal matroids on $E$. 
  \end{enumerate}
\end{theorem}

We will provide proofs for Theorems~\ref{thm:main}, \ref{thm:main-weak1}, and~\ref{thm:main-weak2} in Section~\ref{sec:OMoverTracts}. Since the notions of weak and strong orthogonal $F$-matroid coincide if $F$ is a partial field~\cite{BJ2022}, the tropical hyperfield~$\T$~\cite{Rincon2012}, or the Krasner hyperfield~$\K$, it follows that the three notions of orthogonal $F$-matroids are equivalent when $F$ is any of these specific tracts. 

We summarize our results in Figure~\ref{fig:summary}.
Additionally, we remark that in Figure~\ref{fig:summary}, each inclusion is strict for certain tracts; see Examples~\ref{eg:nonidyll} and~\ref{eg:nonidyll2}.

\begin{figure}[!h]
  \centering
  \scalebox{0.88}{
  \begin{tikzpicture}
    \node (SM)   at (-4.5,0) {$\begin{Bmatrix}\text{Strong}\\ \text{$F$-matroids}\end{Bmatrix}$};
    \node        at (-4.5,-2) {$\bigcap$};
    \node (WM)   at (-4.5,-4) {$\begin{Bmatrix}\text{Weak}\\ \text{$F$-matroids}\end{Bmatrix}$};

    \node (SOM)  at (0,0) {$\begin{Bmatrix}\text{Strong orthogonal}\\ \text{$F$-matroids}\end{Bmatrix}$};
    \node at (0,-1) {$\bigcap$};
    \node (DWOM) at (0,-2) {$\begin{Bmatrix}\text{Moderately weak}\\ \text{orthogonal $F$-matroids}\end{Bmatrix}$};
    \node        at (0,-3) {$\bigcap$};
    \node (WOM)  at (0,-4) {$\begin{Bmatrix}\text{Weak orthogonal}\\ \text{$F$-matroids}\end{Bmatrix}$};
    
    \node (SC)   at (5.5,1.5) {$\begin{Bmatrix}\text{Strong $F$-circuit sets}\\ \text{of orthogonal matroids}\\\end{Bmatrix}$};
    \node        at (5.5,0.75) {\rotatebox{90}{\Large\textbf{=}}};
    \node        at (6.12,0.75) {\ref{thm:strong orthogonal signatures and circuit sets}};
    \node (SOS)  at (5.5,0) {$\begin{Bmatrix}\text{Strong orthogonal}\\ \text{$F$-signatures}\end{Bmatrix}$};
    \node        at (5.5,-1) {$\bigcap$};
    \node (WOS)  at (5.5,-2) {$\begin{Bmatrix}\text{Weak orthogonal}\\ \text{$F$-signatures}\end{Bmatrix}$};
    \node        at (5.5,-3) {$\bigcap$};
    \node        at (6.12,-3) {\ref{thm:weak orthogonal signatures and circuit sets}};
    \node (WC)   at (5.5,-4) {$\begin{Bmatrix}\text{Weak $F$-circuit sets}\\ \text{of orthogonal matroids}\end{Bmatrix}$};

    \node        at (10.5,1.5) {$\begin{Bmatrix}\text{Lagrangian}\\ \text{subspaces}\\\end{Bmatrix}$};
    \node        at (10.5,0.75) {\rotatebox{90}{\Large\textbf{=}}};
    \node        at (10.1,0.75) {{\tiny (b)}};
    \node        at (11.1,0.75) {\ref{thm:lagrangian subspace}};
    \node (OV)   at (10.5,0) {$\begin{Bmatrix}\text{Orthogonal}\\ \text{$F$-vector set}\end{Bmatrix}$};

    \draw[-stealth] (SM) -- node[above] {\ref{prop:W and GP}} (SOM);
    \draw[-stealth] (WM) -- (WOM);
    
    \draw[stealth-stealth] (SOM) -- node[above] {\ref{thm:main}} (SOS);
    \draw[stealth-stealth] (DWOM) -- node[above] {\ref{thm:main-weak1}} (WOS);
    \draw[stealth-stealth] (WOM) -- node[above] {\ref{thm:main-weak2}} (WC);

    \draw[stealth-stealth] (SOM.south west) to [bend right = 35] 
    node[left] {{\tiny (a)}}
    (WOM.north west);

    \draw[stealth-stealth] (SOS) -- node[above] {\ref{thm:main}} (OV);
  \end{tikzpicture}
  }
  \caption{Summary of results in Section~\ref{sec:Wick functions}--\ref{sec:main results}. \\ In (a), we assume that $F\in \{\T,\K\}$ or $F$ is a partial field~\cite{BJ2022,Rincon2012}. \\ In (b), we assume that $F$ is a field with $\charac(F) \ne 2$.}
  \label{fig:summary}
\end{figure}

\begin{remark}
        In~\cite{Baker2019}, Baker and Bowler defined strong and weak matroids over a tract, and showed cryptomorphisms among different axiom systems. For orthogonal matroids, we introduce a third moderately weak orthognal $F$-matroids over a tract $F$. We have various reasons for this. First, Example~\ref{eg:nonidyll} shows that there are no bijections between weak orthogonal $F$-matroids and weak orthognal $F$-signatures, while Theorems~\ref{thm:W2O-weak} and~\ref{thm:O2W-weak} prove that there is a natural bijection between moderately weak orthognal $F$-matroids and weak orthognal $F$-signatures. Second, Wenzel defined in~\cite{wenzel-maurer2} the Tutte group and the inner Tutte group of an orthogonal matroid, where the multiplicative relations for the Tutte groups recognize the axiom system for the moderately weak orthogonal matroids; see the Remark after~\cite[Definition 2.5]{wenzel-maurer2}. Finally, our ongoing work shows that  we can define the universal pasture and the foundation of an orthogonal matroid that represent respectively the functors taking a pasture to the set of moderately weak orthogonal $F$-matroids and to the set of rescaling equivalence classes of moderately weak orthogonal $F$-matroids in the sense of~\cite{BO21} and~\cite{BO23}; we will not further elaborate on this direction as it is not the main goal of the present paper. 
\end{remark}

\subsection{Functoriality. }\label{sec:OM functionality}

Now we discuss the behavior of strong and weak orthogonal matroids over tracts with respect to tract homomorphisms. 
The following propositions are all straightforward from the definitions. 

\begin{proposition}\label{prop:pushforward Wick}
Let $f : F_1 \to F_2$ be a tract homomorphism, and let $\varphi$ be a strong 
Wick function with coefficients in $F_1$. Then the composition $f \circ \varphi$ is a strong 
Wick function with coefficients in $F_2$.
The same results hold for weak and moderately weak Wick functions.
    \qed
\end{proposition}

By the above proposition, we define the {\em pushforward} operator $f_*$ taking 
orthogonal $F_1$-matroids to 
orthogonal $F_2$-matroids by $f_*([\varphi]) := [f \circ \varphi]$. Notice that the pushforwards are functorial: if $F_1 \xrightarrow{f} F_2 \xrightarrow{g} F_3$ are tract homomorphisms, then $(g \circ f)_\ast = g_\ast \circ f_\ast$. If $F_2 = \K$, the terminal object of the category of tracts, then the orthogonal $\K$-matroid $[f \circ \varphi]$ is the same thing as $\underline{M}_\varphi$. 

\begin{proposition}
Let $F_1$ and $F_2$ be tracts equipped with involutions $\iota_1$ and $\iota_2$, respectively, and let $f : F_1 \to F_2$ be a tract homomorphism that respects the involutions, i.e. $f \circ \iota_1 = \iota_2 \circ f$. If $\cC$ is a strong \textup{(}resp. weak or moderately weak\textup{)} orthogonal $F_1$-signature of an ordinary orthogonal matroid $\ul{M}$, then $f_*(\cC) := \{cf(X) : c \in F_2^\times, X \in \cC\}$ is a strong \textup{(}resp. weak or moderately weak\textup{)} orthogonal $F_2$-signature of $\ul{M}$. 
The same results hold for strong and weak circuit sets over tracts of an orthogonal matroid.
    \qed
\end{proposition}

Therefore, we also have the {\em pushforward} operator $f_*$ taking 
orthogonal $F_1$-signatures (resp. $F_1$-circuit set) to 
orthogonal $F_2$-signatures (resp. $F_2$-circuit set). If $F_2 = \K$, then $f_*(\cC)$ is the same thing as the set of circuits of $\underline{M}_{\cC}$. 

However, the simple pushforwards of orthogonal $F$-vector sets are not defined properly. In fact, let $f : F_1 \to F_2$ be a tract homomorphism respecting the involutions and let $\cV$ be an orthogonal $F_1$-vector set, then the set $f_*(\cV) = \{c f(X) : c\in F_2^\times, X \in \cV\}$ is not necessarily an orthogonal $F_2$-vector set; see Example~\ref{eg:no pushforward for orthogonal vector sets}.

\begin{proposition}\label{prop:product}
Let $F_1, F_2$ be tracts, and let $\varphi_1, \varphi_2$ be strong 
Wick functions with coefficients in $F_1, F_2$, respectively, with the same underlying orthogonal matroid $\ul{M}$. Then $\varphi_1 \times \varphi_2 : \cT_n \to F_1 \times F_2$ defined as $(\varphi_1 \times \varphi_2)(T) = ( \varphi_1(T), \varphi_2(T) )$ is a strong 
Wick function with coefficients in the product $F_1 \times F_2$.
The same results hold for weak and moderately weak Wick functions.
    \qed
\end{proposition}

\subsection{Duality. }

Let $\varphi: \cT_n \to F$ be a strong 
Wick function over $F$. Its {\em dual strong 
Wick function} $\varphi^\ast: \cT_n \to F$ is defined as 
\[
\varphi^\ast(T) := \varphi(T^\ast)
\]
for all $T \in \cT_n$. It is indeed a strong 
Wick function with underlying orthogonal matroid $(\ul{M}_{\varphi})^\ast$ from definitions.
We define the duals of weak and moderately weak Wick $F$-functions in the same way.

Given a strong 
(resp. weak) orthogonal $F$-signature $\cC$, we can define its {\em dual strong \textup{(}resp. weak\textup{)} orthogonal $F$-signature} $\cC^\ast$ by setting
\[
\cC^\ast := \{X^\ast : X \in \cC\}, 
\]
and the underlying orthogonal matroid of $\cC^\ast$ is $(\ul{M}_{\cC})^\ast$. 
The duals of strong and weak $F$-circuit sets of orthogonal matroids are defined in the same way.

\subsection{Minors. }\label{sec:OM minors}

Let $\varphi$ be a strong or (moderately) weak Wick function on $E$ with coefficients in $F$ and take $e \in E$. Then we define $\varphi|e$ to be the function from the set of transversals of $E \setminus \skewpair{e}$ to $F$ as
\begin{align*}
    (\varphi|e) (T) :=
    \begin{cases}
        \varphi(T\cup \{e\})    & \text{if $e$ is nonsingular in $\underline{M}_\varphi$}, \\
        \varphi(T \cup \{e^*\}) & \text{otherwise}.
    \end{cases}
\end{align*}

\begin{proposition}
Let $\varphi$ be a strong 
Wick $F$-function $\varphi$ on $E$ and $e\in E$.
Then $\varphi|e$ is a strong 
Wick $F$-function and $\underline{M}_{\varphi|e} = \underline{M}_{\varphi}|e$. 
The same holds for weak and moderately weak Wick $F$-functions.
\qed
\end{proposition}

We define minors of strong or weak orthogonal signatures as follows. Let $\mathcal{C}$ be a strong or weak orthogonal $F$-signature of an orthogonal matroid $\ul{M}$ on $E$. For $e\in E$, let $\cC|e$ be the set of functions in
\begin{align*}
     \{\pi(X) \in F^{E \setminus \skewpair{e}} : \text{$X\in \mathcal{C}$ with $X(e^*) = 0$ 
     and $\underline{X} \neq \{e\}$}\} 
\end{align*}
that have minimal supports, 
where $\pi:F^E \to F^{E\setminus \skewpair{e}}$ is the obvious projection. 

The next proposition is direct from Proposition~\ref{prop:circuits of a minor}.

\begin{proposition}\label{prop:minors of orthogonal signatures}
    Let $\cC$ be a strong \textup{(}resp. weak\textup{)} orthogonal $F$-signature of an orthogonal matroid $\ul{M}$ on $E$ and let $e\in E$.
    Then
    $\mathcal{C}|e$ is a strong \textup{(}resp. weak\textup{)} orthogonal $F$-signature of $\ul{M}|e$.
    \qed
\end{proposition}

Minors of a strong or weak $F$-circuit set of an orthogonal matroid are defined in the same way as minors of an orthogonal $F$-signature, and an analogue of Proposition~\ref{prop:minors of orthogonal signatures} holds.

One possible candidate of a minor of an orthogonal $F$-vector set $\cV \subseteq F^E$ with respect to $e\in E$ is
\begin{align*}
  \cV|e
    := \{\pi(X) \in F^{E \setminus \skewpair{e}} : \text{$X\in \cV$ with $X(e^*) = 0$}\},
\end{align*}
which coincides with the deletion and the contraction of an $F$-vector set of a matroid in~{\cite[Section~4.2]{Anderson2019}}.
However, $\cV|e$ is not necessarily an orthogonal $F$-signature in general, even if $F$ is a partial field and the underlying orthogonal matroid of $\cV$ is the lift of a matroid; see Example~\ref{eg:minor of orthogonal vertex set}.
We remark that if $F$ is a field, then $\cV|e$ is an orthogonal $F$-vector set by~{\cite[Proposition~3.8]{Oum2012wqo}}.

\subsection{Other related work}\label{sec:equiv of diff defs}
We briefly indicate how our notions of strong and weak orthogonal $F$-matroids generalize various flavors of orthogonal matroids in the literature, as mentioned in Section~\ref{section:intro}. 

\begin{example}
If the support of a strong or weak orthogonal matroid 
on $E$ over $F$ is the lift of an ordinary matroid on $[n]$, then an orthogonal matroid on $E$ over $F$ is the same thing as a strong or weak matroid on $[n]$ over $F$ in the sense of~\cite{Baker2019}. This follows from Proposition~\ref{prop:W and GP}. 
\end{example}
 
\begin{example}
A strong or weak orthogonal matroid over the Krasner hyperfield $\mathbb{K}$ is the same thing as an ordinary orthogonal matroid. 
\end{example}

\begin{example}
When $F = K$ is a field, a strong or weak Wick $K$-matroid is the same thing as a projective solution to Wick equations in $\bP^N(K)$, where $N = 2^{n}-1$. In addition, when $\mathrm{char}(K) \ne 2$, a strong or weak orthogonal $K$-matroid is the same thing as a maximal isotropic subspace of $K^{2n}$ in the usual sense. Indeed, a weak Wick function with coefficients in the field $K$ automatically satisfies~\ref{item:W2}. This follows from~\cite[Theorem~1.6]{BJ2022}. 
\end{example}

\begin{example}
A strong or weak orthogonal matroid over the regular partial field $\mathbb{U}_0$ is the same thing as a regular orthogonal matroid in the sense of~\cite{Geelen1996thesis}. This follows from the discussion on page 33 in {\em loc. cit.} and~\cite[Theorem~1.6]{BJ2022}. 
\end{example}

\begin{example}
A strong or weak orthogonal matroid over the tropical hyperfield $\T$ is the same thing as a valuated orthogonal matroid in the sense of~\cite{Dress1991} 
or a tropical Wick vector in the sense of~\cite{Rincon2012}. This follows from Theorem~5.1 of {\em loc. cit.} 
\end{example}

\begin{example}
A strong orthogonal matroid over the sign hyperfield $\bS$ is the same thing as an oriented orthogonal matroid in the sense of~\cite{Wenzel1993,Wenzel1996}. This follows from the discussion at the top of page 241 of~\cite{Wenzel1996}. 
\end{example}


\section{Cryptomorphisms for Orthogonal Matroids over Tracts}\label{sec:OMoverTracts}

In this section, we give proofs of the main theorems of the paper and confirm Figure~\ref{fig:summary}. 
Our plan is as follows. We first construct strong (resp. weak) orthogonal signatures from strong (resp. moderately weak) Wick functions in Section~\ref{sec:W2O}, and show the converse in Section~\ref{sec:O2W}.
In Section~\ref{sec:weak Wick functions and circuit sets} we show the equivalence between weak orthogonal $F$-matroids and weak $F$-circuit sets using the constructions in Sections~\ref{sec:W2O} and~\ref{sec:O2W}.
In Section~\ref{sec:OC} we prove Theorem~\ref{thm:strong orthogonal signatures and circuit sets} that orthogonal $F$-signatures and $F$-circuit sets coincide for the strong case.
In Section~\ref{sec:OV} we show the equivalence between strong orthogonal signatures and orthogonal vector sets, as well as Theorem~\ref{thm:lagrangian subspace}(i).
We sum up all main theorems in Section~\ref{sec:strong equiv}. Section~\ref{sec:more-examples} provides several pathological examples. 

Recall that $\mathcal{T}_n$ denotes the family of all transversals of $E = [n] \cup [n]^\ast$.
For every $i \in E$, let $\nostar{i}$ be the element in $[n]$ such that $\skewpair{i} \cap [n] = \{\nostar{i}\}$. For $i,j \in [n]$, let $\ocinterval{i}{j}$ be the subset $\{k\in[n]: i<k\leq j\}$ if $i \leq j$, and $\ocinterval{j}{i}$ otherwise. For $T \in \cT_n$ and $i,j\in E$, let $m^T_{i,j}$ denote $|T \cap \ocinterval{\nostar{i}}{\nostar{j}}|$. We often omit the superscription $T$ in $m^T_{i,j}$ if it is clear from the context. If $\alpha,\beta \in F^\times$, we write $\frac{\beta}{\alpha}$ for $\alpha^{-1} \beta$. We often denote a finite set $S = \{a_1,a_2,\dots,a_m\}$ by enumerating its elements, such as $a_1a_2\dots a_m$.

\subsection{From Wick functions to orthogonal signatures}\label{sec:W2O}

Let $\varphi$ be a weak Wick function on $E$ with coefficients in a tract $F$.
We denote by $\ul{M} = \underline{M}_\varphi$ the underlying orthogonal matroid of $[\varphi]$. We first suggest a candidate for the orthogonal signature induced from the given Wick function $\varphi$. 

Recall that the set of bases of $\ul{M}$ is $\supp(\varphi) = \{B \in \cT_n : \varphi(B) \ne 0\}$. For each circuit $C$ of $\ul{M}$, we define a function $X \in F^E$ such that $\underline{X} = C$ as follows. Let $T \supseteq C$ be a transversal such that $T \symdiff \{x, x^\ast\} \in \supp(\varphi)$ for all $x \in C$, which exists by Lemma~\ref{lem:extendcircuit}. Then for every $e, f \in C$, we set 
\begin{align}
    \frac{\conj{X}(e)}{\conj{X}(f)} = (-1)^{m^T_{e,f}} \frac{\varphi(T\symdiff \skewpair{e})}{\varphi(T\symdiff \skewpair{f})}. 
    \label{eq:defineX}
\end{align}
We call $X$ an {\em $F$-circuit of $\varphi$ with support $C$}.

\begin{lemma}\label{lem:defineX}
The ratio $\frac{\conj{X}(e)}{\conj{X}(f)}$ is independent of the choice of $T$. Explicitly, let $T_1, T_2$ be distinct transversals containing $C$ such that both $T_1 \symdiff \skewpair{x}$ and $T_2 \symdiff \skewpair{x}$ are bases of $\ul{M}$ for all $x \in C$. Then
  \[
    (-1)^{m_1}
    \frac{\varphi(T_1 \symdiff \{e, e^\ast\})}{\varphi(T_1 \symdiff \{f, f^\ast\})}
    =
    (-1)^{m_2}
    \frac{\varphi(T_2 \symdiff \{e, e^\ast\})}{\varphi(T_2 \symdiff \{f, f^\ast\})}.
  \]
  where $m_i = |T_i \cap \ocinterval{\nostar{e}}{\nostar{f}}|$ for each $i = 1, 2$.
\end{lemma}

\begin{proof}
We proceed by induction on $|T_1 \setminus T_2|$.
Since $T_1 \symdiff \{x, x^\ast\}$ and $T_2 \symdiff \skewpair{x}$ with $x \in C \neq\emptyset$ are distinct bases of $\ul{M}$, we know that $|T_1 \setminus T_2| = |B_1 \setminus B_2|$ is even and at least $2$.

Suppose that $|T_1 \setminus T_2| = 2$.
Write $T_1 \setminus T_2 = \{a,b\}$ so that $T_1 \symdiff T_2 = \{a, a^\ast, b, b^\ast\}$. Then neither $T_1 \symdiff \skewpair{a}$ nor $T_1 \symdiff \skewpair{b}$ is a basis since they contain $C$. Thus, $\varphi(T_1 \symdiff \skewpair{a}) = \varphi(T_1 \symdiff \skewpair{b}) = 0$. Denote $m = |\{\nostar{a}, \nostar{b}\} \cap \ocinterval{\nostar{e}}{\nostar{f}}| $. Note that $m_1 + m \equiv m_2 \pmod{2}$. By the axiom~\ref{item:W2} applied to $T_1 \symdiff \{e, e^\ast, f, f^\ast\}$ and $T_1 \symdiff \{a, a^\ast, b, b^\ast\} = T_2$, we have 
\begin{align*}
  \varphi(T_1 \symdiff \skewpair{f}) \varphi(T_2 \symdiff \skewpair{e})
  +
  (-1)^{m+1} \varphi(T_1 \symdiff \skewpair{e}) \varphi(T_2 \symdiff \skewpair{f})
  \in N_F,
\end{align*}
which implies the desired equality.
   
Now we assume that $|T_1 \setminus T_2| > 2$.
Fix $x \in C$ and let $B_i := T_i \symdiff \skewpair{x}$ with $i \in \{1,2\}$.
Then $B_1$ and $B_2$ are bases of $\ul{M}$.
Take $y \in T_1 \setminus T_2$.
By the symmetric exchange axiom, there is $z \in (T_1 \setminus T_2) \setminus \{y\}$ such that $B_1 \symdiff \{y,y^*,z,z^*\}$ is a basis of $\ul{M}$.
Let $T_0 := T_1 \symdiff \{y,y^*,z,z^*\}$. We claim that for every $w \in C$, the transversal $T_0 \symdiff \skewpair{w}$ is a basis of $\ul{M}$.
Indeed, since $T_0 \symdiff \skewpair{x} = B_1 \symdiff \{y,y^*,z,z^*\}$, we may assume that $w \neq x$.
Then by Proposition~\ref{prop:fundamental-circuit}, $T_1 \symdiff \skewpair{w} = B_1 \symdiff \{x, x^\ast, w, w^\ast\}$ is a basis of $\ul{M}$. 
Both $B_1 \symdiff \{x,x^*,y,y^*\} = T_1 \symdiff \skewpair{y}$ and $B_1 \symdiff \{x,x^*,z,z^*\} = T_1 \symdiff \skewpair{z}$ contain $C$ and hence neither of them is a basis.
Then the symmetric exchange forces that $T_0 \symdiff \skewpair{w} = B_1 \symdiff \{x,x^*,y,y^*,z,z^*,w,w^*\}$ is a basis. Notice that $|T_0 \symdiff T_1| = 4$ and $|T_0 \symdiff T_2| < |T_1 \symdiff T_2|$. Therefore, we conclude the desired equality by the claim and the induction hypothesis.
\end{proof}

\begin{proposition}\label{prop:defineX2}
Every circuit $C$ of $\ul{M}$ corresponds to a well-defined projective $F$-circuit $X \in F^E$ of $\varphi$ with support $C$, i.e. $X$ is well-defined and unique up to multiplication by an element in $F^\times$.
\end{proposition}

\begin{proof}
Lemma~\ref{lem:defineX} shows 
 the uniqueness. We assert furthermore that $X$ is well-defined. 
This can be proved directly. Let $T$ be a transversal such that $C \subseteq T$ and $T \symdiff \skewpair{x} \in \supp(\varphi)$ for all $x\in C$. Take $e, f, g \in C$. Since $m_{e,f} + m_{f,g} + m_{e,g} \equiv 0 \pmod{2}$, we have 
    \begin{align*}
      \frac{\conj{X}(e)}{\conj{X}(f)} \frac{\conj{X}(f)}{\conj{X}(g)}
      = (-1)^{m_{e,f}} \frac{\varphi(T\symdiff \skewpair{e})}{\varphi(T\symdiff \skewpair{f})}
      \cdot& (-1)^{m_{f,g}} \frac{\varphi(T\symdiff \skewpair{f})}{\varphi(T\symdiff \skewpair{g})} \\
      =& (-1)^{m_{e,g}} \frac{\varphi(T\symdiff \skewpair{e})}{\varphi(T\symdiff \skewpair{g})}
      = \frac{\conj{X}(e)}{\conj{X}(g)}.
      \qedhere
    \end{align*}
\end{proof}

By Proposition~\ref{prop:defineX2}, 
the set $\cC_\varphi$ of all projective $F$-circuits induced from $\varphi$ is an $F$-signature of $\ul{M}$.

\begin{theorem}\label{thm:W2O}
Let $F$ be a tract. If $\varphi$ is a strong Wick function over $F$, then $\mathcal{C}_\varphi$ is a strong orthogonal $F$-signature.
\end{theorem}

\begin{proof}
    Let $X_1, X_2 \in \mathcal{C}_\varphi$. We may assume that $\underline{X_1} \cap \underline{X_2}^* \neq \emptyset$.
  By Lemma~\ref{lem:extendcircuit}, there is a transversal $T_i$ containing $\underline{X_i}$ such that $T_i \symdiff \skewpair{e}$ is a basis for every $e \in \underline{X_i}$ and $i = 1,2$.
  Note that
  \[
    \varphi(T_1 \symdiff \skewpair{e})
    \varphi(T_2 \symdiff \skewpair{e})
    = 0
  \]
  for all $e \in (T_1 \symdiff T_2) \setminus (\underline{X_1} \symdiff \underline{X_2})$.
  Write $T_1 \cap T_2^* = \{e_1, e_2, \dots, e_a\}$ with $\nostar{e_1} < \nostar{e_2} < \dots < \nostar{e_a}$ and write $\underline{X_1} \cap \underline{X_2}^* = \{e_{\alpha_1}, \dots, e_{\alpha_b}\}$ with $\alpha_1 < \dots < \alpha_b$.
  Then $(T_1 \triangle T_2) \cap [n] = \{\nostar{e_1},\dots,\nostar{e_a}\}$.
  Let $m_{j} := |T_1 \cap \ocinterval{\nostar{e_{\alpha_1}}}{\nostar{e_{\alpha_j}}}|$ and $n_{j} := |T_2 \cap \ocinterval{\nostar{e_{\alpha_1}}}{\nostar{e_{\alpha_j}}}|$ for each 
  $j \in [b]$.
  Since $(T_1 \symdiff T_2) \cap \ocinterval{\nostar{e_{\alpha_1}}}{\nostar{e_{\alpha_j}}} = \{\nostar{e_{k}} : \alpha_1 < k \leq \alpha_j\}$, we have ${m_j+n_j} \equiv {\alpha_j - \alpha_1} \pmod{2}$.
  By~\ref{item:W2} applied to $T_1$ and $T_2$, we have
  \begin{align*}
    N_F
    \ni \sum_{i=1}^a (-1)^{i} \varphi(T_1 \symdiff \skewpair{e_i}) \varphi(T_2 \symdiff \skewpair{e_i})
    = \sum_{i=1}^b (-1)^{\alpha_i} \varphi(T_1 \symdiff \skewpair{e_{\alpha_i}}) \varphi(T_2 \symdiff \skewpair{e_{\alpha_i}}).
  \end{align*}
  Therefore,
  \begin{align*}
    \langle X_1, X_2^\ast \rangle
    &=
    \sum_{i = 1}^b \conj{X_1}(e_{\alpha_i}) \conj{X_2}(e_{\alpha_i}^*) \\
    &=
    \conj{X_1}(e_{\alpha_1}) \conj{X_2}(e_{\alpha_1}^*) \sum_{i=1}^b \frac{\conj{X_1}(e_{\alpha_i})}{\conj{X_1}(e_{\alpha_1})} \frac{\conj{X_2}(e_{\alpha_i}^*)}{\conj{X_2}(e_{\alpha_1}^*)} \\
    &=
    \conj{X_1}(e_{\alpha_1}) \conj{X_2}(e_{\alpha_1}^*) \sum_{i=1}^b
    (-1)^{m_i}
    \frac{\varphi(T_1 \symdiff \skewpair{e_{\alpha_i}})}{\varphi(T_1 \symdiff \skewpair{e_{\alpha_1}})}
    (-1)^{n_i}
    \frac{\varphi(T_2 \symdiff \skewpair{e_{\alpha_i}})}{\varphi(T_2 \symdiff \skewpair{e_{\alpha_1}})} \\
    &=
    (-1)^{\alpha_1} \frac{\conj{X_1}(e_{\alpha_1}) \conj{X_2}(e_{\alpha_1}^*)}{\varphi(T_1 \symdiff \skewpair{e_{\alpha_1}}) \varphi(T_2 \symdiff \skewpair{e_{\alpha_1}})} %
    \cdot \sum_{i=1}^b
    (-1)^{\alpha_i} \varphi(T_1 \symdiff \skewpair{e_{\alpha_i}}) \varphi(T_2 \symdiff \skewpair{e_{\alpha_i}})
    \in N_F. \qedhere
  \end{align*}
\end{proof}

\begin{theorem}\label{thm:W2O-weak}
Let $F$ be a tract. If $\varphi$ is a moderately weak Wick function over $F$, then $\mathcal{C}_\varphi$ is a weak orthogonal $F$-signature.
\end{theorem}
\begin{proof}
  It is not hard to see that $\cC_\varphi$ satisfies~\ref{item:O'} if we replace~\ref{item:W2} with~\ref{item:W2'} in the proof of Theorem~\ref{thm:W2O}.
\end{proof}

\subsection{From orthogonal signatures to Wick functions. }\label{sec:O2W}

Throughout this part, $\mathcal{C} \subseteq F^{E}$ is an $F$-signature of an ordinary orthogonal matroid $\ul{M}$ satisfying the $2$-term orthogonality:
\begin{enumerate}[label=\rm(O$_{\arabic*}$)]
    \setcounter{enumi}{1}
    \item[\ref{item:Ot2}]
    $\langle X,Y^\ast \rangle \in N_F$ for all $X,Y \in \mathcal{C}$ with $|\underline{X} \cap \underline{Y}^*| = 2$.
\end{enumerate}
Recall that by Lemma~\ref{lem:F-circuit unique up to}, for each circuit $C$ of $\ul{M}$, the $F$-circuit $X\in \cC$ \textup{(}and equivalently, $\conj{X}$\textup{)} with $\underline{X} = C$ is unique up to multiplication by an element in $F^\times$.



We first set $\gamma(B,B) = 1$ for every basis $B$ of $\ul{M}$. Let $B_1, B_2$ be two bases of $\ul{M}$ with $|B_1 \symdiff B_2| = 4$. We can write $B_1 = T \symdiff \skewpair{f}$ and $B_2 = T \symdiff \skewpair{e}$ for some transversal $T$ containing  $e$ and $f$. 
Let $X \in \mathcal{C}$ be the $F$-circuit whose support $\underline{X}$ is the fundamental circuit $C(B_1, f)$.
Then $\underline{X} = C(B_2, e) \subseteq T$, and in particular, $e,f \in \underline{X}$.
We define 
  \begin{align*}
    \gamma(B_1,B_2) := (-1)^{m^T_{e,f}} \frac{\conj{X}(e)}{\conj{X}(f)}.
  \end{align*}

\begin{proposition}
    $\gamma(B_1,B_2)$ is well-defined. 
\end{proposition}

\begin{proof}
By Lemma~\ref{lem:F-circuit unique up to}, $\gamma(B_1,B_2)$ is independent of the choice of $X$ for fixed $T$.
Let $T_1 = T \symdiff \skewpair{e} \symdiff \skewpair{f} \ni e^*, f^*$.
Let $X_1\in \mathcal{C}$ be such that $\underline{X}_1 = C(B_1,e^*) = C(B_2,f^*) \ni e^*, f^*$.
It suffices to show that
\begin{align*}
  (-1)^{m^{T}_{e,f}} \frac{\conj{X}(e)}{\conj{X}(f)} = (-1)^{m^{T_1}_{e,f}} \frac{\conj{X}_1(f^*)}{\conj{X}_1(e^*)}.
\end{align*}
Since $(T \symdiff T_1) \cap \ocinterval{\nostar{e}}{\nostar{f}} = \max\{\nostar{e},\nostar{f}\}$, we have $|m^T_{e,f}-m^{T_1}_{e,f}| = 1$.
By~\ref{item:Ot2}, $\conj{X}(e)\conj{X}_1(e^*) + \conj{X}(f)\conj{X}_1(f^\ast) = \langle X, X_1^\ast \rangle \in N_F$ and therefore we obtain the desired equality.
\end{proof}

The next lemma is obvious from the definition.
\begin{lemma}\label{lem:2cycle}
If $B_1,B_2$ are bases of $\underline{M}$ with $|B_1 \symdiff B_2| = 4$, then we have $\gamma(B_1,B_2) = \gamma(B_2,B_1)^{-1}$. \qed
\end{lemma}

Now we define a candidate for a Wick function on $E$ with coefficients in $F$ whose underlying matroid is exactly $\ul{M}$. Fix a basis $B_0$ of $\ul{M}$, and let $\varphi_\mathcal{C} : \mathcal{T}_n \to F$ be such that:
\begin{enumerate}[label=\rm(\roman*)]
    \item $\varphi_\mathcal{C}(B_0) = 1$ $(\not\in N_F)$. 
    \item For each basis $B$ of $\ul{M}$ other than $B_0$, we set 
    \begin{align*}
      \varphi_\mathcal{C}(B) := \gamma(B', B) \varphi_\mathcal{C}(B'), 
    \end{align*}  
    where $B'$ is a basis of $\ul{M}$ such that $|B\setminus B'| = 2$ and $|B \setminus B_0| = |B' \setminus B_0| + 2$.
    \item For each non-basis transversal $T$, we set $\varphi(T) = 0$.
\end{enumerate}

To show that $\varphi_{\mathcal{C}}$ is well-defined, we need a result on the basis graphs of the orthogonal matroids. 

The {\em basis graph} $\Gamma_{\ul{N}}$ of an ordinary orthogonal matroid $\ul{N}$ is a graph whose vertex set is the set of bases of $\ul{N}$, and two vertices $B_1$ and $B_2$ are adjacent if and only if $|B_1 \setminus B_2| = 2$. For every graph~$G$, a {\em directed cycle $C$ of length $\ell \geq 2$} is a sequence $(v_0,v_1)$, $(v_1,v_2)$, $\cdots$, $(v_{\ell-1},v_{\ell})$ of ordered pairs of adjacent vertices in $G$ such that all $v_k$ are distinct except for $v_0 = v_\ell$. We simply write $C$ as a sequence $v_0,v_1,\ldots,v_{\ell-1},v_0$ of vertices. We denote by $-C$ the directed cycle $v_0,v_{\ell-1},\dots,v_1,v_0$. For directed cycles $C_0,C_1,\dots,C_m$ of $G$, we say $C_0$ is {\em generated by $C_1,\dots,C_m$} if for all vertices $u,v$ in~$G$, two ordered pairs $(u,v)$ and $(v,u)$ appear the same number of times in $-C_0$, $C_1$, $\cdots$, $C_m$; see Figure~\ref{fig:generated directed cycles}.

\begin{figure*}
  \centering
  \begin{tikzpicture}
    \node[shape=circle,fill=black, scale=0.25] (v1) at (0.1,0) {};
    \node[shape=circle,fill=black, scale=0.25] (v2) at (-0.1,0.8) {};
    \node[shape=circle,fill=black, scale=0.25] (v3) at (0.7,1.4) {};
    \node[shape=circle,fill=black, scale=0.25] (v4) at (1.6,1.5) {};
    \node[shape=circle,fill=black, scale=0.25] (v5) at (2.4,1) {};
    \node[shape=circle,fill=black, scale=0.25] (v6) at (2.5,0.2) {};
    \node[shape=circle,fill=black, scale=0.25] (v7) at (3.1,-0.3) {};
    \node[shape=circle,fill=black, scale=0.25] (v8) at (3.2,-1) {};
    \node[shape=circle,fill=black, scale=0.25] (v9) at (2.6,-1.5) {};
    \node[shape=circle,fill=black, scale=0.25] (v10) at (1.8,-1.4) {};
    \node[shape=circle,fill=black, scale=0.25] (v11) at (1.15,-0.85) {};
    \node[shape=circle,fill=black, scale=0.25] (v12) at (0.45,-1.4) {};
    \node[shape=circle,fill=black, scale=0.25] (v13) at (-0.5,-1.6) {};
    \node[shape=circle,fill=black, scale=0.25] (v14) at (-1,-1) {};
    \node[shape=circle,fill=black, scale=0.25] (v15) at (-0.6,-0.4) {};

    \node[shape=circle,fill=black, scale=0.25] (w1) at (0.8,-0.3) {};
    \node[shape=circle,fill=black, scale=0.25] (w2) at (1.65,-0.2) {};

    \node (o3) at (0.4,2.05) {};
    \node (o5) at (3.2,1.3) {};
    \node (o8) at (3.8,-1.35) {};
    \node (o12) at (0.65,-2) {};
    \node (o15) at (-1.25,-0.1) {};

    \draw (v1) -- (v2) -- (v3) -- (v4) -- (v5) -- (v6) -- (v7) -- (v8) -- (v9) -- (v10) -- (v11) -- (v12) -- (v13) -- (v14) -- (v15) -- (v1);
    \draw (v1) -- (w1) -- (w2) -- (v6);
    \draw (w1) -- (v11);
    \draw (v3) -- (o3);
    \draw (v5) -- (o5);
    \draw (v8) -- (o8);
    \draw (v12) -- (o12);
    \draw (v15) -- (o15);

    \draw[color=red, -Stealth] (0.2,0.2) to[out=110,in=185] (1.2,1.3) to[out=5,in=95] (2.3,0.6) to[out=275,in=10] (1.1,-0.07) to[out=190,in=-20] (0.3,0.1);
    \draw[color=red] (0.42,0.4) node {\tiny{$C_1$}};

    \draw[color=red, -Stealth] (0.4,-1.25) to[out=195,in=-75] (-0.75,-1) to[out=105,in=210] (-0.1,-0.3) to[out=30,in=150] (0.7,-0.45) to[out=-30,in=30] (0.45,-1.2);
    \draw[color=red] (0.2,-1.08) node {\tiny{$C_2$}};

    \draw[color=red, -Stealth] (1.85,-1.25) to[out=135,in=-95] (1.2,-0.5) to[out=85,in=190] (1.8,-0.4) to[out=10,in=210] (2.35,-0.08) to[out=30,in=95] (3,-0.8) to[out=-85,in=-10] (1.95,-1.3);
    \draw[color=red] (1.88,-0.99) node {\tiny{$C_3$}};

    \draw[color=red, -Stealth] (-0.65,-1.7) to[out=150,in=225] (-0.95,-0.5) to[out=45,in=-135] (-0.2,0.1) to[out=45,in=-100] (-0.25,0.75) to[out=80,in=200] (0.9,1.6) to[out=20,in=100] (2.65,0.6) to[out=-80,in=130] (3.15,-0.15) to[out=-50,in=40] (3.08,-1.38) to[out=-140,in=-10] (2.2,-1.6) to[out=170,in=-50] (1.3,-1.17) to[out=130,in=40] (0.9,-1.25) to[out=-140,in=0] (-0.55,-1.73);
    \draw[color=red] (-1.05,-1.65) node {\tiny{$C_0$}};

  \end{tikzpicture}
  \caption{$C_0$ is generated by $C_1, C_2, C_3$.}
  \label{fig:generated directed cycles}
\end{figure*}

The following theorem generalizes Maurer's Homotopy Theorem for matroids~\cite{Maurer1973a}. 

\begin{theorem}[Wenzel, Theorem~5.7~of~\cite{Wenzel1995}]\label{thm:homotopy}
   Let $\ul{N}$ be an orthogonal matroid.
   Then every directed cycle in the basis graph $\Gamma_{\ul{N}}$ is generated by directed cycles of length at most $4$.
\end{theorem}

\begin{lemma}
  \label{lem:34cycle}
  The following hold for the basis graph $\Gamma_{\ul{M}}$ of an orthogonal matroid $\ul{M}$ with an $F$-signature $\mathcal{C}$ satisfying~\ref{item:Ot2}. 
  \begin{enumerate}[label=\rm(\roman*)]
      \item If $B_1, B_2, B_3, B_1$ is a directed cycle of length $3$ in $\Gamma_{M}$, then 
      \[
        \gamma(B_1,B_2) \gamma(B_2,B_3) \gamma(B_3,B_1) = 1.
      \]
      \item If $B_1, B_2, B_3, B_4, B_1$ is a directed cycle of length $4$ in $\Gamma_{M}$, then
      \[
        \gamma(B_1,B_2) \gamma(B_2,B_3) \gamma(B_3,B_4) \gamma(B_4,B_1) = 1.
      \]
  \end{enumerate}
\end{lemma}

\begin{proof}
  (i) If $B_1,B_2,B_3$ are bases of $\ul{M}$ with $|B_i\setminus B_j| = 2$ for all distinct $i,j \in [3]$, then there exist a transversal $T$ and distinct elements $e_1^*,e_2^*,e_3^* \in T$ such that $B_i = T \symdiff \skewpair{e_i} \ni e_i$ for each $i\in[3]$.
  For distinct $i,j \in [3]$, let $T_{ij} := T \symdiff \skewpair{e_i} \symdiff \skewpair{e_j}$ and $X_{ij} \in \mathcal{C}$ be the $F$-circuit with $\underline{X_{ij}} = C(B_i, e_j) = C(B_j, e_i) \subseteq T_{ij}$.
  Then
  \begin{align*}
    \gamma(B_i,B_j) = (-1)^{m_{ij}} \frac{\conj{X}_{ij}(e_i)}{\conj{X}_{ij}(e_j)}, 
  \end{align*}
  where $m_{ij} := |T_{ij} \cap \ocinterval{\nostar{e_i}}{\nostar{e_j}}|$.
  Let $Y$ be the $F$-circuit with $\underline{Y} = C(B_1, e_1^*) \subseteq T$.
  Then $\underline{Y} = C(B_2, e_2^*)= C(B_3, e_3^*)$ and thus $\{e_1^*,e_2^*,e_3^*\} \subseteq \underline{Y}$.
  By~\ref{item:Ot2}, if $i \ne j \in [3]$, we have that $\conj{X}_{ij}(e_i)\conj{Y}(e_i^*) + \conj{X}_{ij}(e_j)\conj{Y}(e_j^*) = \langle X_{ij}, Y^\ast \rangle \in N_F$.
  Thus,
  \begin{align*}
    \frac{\conj{X}_{12}(e_1)}{\conj{X}_{12}(e_2)} \frac{\conj{X}_{23}(e_2)}{\conj{X}_{23}(e_3)} \frac{\conj{X}_{31}(e_3)}{\conj{X}_{31}(e_1)}
    =
    \left( - \frac{\conj{Y}(e_2^*)}{\conj{Y}(e_1^*)} \right) \left( - \frac{\conj{Y}(e_3^*)}{\conj{Y}(e_2^*)} \right) \left( - \frac{\conj{Y}(e_1^*)}{\conj{Y}(e_3^*)} \right)
    = -1.
  \end{align*}
  By relabelling, we may assume that $\nostar{e_1}<\nostar{e_2}<\nostar{e_3}$.
  Then
    $(T_{12} \cap \ocinterval{\nostar{e_1}}{\nostar{e_2}}) \symdiff (T_{23} \cap \ocinterval{\nostar{e_2}}{\nostar{e_3}}) \symdiff (T_{13} \cap \ocinterval{\nostar{e_1}}{\nostar{e_3}}) = \{\nostar{e_{2}}\}$. 
  Hence $m_{12} + m_{23} + m_{13}$ is odd and
  therefore $\gamma(B_1,B_2) \gamma(B_2,B_3) \gamma(B_3,B_1) = 1$.

  (ii)
  By (i) and Lemma~\ref{lem:2cycle}, we may assume that the directed cycle $B_1, B_2, B_3, B_4, B_1$ is not generated by directed cycles of length $3$.
  Then $|B_i\setminus B_{i+1}| = 2$ and $|B_i \setminus B_{i+2}| = 4$ for all $i\in[4]$, where all subscripts are read modulo~$4$.
  Thus, there exist a transversal $T$ and distinct elements $e_1,e_2,e_3,e_4 \in T$ such that $\text{$B_1 = T_{12}$, $B_2 = T_{13}$, $B_3 = T_{34}$, and $B_4 = T_{24}$}$, 
  where $T_{I} = T \symdiff \bigcup_{i\in I }\skewpair{e_i}$ for all $I \subseteq [4]$.
  In addition, none of $T$, $T_{14}$, $T_{23}$, and $T_{1234}$ is a basis. 
  
  Let $X_1, X_3, Y_3, Y_1 \in \mathcal{C}$ be $F$-circuits such that $\underline{X_i} \subseteq T_i$ and $\underline{Y_i} \subseteq T_{2j4}$ for each $\{i,j\} = \{1,3\}$.
  Then
  \begin{align*}
    \gamma(T_{12},T_{13}) = (-1)^{m_1} \frac{\conj{X}_1(e_3)}{\conj{X}_1(e_2)}, \\
    \gamma(T_{13},T_{34}) = (-1)^{m_3} \frac{\conj{X}_3(e_4)}{\conj{X}_3(e_1)}, \\
    \gamma(T_{12},T_{24}) = (-1)^{n_3} \frac{\conj{Y}_3(e_1^*)}{\conj{Y}_3(e_4^*)}, \\
    \gamma(T_{24},T_{34}) = (-1)^{n_1} \frac{\conj{Y}_1(e_2^*)}{\conj{Y}_1(e_3^*)},
  \end{align*}
  where $m_1 := |T_1 \cap \ocinterval{\nostar{e_2}}{\nostar{e_3}}|$, $m_3 := |T_3 \cap \ocinterval{\nostar{e_1}}{\nostar{e_4}}|$, $n_3 := |T_{124} \cap \ocinterval{\nostar{e_1}}{\nostar{e_4}}|$, and $n_1 := |T_{234} \cap \ocinterval{\nostar{e_2}}{\nostar{e_3}}|$.
  Note that $m_1+m_3+n_3+n_1$ is even, since $(T_1 \cap \ocinterval{\nostar{e_2}}{\nostar{e_3}}) \symdiff (T_{234} \cap \ocinterval{ \nostar{e_2}}{\nostar{e_3}})
    =\{ \nostar{e_1}, \nostar{e_2}, \nostar{e_3}, \nostar{e_4}\} \cap \ocinterval{\nostar{e_2}}{\nostar{e_3}}$,
and $(T_3 \cap \ocinterval{ \nostar{e_1}}{\nostar{e_4}}) \symdiff (T_{124} \cap \ocinterval{ \nostar{e_1}}{\nostar{e_4}}) = \{ \nostar{e_1},\nostar{e_2},\nostar{e_3},\nostar{e_4}\} \cap \ocinterval{\nostar{e_1}}{\nostar{e_4}}$. 
  
  Suppose for contradiction that $X_1(e_1^*) \neq 0$, i.e., $e_1^* \in \underline{X_1}$.
  Notice that $\underline{X_1} = C(B_1,e_2)$ is the unique circuit of $\ul{M}$ contained in $T_1$, and the subtransversal $T_1 \setminus \{e_1^*\}$ is independent.
  Since $T_1$ is not a basis, $T = (T_1 \setminus \{e_1^*\}) \cup \{e_1\}$ is a basis, a contradiction. 
  Thus, $X_1(e_1^*) = 0$.
  Similarly, one can check that all of $X_1(e_4)$, $X_3(e_2)$, $X_3(e_3^*)$, $Y_3(e_2^*)$, $Y_3(e_3)$, $Y_1(e_1)$, and $Y_1(e_4^*)$ are zero, because none of $T$, $T_{14}$, $T_{23}$, and $T_{1234}$ is a basis of $\ul{M}$.
  Therefore, by~\ref{item:Ot2}, we have
  \begin{align*}
    \conj{X}_1(e_2)\conj{Y}_1(e_2^*) + \conj{X}_1(e_3)\conj{Y}_1(e_3^*) = \langle X_1, Y_1^\ast \rangle \in N_F, 
  \end{align*}
  and
  \begin{align*}
    \conj{X}_3(e_1)\conj{X}_3(e_1^*) + \conj{X}_3(e_4)\conj{Y}_3(e_4^*) = \langle X_3, Y_3^\ast \rangle \in N_F.
  \end{align*}
  Therefore, 
  \begin{align*}
    \gamma(T_{12},T_{13}) \gamma(T_{13},T_{34})
    =
    (-1)^{m_1+m_3}
    \frac{\conj{X}_1(e_3)}{\conj{X}_1(e_2)}
    \frac{\conj{X}_3(e_4)}{\conj{X}_3(e_1)} 
    =
    (-1)^{n_1+n_3}
    \frac{\conj{Y}_1(e_2^*)}{\conj{Y}_1(e_3^*)}
    \frac{\conj{Y}_3(e_1^*)}{\conj{Y}_3(e_4^*)}
    =
    \gamma(T_{24},T_{34}) \gamma(T_{12},T_{24}).
  \end{align*}
  By Lemma~\ref{lem:2cycle}, we obtain that 
  \begin{align*}
    &\gamma(B_1,B_2) \gamma(B_2,B_3) \gamma(B_3,B_4) \gamma(B_4,B_1) \\
    &=
    \gamma(T_{12},T_{13}) \gamma(T_{13},T_{34}) \gamma(T_{24},T_{34})^{-1}\gamma(T_{12},T_{24})^{-1}
    = 1. \qedhere
  \end{align*}
\end{proof}

\begin{corollary}
    $\varphi_\mathcal{C}$ is well-defined.
\end{corollary}

\begin{proof}
    It suffices to show that for arbitrary paths $P = B_0 B_1 \dots B_k$ and $P' = B_0' B_1' \dots B_\ell'$ in $\Gamma_{\ul{M}}$, if $B_0 = B_0'$ and $B_k = B_\ell'$, then
    \begin{align*}
        \prod_{i=0}^{k-1} \gamma(B_{i},B_{i+1})
        =
        \prod_{j=0}^{\ell-1} \gamma(B_{j},B_{j+1}).
    \end{align*}
    This is straightforward from Lemmas~\ref{lem:2cycle},~\ref{lem:34cycle}, and Theorem~\ref{thm:homotopy}.
\end{proof}

\begin{theorem}\label{thm:O2W}
If $\mathcal{C}$ satisfies the orthogonality~\ref{item:O}, then $\varphi_\mathcal{C}$ is a strong Wick function on $E$ with coefficients in $F$.
\end{theorem}

\begin{proof}
  We only need to prove~\ref{item:W2}.
  Take $T_1,T_2 \in \mathcal{T}_n$ with $T_1 \cap T_2^* = \{e_1,\dots,e_a\}$, where $\nostar{e_1}<\dots<\nostar{e_a}$.
  If $T_1$ is a basis of $\ul{M}$, then $\varphi(T_1 \symdiff \skewpair{i}) = 0$ for all $i\in [n]$ and thus $\sum_{i=1}^a (-1)^i \varphi(T_1 \symdiff \skewpair{e_i}) \varphi(T_2 \symdiff \skewpair{e_i}) \in N_F$.
  Therefore, we may assume that $T_1$ is not a basis, and similarly we may assume that $T_2$ is not a basis.
  Then there exist $X_1, X_2 \in \mathcal{C}$ such that $\underline{X_i} \subseteq T_i$ for $i = 1, 2$.
  Write $\underline{X_1} \cap \underline{X_2}^* = \{e_{\alpha_1}, \dots, e_{\alpha_b}\}$ with $\alpha_1 < \dots < \alpha_b$.
  For $i \in [a] \setminus \{\alpha_1,\dots,\alpha_b\}$, at least one of $T_1 \symdiff \skewpair{e_i}$ and $T_2 \symdiff \skewpair{e_i}$ is not a basis.
  Hence
  \begin{align*}
    \sum_{i=1}^a (-1)^i \varphi(T_1 \symdiff \skewpair{e_i}) \varphi(T_2 \symdiff \skewpair{e_i}) = \sum_{i=1}^b (-1)^{\alpha_i} \varphi(T_1 \symdiff \skewpair{e_{\alpha_i}}) \varphi(T_2 \symdiff \skewpair{e_{\alpha_i}}).
  \end{align*}
  Therefore, we may assume that $b \geq 1$.
  We can also assume that there exists $c\in [b]$ such that both $B_1 := T_1 \symdiff \skewpair{e_{\alpha_c}}$ and $B_2 := T_2 \symdiff \skewpair{e_{\alpha_c}}$ are bases.
  Then $\underline{X_1} = C(T_1 \symdiff \skewpair{e_{\alpha_c}}, e_{\alpha_c})$ and $\underline{X_2} = C(T_2 \symdiff \skewpair{e_{\alpha_c}}, e_{\alpha_c}^*)$, and therefore $T_j \symdiff \skewpair{e_{\alpha_i}}$ is a basis for each $i\in[b]$ and $j = 1, 2$.
  
  For each $i\in[b]$, let $m_i := |T_1 \cap \ocinterval{\nostar{e_{\alpha_c}}}{\nostar{e_{\alpha_i}}}|$ and $n_i := |T_2 \cap \ocinterval{\nostar{e_{\alpha_c}}}{\nostar{e_{\alpha_i}}}|$.
  By the definition of $\varphi_\mathcal{C}$, we have 
  \begin{align*}
      \frac{\conj{X}_1(e_{\alpha_i})}{\conj{X}_1(e_{\alpha_c})}
      =
      (-1)^{m_i} \frac{\varphi(T_1\symdiff \skewpair{e_{\alpha_i}})}{\varphi(T_1\symdiff \skewpair{e_{\alpha_c}})}
      \quad\text{and}\quad
      \frac{\conj{X}_2(e_{\alpha_i}^*)}{\conj{X}_2(e_{\alpha_c}^*)}
      =
      (-1)^{n_i} \frac{\varphi(T_2\symdiff \skewpair{e_{\alpha_i}})}{\varphi(T_2\symdiff \skewpair{e_{\alpha_c}})}.
  \end{align*}
  Since $(T_1 \symdiff T_2) \cap \ocinterval{\nostar{e_{\alpha_c}}}{\nostar{e_{\alpha_i}}}$ equals $\{\nostar{e_{k}} : \alpha_c < k \leq \alpha_i\}$ if $c<i$ and $\{\nostar{e_{k}} : \alpha_i < k \leq \alpha_c\}$ otherwise, we have $m_i + n_i \equiv \alpha_i - \alpha_c \pmod{2}$.
  By the orthogonality relation~\ref{item:O}, we have 
  \begin{align*}
    \sum_{i=1}^b \conj{X}_1(e_{\alpha_i}) \conj{X}_2(e_{\alpha_i}^*) = \langle X_1, X_2^\ast \rangle \in N_F.
  \end{align*}
  Therefore,
  \begin{align*}
    &\sum_{i=1}^b (-1)^{\alpha_i} \varphi(T_1 \symdiff \skewpair{e_{\alpha_i}}) \varphi(T_2 \symdiff \skewpair{e_{\alpha_i}})\\
    &=
    (-1)^{\alpha_c} \sum_{i=1}^b (-1)^{m_i+n_i} \varphi(T_1 \symdiff \skewpair{e_{\alpha_i}}) \varphi(T_1 \symdiff \skewpair{e_{\alpha_i}}) \\
    &=
    (-1)^{\alpha_c}
    \frac{\varphi(T_1 \symdiff \skewpair{e_{\alpha_c}})}{\conj{X}_1(e_{\alpha_c})} \frac{\varphi(T_2 \symdiff \skewpair{e_{\alpha_c}})}{\conj{X}_2(e_{\alpha_c}^*)}
    \sum_{i=1}^b \conj{X}_1(e_{\alpha_i}) \conj{X}_2(e_{\alpha_i}^*)
    \in N_F. \qedhere
  \end{align*} 
\end{proof}

\begin{theorem}\label{thm:O2W-weak}
  If $\cC$ satisfies~\ref{item:O'}, then $\varphi_\cC$ is a moderately weak Wick function on $E$ with coefficients in $F$.
\end{theorem}
\begin{proof}
  It yields if we replace~\ref{item:O} with~\ref{item:O'} in the proof of Theorem~\ref{thm:O2W}.
\end{proof}

\subsection{Weak Wick functions and weak circuit sets}\label{sec:weak Wick functions and circuit sets}

In this section, we prove the equivalence between the weak Wick functions and the weak circuit sets, using the constructions in Sections~\ref{sec:W2O} and~\ref{sec:O2W}.


\begin{theorem}\label{prop:weak circuit set to weak Wick function}
  Let $\cC$ be a weak $F$-circuit set of an orthogonal matroid. Then $\varphi_\cC$ is a weak Wick $F$-function.
\end{theorem}

\begin{proof}
  Denote $\varphi:= \varphi_\cC$.
  Let $T_1$ be a transversal, and let $e_1,e_2,e_3,e_4 \in T_1$ be such that $\nostar{e_1} < \nostar{e_2} < \nostar{e_3} < \nostar{e_4}$.
  Let $T_2$ be another transversal such that $T_2 \setminus T_1 = \{e_1^*,e_2^*,e_3^*,e_4^*\}$.
  We may assume that neither $T_1$ nor $T_2$ is a basis, and there is $k \in [4]$ such that both $T_1 \symdiff \skewpair{e_k}$ and $T_2 \symdiff \skewpair{e_k}$ are bases of $\ul{M}$.
  Let $X$ and $Y$ be $F$-circuits in $\cC$ such that $\underline{X} \subseteq T_1$ and $\underline{Y} \subseteq T_2$.
  Then for each $i \in [4]$, we have 
  \begin{align*}
    \frac{\varphi(T_1 \symdiff \skewpair{e_i})}{\varphi(T_1 \symdiff \skewpair{e_k})}
    =(-1)^{m_{i}}\frac{\tilde{X}(e_i)}{\tilde{X}(e_k)}
    \quad\text{and}\quad
    \frac{\varphi(T_2 \symdiff \skewpair{e_i})}{\varphi(T_2 \symdiff \skewpair{e_k})}
    =(-1)^{n_{i}}\frac{\tilde{Y}(e_i^*)}{\tilde{Y}(e_k^*)}, 
  \end{align*}
  where $m_{i} := m^{T_1}_{e_i,e_k} = |T_1 \cap \ocinterval{ \nostar{e_i}}{\nostar{e_k}}|$ and $n_{i} := m^{T_2}_{e_i,e_k} = |T_2 \cap \ocinterval{\nostar{e_i}}{\nostar{e_k}}|$.
  Note that $m_i+n_i \equiv k-i \pmod{2}$.
  Hence, we have 
  \begin{align*}
    \sum_{i\in[4]} (-1)^{i+k} 
    \frac{\varphi(T_1\symdiff\skewpair{e_i})\varphi(T_2\symdiff\skewpair{e_i})}{\varphi(T_1\symdiff\skewpair{e_k})\varphi(T_2\symdiff\skewpair{e_k})}
    =
    1+
    \sum_{i\in[4]\setminus\{k\}}
    \frac{\tilde{X}(e_i) \tilde{Y}(e_i^*)}{\tilde{X}(e_k) \tilde{Y}(e_k^*)}. 
    \tag{$*$}\label{eq:weak circuit to weak Wick}
  \end{align*}
  By the strong symmetric exchange axiom, at least one of $\theta_{i} := \varphi(T_1\symdiff\skewpair{e_i})\varphi(T_2\symdiff\skewpair{e_i})$ with $i\in[4]\setminus\{k\}$ is nonzero.
  If exactly one of $\theta_i$ is nonzero, then~\eqref{eq:weak circuit to weak Wick} is
  $1+\frac{\tilde{X}(e_i) \tilde{Y}(e_i^*)}{\tilde{X}(e_k) \tilde{Y}(e_k^*)} \in N_F$
  by~\ref{item:Ot2}.
  Therefore we may assume that at least two of $\theta_i$ are nonzero.
  We denote by $a$, $b$, $c$ the distinct elements of $[4]\setminus\{k\}$.

  Suppose that $\theta_a$ and $\theta_b$ are nonzero but $\theta_c = 0$.
  Then $\{e_a,e_b,e_k\}\subseteq \underline{X}$ and $\{e_a^*,e_b^*,e_k^*\}\subseteq \underline{Y}$.
  By interchanging roles of $T_1$ and $T_2$ if necessary, we may assume that $T_2 \symdiff \skewpair{e_c}$ is not a basis of $\ul{M}$.
  Then $e_c^* \notin \underline{Y}$.
  Because $\{e_a,e_b,e_k\}\subseteq \underline{X} \subseteq T_1$, neither $T_1 \symdiff\{e_a,e_a^*,e_k,e_k^*\}$ nor $T_1 \symdiff\{e_b,e_b^*,e_k,e_k^*\}$ is a basis. Hence $\cC$ has $F$-circuits $Z_a$ and $Z_b$ such that $\underline{Z_i} \subseteq T_1 \symdiff\{e_i,e_i^*,e_k,e_k^*\}$ for each $i \in \{a,b\}$.
  Because $T_1 \symdiff \skewpair{e_a}$, $T_1 \symdiff \skewpair{e_k}$, and $T_1 \symdiff \skewpair{e_b}$ are bases, we have $\{e_a^*,e_c,e_k^*\} \subseteq \underline{Z_a}$.
  Because $T_2 \symdiff \skewpair{e_c}$ is not a basis, $e_b \notin \underline{Z_a}$.
  Similarly, $\{e_b^*,e_c,e_k^*\} \subseteq \underline{Z_b}$ and $e_a \notin \underline{Z_b}$.
  Then $\underline{Z_a} \cup \underline{Z_b}$ is admissible, and
  by the circuit elimination axiom~\ref{item:C3}, $\ul{M}$ has a circuit $C$ contained in $(\underline{Z_a} \cup \underline{Z_b}) \setminus \{e_c\} \subseteq T_2 \setminus \{e_c\}$.
  Then $\underline{Y} = C$ and hence $\tilde{Y}$ is in the linear span of $\tilde{Z_a}$ and $\tilde{Z_b}$ by~\ref{item:L1}.
  Rescaling $Z_a$ and $Z_b$ if necessary, we may assume that $Z_a(e_a^*) = Y(e_a^*)$ and $Z_b(e_b^*) = Y(e_b^*)$. 
  Then $\tilde{Y}(e_k^*) - \tilde{Z_a}(e_k^*) - \tilde{Z_b}(e_k^*) \in N_F$.
  By~\ref{item:Ot2}, $\frac{\tilde{Z_i}(e_k^*)}{\tilde{Z_i}(e_{i}^*)} = -\frac{\tilde{X}(e_{i})}{\tilde{X}(e_k)}$ for $i\in\{a,b\}$ and thus~\eqref{eq:weak circuit to weak Wick} is equal to $1 - \frac{\tilde{Z_a}(e_k^*)}{\tilde{Y}(e_k^*)} - \frac{\tilde{Z_b}(e_k^*)}{\tilde{Y}(e_k^*)} \in N_F$.

  Now we consider the case where $\theta_a$, $\theta_b$, $\theta_c$ are all nonzero.
  Then there are $F$-circuits $Z_a$, $Z_b$, $Z_c$ in $\cC$ such that $\underline{Z_i} \subseteq T_1 \symdiff \{e_i,e_i^*,e_k,e_k^*\}$ with $i\in\{a,b,c\}$.
  It can be easily checked that $\{e_a,e_b,e_c,e_k^*\} \symdiff\skewpair{e_i} \subseteq \underline{Z_i}$ for every $i \in \{a,b,c\}$.
  Then by~\ref{item:L2}, $\tilde{Y}$ is in the linear span of $\tilde{Z_1}$, $\tilde{Z_2}$, and $\tilde{Z_3}$.
  Rescaling $Z_i$ if necessary, we may assume that $Z_i(e_i^*) = Y(e_i^*)$ for each $i\in\{a,b,c\}$. 
  Then $\tilde{Y}(e_k^*) - \tilde{Z_a}(e_k^*) - \tilde{Z_b}(e_k^*) - \tilde{Z_c}(e_k^*) \in N_F$.
  By~\ref{item:Ot2}, $\frac{\tilde{Z_i}(e_k^*)}{\tilde{Z_i}(e_{i}^*)} = -\frac{\tilde{X}(e_{i})}{\tilde{X}(e_k)}$ with $i=a,b,c$ and therefore~\eqref{eq:weak circuit to weak Wick} is equal to $1 - \frac{\tilde{Z_a}(e_k^*)}{\tilde{Y}(e_k^*)} - \frac{\tilde{Z_b}(e_k^*)}{\tilde{Y}(e_k^*)} - \frac{\tilde{Z_c}(e_k^*)}{\tilde{Y}(e_k^*)} \in N_F$. 
\end{proof}

To prove the converse of Theorem~\ref{prop:weak circuit set to weak Wick function}, we consider the 
following weaker replacement of orthogonality~\ref{item:Ot2}:
\begin{enumerate}[label=\rm(O$_2$)$'$]
  \item\label{item:Ot2'} Let $X,Y \in \cC$ be such that $\underline{X}$ and $\underline{Y}$ are fundamental circuits with respect to the same basis of $\ul{M}_\cC$, then $\langle X, Y^* \rangle \in N_F$.
\end{enumerate}

\begin{lemma}\label{lem:Ot2'}
  Let $\cC$ be an $F$-signature of an orthogonal matroid.
  If $\cC$ satisfies~\ref{item:Ot2'} and~\ref{item:L1}, then it satisfies~\ref{item:Ot2}.
\end{lemma}

\begin{proof}
  Suppose for contradiction that~\ref{item:Ot2} does not hold.
  Let $X$ and $Y$ be $F$-circuits in $\cC$ such that $|\underline{X}\cup \underline{Y}|$ is minimized subject to $|\underline{X} \cap \underline{Y}^*| = 2$ and $\langle X, Y^* \rangle \notin N_F$.
  Write $\underline{X} \cap \underline{Y}^* = \{e,f\}$.
  Then $J := (\underline{X} \cup \underline{Y}) \setminus \{e^*,f\}$ is dependent in $\ul{M}_\cC$, because otherwise there is a basis $B \supseteq J$ such that $X$ and $Y$ are fundamental circuits with respect to $B$ and thus $\langle X, Y^* \rangle \in N_F$ by~\ref{item:Ot2'}, a contradiction.
  Let $C$ be a circuit contained in $J$ which minimizes $|\underline{X} \cup C|$.
  Note that $C \cap \{e,f^*\} = \emptyset$ by~\ref{item:C4}, and there are $x \in C \cap (\underline{X} \setminus \underline{Y})$ and $y \in C \cap (\underline{Y} \setminus \underline{X})$ by~\ref{item:C2}.
  Because of the minimality of $|\underline{X} \cup C|$, we deduce that $J_2 := (\ul{X} \symdiff\skewpair{f}) \cup (C \setminus \{y\})$ is independent. Let $B_2$ be a basis containing $J_2$.
  Then $\underline{X}$ and $C$ are fundamental circuits with respect to $B_2$.
  Let $Z$ be an $F$-circuit whose support is $C$.
  By~\ref{item:L1}, there is an $F$-circuit $X_2$ such that $X_2(x) = 0$ and $\tilde{X_2}$ is in the linear span of $\tilde{X}$ and $\tilde{Z}$.
  Then $\underline{X_2} \cup \underline{Y} \subsetneq \underline{X} \cup \underline{Y}$, and for some $\alpha\in F^\times$, we have $X_2(e) =\alpha X(e)$ and $X_2(f^*) = \alpha X(f^*)$.
  Therefore, $\alpha \langle X,Y^* \rangle = \langle X_2,Y^* \rangle \in N_F$, a contradiction.
\end{proof}

We note that by Lemma~\ref{lem:Ot2'}, an $F$-signature of an orthogonal matroid is a strong $F$-circuit set if and only if it satisfies~\ref{item:L} and~\ref{item:Ot2}.
In addition, \ref{item:Ot2} in Lemma~\ref{lem:F-circuit unique up to} can be replaced by~\ref{item:Ot2'}.

\begin{theorem}\label{prop:weak Wick function to weak circuit set}
  Let $\varphi$ be a weak Wick function. Then $\cC_\varphi$ is a weak $F$-circuit set of $\ul{M}_\varphi$.
\end{theorem}

\begin{proof}
  By Lemma~\ref{lem:Ot2'}, it suffices to show that $\cC_\varphi$ satisfies~\ref{item:Ot2'},~\ref{item:L1}, and~\ref{item:L2}.

  Let $X$ and $Y$ be $F$-circuits in $\cC_\varphi$ such that $\underline{X} = C(B,f)$ and $\underline{Y} = C(B,e)$ for some basis $B$ and distinct elements $e,f\in B^*$.
  We denote $T_1 := B \symdiff \skewpair{f} \supseteq \underline{X}$ and $T_2 := B \symdiff \skewpair{e} \supseteq \underline{Y}$.
  Then
  \begin{align*}
    \frac{\tilde{X}(e)}{\tilde{X}(f)}
    = (-1)^{m^{T_1}_{e,f}} \frac{\varphi(T_1 \symdiff\skewpair{e})}{\varphi(T_1 \symdiff\skewpair{f})}
    = (-1)^{m^{T_2}_{e,f}} \frac{\varphi(T_2 \symdiff\skewpair{f})}{\varphi(T_2 \symdiff\skewpair{e})}
    = -\frac{\tilde{Y}(f^*)}{\tilde{Y}(e^*)}, 
  \end{align*}
  and hence $\langle X, Y^* \rangle \in N_F$.
  Therefore, $\cC_\varphi$ satisfies~\ref{item:Ot2'}.

  \smallskip
  Now we show that~\ref{item:L1} holds.
  Let $B$ be a basis of $\ul{M}_\varphi$ and $e_1,e_2\in B^*$ be distinct elements.
  Let $X_1$ and $X_2$ be $F$-circuits in $\cC$ such that $\underline{X_i} = C(B,e_i)$ for $i=1,2$.
  Suppose that $X_1(e_2^*) = X_2(e_1^*) = 0$ and there is an element $f \in \underline{X_1} \cap \underline{X_2}$.
  Let $Y$ be an $F$-circuit whose support is a subset of $(\underline{X_1} \cup \underline{X_2}) \setminus \{f\}$.
  We claim that $\tilde{Y}$ belongs to the linear span of $\tilde{X_1}$ and $\tilde{X_2}$.
  We may assume that $\tilde{X_i}(e_i) = \tilde{Y}(e_i)$.
  Thus it suffices to show that $\tilde{Y}(g) - \tilde{X_1}(g) - \tilde{X_2}(g) \in N_F$ for all $g \in (\underline{X_1} \cup \underline{X_2}) \setminus \{e_1,e_2\}$.

  Let $Z\in\cC$ be such that $\underline{Z} = C(B,f^*)$.
  By~\ref{item:Ot2'}, $\tilde{Z}(e_i^*)\tilde{X_i}(e_i) + \tilde{Z}(f^*) \tilde{X_i}(f) \in N_F$ with $i=1,2$.
  Again by~\ref{item:Ot2'}, $-Z(f^*)(X_1(f) + X_2(f)) = \tilde{Z}(e_1^*) \tilde{Y}(e_1) + \tilde{Z}(e_2^*) \tilde{Y}(e_2) \in N_F.$
  Hence $X_1(f) + X_2(f)\in N_F$.
  So we may assume that $g \ne f$.
  By symmetry, we may assume that $g \in \underline{X_1}$, implying that $B\symdiff\{e_1,e_1^*,g,g^*\}$ is a basis. Let $W\in\cC$ be such that $\underline{W} = C(B,g^*)$.
  Let $T_1 := B\symdiff\skewpair{g}$ and $T_2 := B \symdiff\{e_1,e_1^*,e_2,e_2^*,f,f^*\}$.
  Then $\underline{W} \subseteq T_1$ and $\underline{Y} \subseteq T_2$.
  Since
  $B\symdiff\{e_2,e_2^*,f,f^*\}$ and $B\symdiff\{e_1,e_1^*,g,g^*\}$ are bases of $\ul{M}_\varphi$, both $Y(e_1)$ and $W(e_1^*)$ are nonzero.
  We rewrite $\{e_1,e_2,f^*,g\} \subseteq T_2$ by $\{x_1,x_2,x_3,x_4\}$ with $\nostar{x_1} < \nostar{x_2} < \nostar{x_3} < \nostar{x_4}$, and let $k\in[4]$ be such that $x_k = e_1$.
  Since $m_{x_i,x_k}^{T_1} + m_{x_i,x_k}^{T_2} \equiv k-i \pmod{2}$,
  by~\ref{item:W2''}, we have that
  \begin{align*}
    \sum_{i=1}^4 \tilde{W}(x_i^*)\tilde{Y}(x_i)
    &=
    \tilde{W}(x_k^*)\tilde{Y}(x_k)
    \sum_{i=1}^4 \frac{\tilde{W}(x_i^*)\tilde{Y}(x_i)}{\tilde{W}(x_k^*)\tilde{Y}(x_k)} \\
    &=
    \tilde{W}(x_k^*)\tilde{Y}(x_k)
    \sum_{i=1}^4 (-1)^{k-i}
    \frac{\varphi(T_1\symdiff\skewpair{x_i})\varphi(T_2\symdiff\skewpair{x_i})}{\varphi(T_1 \symdiff\skewpair{x_k})\varphi(T_2 \symdiff\skewpair{x_k})}
    \in N_F.
  \end{align*}
  By~\ref{item:Ot2'}, $\tilde{W}(e_i^*)\tilde{Y}(e_i)= \tilde{W}(e_i^*)\tilde{Z_i}(e_i) = -\tilde{W}(g^*)\tilde{Z_i}(g)$ for each $i$.
  Since $W(g^*)\ne 0$ and $Y(f^*) = 0$, we conclude that $Y(g) - Z_1(g) - Z_2(g) \in N_F$.
  Therefore~\ref{item:L1} holds.

  \smallskip
  Finally, we show that~\ref{item:L2} holds.
  Let $B$ be a basis of $\ul{M}_\varphi$ and let $X_1$, $X_2$, $X_3$ be $F$-circuits in $\cC$ such that their supports are $C(B,e_1)$, $C(B,e_2)$, $C(B,e_3)$ for some distinct $e_1, e_2, e_3\in B^*$, and $X_i(e_j^*) \ne 0$ for all $i\ne j$.
  Then $B \symdiff \{e_1,e_1^*,e_2,e_2^*\}$, $B \symdiff \{e_1,e_1^*,e_3,e_3^*\}$, and $B \symdiff \{e_2,e_2^*,e_3,e_3^*\}$ are all bases.
  Let $Y$ be an $F$-circuit in $\cC$ whose support is $C(B \symdiff \{e_1,e_1^*,e_2,e_2^*\}, e_3)$.
  Then $\{e_1,e_2,e_3\} \subseteq \underline{Y} \subseteq B \symdiff \{e_1,e_1^*,e_2,e_2^*,e_3,e_3^*\}$.
  We claim that $\tilde{Y}$ belongs to the linear span of $\tilde{X_i}$ with $i=1,2,3$.
  We may assume that $X_i(e_i) = Y(e_i)$ for each $i$.
  Hence it suffices to show that $\tilde{Y}(f) - \tilde{X_1}(f) - \tilde{X_2}(f) - \tilde{X_3}(f) \in N_F$ for all $f \in B$.
  Denote $\alpha := \tilde{X_1}(e_2^*) \tilde{Y}(e_2) \in F^\times$.
  Then $\tilde{X_1}(e_3^*) \tilde{Y}(e_3) = -\alpha$ by~\ref{item:Ot2'} applied to~$X_1$ and~$Y$.
  Applying again~\ref{item:Ot2'} to $X_1$ and $X_2$, we have $\tilde{X_2}(e_1^*) \tilde{Y}(e_1) = -\alpha$.
  Similarly, we deduce that $\tilde{X_2}(e_3^*)\tilde{Y}(e_3) = \tilde{X_3}(e_1^*)\tilde{Y}(e_1) = -\tilde{X_3}(e_2^*)\tilde{Y}(e_2) = \alpha$.
  Then $X_{i+1}(e_{i}^*)+X_{i+2}(e_{i}^*) \in N_F$ for each $i$, where the subscripts are read modulo $3$.
  Thus we may assume that $f \ne e_1^*,e_2^*,e_3^*$.

  Let $T_1 := B \symdiff \skewpair{f}$ and $T_2 := B\symdiff \{e_1,e_1^*,e_2,e_2^*,e_3,e_3^*\} \supseteq \underline{Y}$.
  Let $Z$ be an $F$-circuit in $\cC$ such that $\underline{Z} = C(B,f^*) \subseteq T_1$.
  Note that $X_i(f) \ne 0$ if and only if $T_1 \symdiff \skewpair{e_i}$ is a basis.
  Hence if $X_i(f) = 0$ for all $i$, then by~\ref{item:W2''}, $\varphi(T_2 \symdiff \skewpair{f}) = 0$ so $Y(f) = 0$.
  Therefore we may assume that at least one of $X_i(f)$ is nonzero.
  By relabelling, we may assume that $X_1(f) \ne 0$ and hence $T_1 \symdiff \skewpair{e_1}$ is a basis.
  Then $Z(e_1) \ne 0$.
  
  We rewrite $\{e_1,e_2,e_3,f^*\} \subseteq T_2$ by $\{x_1,x_2,x_3,x_4\}$ with $\nostar{x_1} < \nostar{x_2} < \nostar{x_3} < \nostar{x_4}$, and let $k \in [4]$ be such that $x_k = e_1$.
  Note that $m_{x_i,x_k}^{T_1} + m_{x_i,x_k}^{T_2} \equiv k-i \pmod{2}$.
  Then by~\ref{item:W2''},
  \begin{align*}
    \sum_{i=1}^4 \tilde{Z}(x_i^*)\tilde{Y}(x_i)
    &=
    \tilde{Z}(x_k^*)\tilde{Y}(x_k)
    \sum_{i=1}^4 (-1)^{k-i}
    \frac{\varphi(T_1\symdiff\skewpair{x_i})\varphi(T_2\symdiff\skewpair{x_i})}
    {\varphi(T_1\symdiff\skewpair{x_k})\varphi(T_2\symdiff\skewpair{x_k})}
    \in N_F.
  \end{align*}
  By~\ref{item:Ot2'}, $\tilde{Z}(e_i^*)\tilde{Y}(e_i) = \tilde{Z}(e_i^*)\tilde{X_i}(e_i) = -\tilde{Z}(f^*)\tilde{X_i}(f)$ for each $i$.
  Because $f^* \in \underline{Z}$, we deduce that $\tilde{Y}(f) - \sum_{i=1}^3 \tilde{X_i}(f) = \tilde{Y}(f) + \tilde{Z}(f^*)^{-1} \sum_{i=1}^3 \tilde{Z}(e_i^*)\tilde{Y}(e_i) \in N_F$.
\end{proof}

\subsection{Strong orthogonal signatures and strong circuit sets}\label{sec:OC}

Let $\cC$ be an $F$-signature of an orthogonal matroid $\ul{M}$ on $E$ satisfying~\ref{item:Ot2}. We say that $X \in F^E$ is {\em consistent with $\cC$} if for each basis $B$ of $\ul{M}$, the vector $\tilde{X}$ belongs to the linear span of $\{\tilde{X_e}: e\in B^*\}$, where $X_e$ is the unique $F$-circuit in $\cC$ such that $\underline{X_e} = C(B,e)$ and $X_e(e)=1$.
Hence~\ref{item:L} is equivalent to that every $F$-circuit in $\cC$ is consistent with $\cC$.

The {\em orthogonal complement} of $\cW \subseteq F^E$ is $\cW^\perp := \{X \in F^E : \langle X,Y^* \rangle \in N_F \text{ for all } Y \in \cW\}$.
Therefore, the orthogonality~\ref{item:O} is equivalent to that $\cC \subseteq \cC^\perp$.

\begin{lemma}\label{lem:consistent to perp}
  Let $\mathcal{C}$ be an $F$-signature of an orthogonal matroid on $E$ satisfying~\ref{item:Ot2}.
  If $X \in F^E$ is consistent with $\cC$, then $X \in \mathcal{C}^\perp$.
\end{lemma}

\begin{proof}
  We claim that $\langle X, Y^\ast \rangle \in N_F$ for all $Y \in \mathcal{C}$.
  We may assume that $\underline{X} \cap \underline{Y}^* \neq \emptyset$.
  Write $\underline{X} \cap \underline{Y}^* = \{e_0^*,e_1,\dots,e_\ell\}$, and let $B$ be a basis of the underlying orthogonal matroid $\underline{M}_\mathcal{C}$ such that $\underline{Y} \symdiff \skewpair{e_0} \subseteq B$.
  Then $\{e_0^*,\dots,e_\ell^*\} \subseteq B$.
  We denote by $m := |\underline{X} \cap B^*|$, and if $\underline{X} \cap (B \setminus \underline{Y})^*$ is nonempty, then we enumerate its elements as $e_{\ell+1}, e_{\ell+2}, \dots, e_m$.
  Then $\underline{X} \cap B^* = \{e_1,\dots,e_m\}$.
  For $0 \leq i \leq m$, let $X_i$ be the $F$-circuit in $\mathcal{C}$ such that $\underline{X_i} = C(B,e_i)$ and $X_i(e_i) = 1$.
  Then $\conj{X} - \sum_{i=1}^m \conj{X}(e_i) \conj{X}_i \in (N_F)^{E}$ since $X$ is consistent with $\mathcal{C}$.
  Note that $\underline{Y} = C(B,e_0) = \underline{X_0}$.
  By multiplying $Y$ with $Y(e_0)^{-1} \in F^\times$, we can assume that $Y(e_0) = 1$.
  For each $1\leq i \leq m$, $\conj{X}_0(e_i^*) + \conj{X}_i(e_0^*) = \langle X_0, X_i^\ast \rangle \in N_F$ by~\ref{item:Ot2} and so $\conj{Y}(e_i^*) = -\conj{X}_i(e_0^*)$.
  Therefore,
  \begin{align*}
    \langle X, Y^\ast \rangle
    &=
    \conj{X}(e_0^*) + \sum_{i=1}^m \conj{X}(e_i) \conj{Y}(e_i^*)
    =
    \conj{X}(e_0^*) - \sum_{i=1}^m \conj{X}(e_i) \conj{X}_i(e_0^*)
    \in N_F.\qedhere
  \end{align*}
\end{proof}

\begin{lemma}\label{lem:perp to consistent}
  Let $\mathcal{C}$ be an orthogonal $F$-signature of an orthogonal matroid on $E$.
  If $X\in \mathcal{C}^\perp$, then $X$ is consistent with $\cC$.
\end{lemma}

\begin{proof}
  Let $B$ be a basis of $\underline{M}_\cC$. Write $\underline{X} \cap B^* = \{e_1,\dots,e_m\}$, and let $X_i$ be the $F$-circuit in $\mathcal{C}$ such that $\underline{X_i} = C(B,e_i)$ and $X_i(e_i) = 1$. We claim that $\conj{X}(f) - \sum_{i} \conj{X}(e_i) \conj{X}_i(f) \in N_F$ for all $f \in E$.
  We may assume that $f\in B$.
  Let $Y \in \mathcal{C}$ be such that $\underline{Y} = C(B,f^*)$ and $Y(f^*) = 1$.
  If $f^* = e_i$, then $X_i(f) = X_i(e_i^*) = 0$ and $Y(e_i^*) = Y(f) = 0$. 
  Otherwise, we have $\conj{X}_i(f) + \conj{Y}(e_i^*) = \langle X_i, Y^\ast \rangle \in N_F$ and hence $-\conj{X}_i(f) = \conj{Y}(e_i^*)$.
  Therefore, by the orthogonality~\ref{item:O},
  \begin{align*}
    \conj{X}(f) - \sum_{i} \conj{X}(e_i) \conj{X}_i(f)
    &=
    \conj{X}(f) + \sum_{i} \conj{X}(e_i) \conj{Y}(e_i^*)
    =
    \langle X, Y^\ast \rangle \in N_F.
    \qedhere
  \end{align*}
\end{proof}

We now prove Theorem~\ref{thm:strong orthogonal signatures and circuit sets} using the previous lemmas.

\begin{proof}[Proof of Theorem~\ref{thm:strong orthogonal signatures and circuit sets}]
    Let $\cC$ be an $F$-signature of an orthogonal matroid.
    Suppose that $\cC$ is orthogonal.
    Then $\cC \subseteq \cC^\perp$ and $\cC$ satisfies~\ref{item:Ot2}.
    By Lemma~\ref{lem:perp to consistent}, $\cC$ satisfies~\ref{item:L}. Conversely, suppose $\cC$ is a strong $F$-circuit set, then by Lemma~\ref{lem:consistent to perp}, we deduce that $\cC \subseteq \cC^\perp$, or equivalently, $\cC$ is orthogonal.
\end{proof}

\subsection{Orthogonal signatures and orthogonal vector sets}\label{sec:OV}

In~\cite{Anderson2019}, Anderson showed the equivalence between strong $F$-matroids and $F$-vector sets for matroids.
The orthogonal complement of an $F$-cocircuit set of an ordinary matroid $\ul{M}$ (i.e., an $F$-circuit set of the dual matroid $\ul{M}^*$) is an $F$-vector set of $\ul{M}$, and nonzero vectors having minimal supports in an $F$-vector set of $\ul{M}$ form an $F$-cocircuit set of $\ul{M}$.
We prove that the strong orthogonal $F$-signatures and the orthogonal $F$-vector sets can be derived from each other in a similar sense.

\begin{lemma}\label{lem:suppbase}
  Let $\mathcal{V}$ be an orthogonal $F$-vector set. Then there exists an ordinary orthogonal matroid $\ul{M}$ whose set of bases equals the set of support bases of $\mathcal{V}$. 
  Furthermore, the set of supports of elementary vectors in $\cV$ equals the set of circuits of $\ul{M}$. 
\end{lemma}

\begin{proof}
  Let $\ul{\cB}$ be the set of support bases of $\cV$.
  It suffices to check that $\ul{\cB} \neq \emptyset$ and $\ul{\cB}$ satisfies the symmetric exchange axiom.
  
  We first show that $\ul{\cB} \neq \emptyset$.
  We may assume that $\cV$ has an elementary vector $X$, 
  since otherwise every transversal is a support basis. 
  Let $I_0 = \underline{X} \setminus \skewpair{e}$ for an arbitrary $e \in \underline{X}$.
  We say that a subtransversal is {\em $\mathcal{V}$-independent} if it does not contains any $\underline{Y}$ where $Y\in\mathcal{V}\setminus\{\mathbf{0}\}$.
  Then $I_0$ is $\mathcal{V}$-independent.

  We claim that if a subtransversal $I$ is $\mathcal{V}$-independent and $f \in [n] \setminus \nostar{I}$, then $I \cup \{f\}$ or $I \cup \{f^*\}$ is $\mathcal{V}$-independent.
  Suppose for contradiction that neither $I \cup \{f\}$ nor $I \cup \{f^*\}$  is $\mathcal{V}$-independent.
  Then there are $Y_1,Y_2 \in \mathcal{V} \setminus \{\mathbf{0}\}$ such that $\underline{Y_1} \subseteq I \cup \{f\}$ and $\underline{Y_2} \subseteq I \cup \{f^*\}$.
  We may assume that $Y_1$ and $Y_2$ are elementary.
  Since $I$ is $\mathcal{V}$-independent, $f \in \underline{Y_1}$ and $f^* \in \underline{Y_2}$.
  Then $\langle Y_1, Y_2^\ast \rangle = \conj{Y}_1(f) \conj{Y}_2(f^*) \not\in N_F$, which contradicts~\ref{item:V1}. By the claim, for $i = 0, 1, 2, \dots$, there is a $\cV$-independent set $I_{i+1}$ such that $I_i \subseteq I_{i+1}$ and $|I_{i+1}| = |I_i| + 1$, unless $|I_i| \geq n$.
  Then for $k := n -|I_0|$, the subtransversal $I_{k}$ is a $\cV$-independent set of size $n$ and hence $I_{k}$ is a support basis of $\cV$, implying that
  $\ul{\cB} \neq \emptyset$.

  Next we show that $\ul{\cB}$ satisfies the symmetric exchange axiom.
  Let $B_1, B_2 \in \ul{\cB}$ and $e\in B_1 \setminus B_2$.
  By~\ref{item:V2},
  there is a fundamental circuit form $\{X_g: g\in B_1^*\}$ of $\cV$ with respect to $B_1$, where $\underline{X_g} \subseteq B_1 \symdiff \skewpair{g}$ and $X_g(g) = 1$.
  Let $X := X_{e^*}$.
  Note that $X$ is elementary in $\cV$ by~\ref{item:V3}.
  Since $B_2$ is a support basis, $\underline{X} \not\subseteq B_2$.
  Thus there is $f\in \underline{X} \setminus B_2 \subseteq (B_1 \symdiff \skewpair{e}) \setminus B_2 = (B_1 \setminus B_2) \setminus \{e\}$. It suffices to show that $B_1 \symdiff \skewpair{e} \symdiff \skewpair{f}$ is a support basis of $\cV$.
  If not, then there is $Y \in \mathcal{V} \setminus \{\mathbf{0}\}$ with support $\underline{Y} \subseteq B_1 \symdiff \skewpair{e} \symdiff \skewpair{f}$.
  We may assume that $Y$ is elementary in $\cV$.
  By~\ref{item:V1}, $\conj{X}(f) \conj{Y}(f^*) = \langle X,Y^\ast \rangle \in N_F$ and thus $Y(f^*) = 0$.
  Then $\underline{Y} \subseteq B_1 \symdiff \skewpair{e}$.
  Since $B_1$ is a support basis, $e^* \in \underline{Y}$.
  By~\ref{item:V3}, $Y = Y(e^*) X$, which contradicts the fact that $Y(f) = 0 \neq X(f)$. 
  Therefore, $B_1 \symdiff \{e,e^*,f,f^*\}$ is a support basis.

  From the definitions of $\ul{\cB}$ and $\ul{M}$, it is straightforward to see that the set of circuits of $\ul{M}$ equals the set of supports of elementary vectors of $\cV$.
\end{proof}

In Lemma~\ref{lem:perp to consistent}, if we assume additionally that $X$ is elementary in $\cC^\perp$, then $X$ is indeed in~$\cC$ rather than merely being consistent with $\cC$, as the next lemma shows.
For $\cW \subseteq F^E$, let $\elem(\cW)$ be the set of elementary vectors in $\cW$.

\begin{lemma}\label{lem:MSS}
  Let $\mathcal{C}$ be an orthogonal $F$-signature of an orthogonal matroid on $E$. Then $\elem(\cC^\perp) = \cC$.
\end{lemma}

\begin{proof}
    Denote $\ul{M} := \underline{M}_{\cC}$.
    Note that $\cC \subseteq \cC^\perp$, since $\cC$ is orthogonal.

    We first show that $\elem(\cC^\perp) \supseteq \cC$.
    Suppose $X \in \cC$ is not elementary in $\cC^\perp$.
    Then there is $X' \in \cC^\perp \setminus \{\mathbf{0}\}$ such that $\underline{X'} \subsetneq \underline{X}$.
    Let $e \in \underline{X} 
    \setminus \underline{X}'$ and let $B$ be a basis of $\ul{M}$ containing $\underline{X} \symdiff \skewpair{e}$.
    Choose $f \in \underline{X'}$ and $Y \in \cC$ so that $\underline{Y} = C(B,f^*)$.
    Then $\langle X', Y^\ast \rangle = \conj{X'}(f) \conj{Y}(f^*) \not\in N_F$, a contradiction. 

    Next, we prove that $\elem(\cC^\perp) \subseteq \cC$.
    Let $X$ be an elementary vector in $\cC^\perp$.
    Suppose for contradiction that $\underline{X}$ is independent in $\ul{M}$.
    Take an element $e\in \underline{X}$ and a basis $B$ of $\ul{M}$ containing $\underline{X}$, and let $Y \in \mathcal{C}$ be such that $\underline{Y} = C(B,e^*)$.
    Then $\langle X, Y^\ast \rangle = \conj{X}(e) \conj{Y}(e^*) \not\in N_F$, a contradiction.
    Therefore, $\underline{X}$ is dependent in $\ul{M}$.
    Then there is $X' \in \cC$ such that $\underline{X'} \subseteq \underline{X}$.
    Since $\cC \subseteq \cC^\perp$ and $X$ is elementary in $\mathcal{C}^\perp$, we have $\underline{X} \subseteq \underline{X}'$.
    Hence $\underline{X} = \underline{X}'$. Now it suffices to show $X = \alpha X'$ for some $\alpha  \in F^\times$.
    For $e\in \underline{X}$, we may assume that $X(e) = X'(e) = 1$.
    Suppose that $X \neq X'$.
    Then $X(f) \neq X'(f)$ for some $f \in \underline{X}$.
    For a basis $B$ of $\ul{M}$ containing $\underline{X} \symdiff \skewpair{e}$, let $Y \in \mathcal{C}$ be such that $\underline{Y} = C(B,f^*)$ and $Y(f^*) = 1$.
    Because $\conj{X}(f) + \conj{Y}(e^*) = \langle X, Y^\ast \rangle \in N_F$, we have $\conj{X}(f) = -\conj{Y}(e^*)$.
    We similarly deduce that $\conj{X}'(f) = -\conj{Y}(e^*)$, which contradicts the fact that $X(f) \neq X'(f)$.
    Thus, $X = X' \in \mathcal{C}$.
\end{proof}


\begin{theorem}\label{thm:OV}
  The following hold:
  \begin{enumerate}[label=\rm(\roman*)]
    \item If $\mathcal{C}$ is an orthogonal $F$-signature, then $\mathcal{C}^\perp$ is an orthogonal $F$-vector set and $\cC = \elem(\cC^\perp)$.
    \item If $\cV$ is an orthogonal $F$-vector set, then $\elem(\cV)$ is an orthogonal $F$-signature and $\cV = \elem(\cV)^\perp$.
  \end{enumerate}
\end{theorem}

\begin{proof}
  (i)
  By Lemma~\ref{lem:MSS}, $\elem(\cC^\perp) = \cC$ and thus $\cC^\perp$ satisfies~\ref{item:V1}.
  In addition, the set of support bases of $\cC^\perp$ is equal to the set of bases of $\underline{M}_{\cC}$.
  Therefore, by~\ref{item:C5}, $\cC^\perp$ satisfies~\ref{item:V2}.
  By Lemmas~\ref{lem:consistent to perp} and \ref{lem:perp to consistent}, $\cC^\perp$ satisfies~\ref{item:V3}.

  (ii)
  Let $\cC := \elem(\cV)$.
  By~\ref{item:V1}, $\mathcal{C}$ satisfies the $2$-term orthogonality relation~\ref{item:Ot2}.
  By Lemma~\ref{lem:suppbase}, the set of support bases of $\cV$ coincides with the set of support bases of $\cC$. Moreover, it is the set of bases of some ordinary orthogonal matroid $\ul{M}$.
  Then
  $\mathcal{C}$ is an $F$-signature of $\ul{M}$ and
  every fundamental circuit form of $\cC$ is a fundamental circuit form of~$\cV$.
  Conversely, by~\ref{item:V3}, every fundamental circuit form of $\cV$ is a fundamental circuit form of~$\cC$.
  Therefore, $X\in F^E$ is in $\cV$ if and only if it is consistent with $\cC$.
  The latter condition implies that $X \in \cC^\perp$ by Lemma~\ref{lem:consistent to perp}.
  Then $\cC \subseteq \cV \subseteq \cC^\perp$.
  Therefore, $\cC$ is an orthogonal $F$-signature of $\ul{M}$.
  By Lemma~\ref{lem:perp to consistent}, if $X \in \cC^\perp$, then $X$ is consistent with $\cC$.
  Hence $\cC^\perp \subseteq \cV$ 
  and we conclude $\cC^\perp = \cV$.
\end{proof}

We finish the discussion of orthogonal vector sets with the proof of Theorem~\ref{thm:lagrangian subspace}(i) that if $F$ is a field, then every orthogonal $F$-vector set is a Lagrangian subspace.

\begin{proof}[Proof of Theorem~\ref{thm:lagrangian subspace}(i)]
    By~\ref{item:V2} and~\ref{item:V3}, $\cV$ is an $n$-dimensional linear subspace of $F^{[n] \cup [n]^*}$.
    Let $\cC := \elem(\cV)$.
    By Theorem~\ref{thm:OV}(ii), $\langle X, Y^* \rangle =0$ for all $X,Y \in \cC$ and $\cV = \cC^\perp$.
    By Lemma~\ref{lem:perp to consistent}, $\cV$ is the subspace spanned by $\cC$ and thus $\langle X, Y^* \rangle =0$ for all $X,Y \in \cV$.
    Hence $\cV$ is isotropic and therefore Lagrangian.
\end{proof}

\begin{example}
  By~{\cite[Corollary~3.45]{Baker2019}}, if $F$ is a {\em doubly distributive partial hyperfield} such as a field, $\bS$, $\T$, or $\K$, and if $M$ is a strong $F$-matroid, then every vector (resp. covector) of $M$ is orthogonal to all covectors (resp. vectors) of $M$.
  For the proof, it is crucial to show that if $F$ is a doubly distributive partial hyperfield, then every weak $F$-matroid is automatically a strong $F$-matroid.
  In the orthogonal case, if $F$ is a field and $\cW \subseteq F^{[n]\cup[n]^*}$ is an orthogonal $F$-vector set, then $\cW^\perp = \cW$ by Theorem~\ref{thm:lagrangian subspace}(i). 
  So one may ask naturally whether this fact can be generalized to doubly distributive partial hyperfields. However,
  it is false even if we take $F = \K$, the Krasner hyperfield.
  Let $\ul{N}$ be the orthogonal matroid on $[5]\cup[5]^*$ in which a transversal $B$ is a basis of $\ul{N}$ if and only if $|B \cap [5]|$ is even and $B \neq 1^*2345, 12^*345$.
  By computer search, we check that $|\cC| = 15$, $|\cV|=256$, and $|\cV^\perp|=169$, where $\cC$ is the unique $\K$-circuit set of $\ul{N}$ and $\cV := \cC^\perp$ is the corresponding orthogonal $\K$-vector set.
\end{example}

\subsection{Natural bijections}\label{sec:strong equiv}

Summarizing the results in Sections~\ref{sec:W2O}--\ref{sec:OV},
we prove the equivalence between various notions of orthogonal matroids with coefficients in tracts, described in Theorems~\ref{thm:main},~\ref{thm:main-weak1}, and~\ref{thm:main-weak2}.
As a corollary, we deduce Theorem~\ref{thm:weak orthogonal signatures and circuit sets}.

The following lemma is straightforward from definitions.
\begin{lemma}\label{lem:C and varphi}
  Let $F$ be a tract.
  Let $\cC$ be an $F$-signature of an orthogonal matroid satisfying~\ref{item:Ot2}, and let $\varphi$ be a weak Wick $F$-function.
  Then $\cC_{\varphi_\cC} = \cC$ and $[\varphi_{\cC_\varphi}] = [\varphi]$.
\end{lemma}


\begin{proof}[Proof of Theorem~\ref{thm:main}]
By Theorems~\ref{thm:W2O},~\ref{thm:O2W}, and Lemma~\ref{lem:C and varphi}, there is a natural bijection between (1) and (2).
By Theorem~\ref{thm:strong orthogonal signatures and circuit sets}, (2) and (3) are identical.
By Theorem~\ref{thm:OV}, there is a natural bijection between (2) and (4).
\end{proof}

\begin{proof}[Proof of Theorem~\ref{thm:main-weak1}]
  It is straightforward from Theorems~\ref{thm:W2O-weak},~\ref{thm:O2W-weak}, and Lemma~\ref{lem:C and varphi}.
\end{proof}

\begin{proof}[Proof of Theorem~\ref{thm:main-weak2}]
  It is concluded by Theorem~\ref{prop:weak circuit set to weak Wick function},~\ref{prop:weak Wick function to weak circuit set}, and Lemma~\ref{lem:C and varphi}.
\end{proof}

\begin{proof}[Proof of Theorem~\ref{thm:weak orthogonal signatures and circuit sets}]
  It is an immediate corollary of Theorems~\ref{thm:main-weak1} and~\ref{thm:main-weak2}.
\end{proof}

\subsection{More examples. }\label{sec:more-examples}
Strong orthogonal $F$-matroids generalize strong $F$-matroids by Proposition~\ref{prop:W and GP}, and strong orthogonal $F$-signatures of orthogonal matroids generalize strong dual pairs of $F$-signatures of matroids by Remark~\ref{rmk:OS and DP}.
Baker and Bowler showed in~\cite{Baker2019} the equivalence of weak $F$-matroids and weak dual pairs of $F$-signatures.
By Theorem~\ref{thm:main-weak1}, moderately weak orthogonal $F$-matroids and weak orthogonal $F$-signatures are equivalent.
However, in the previous equivalence, moderately weak orthogonal $F$-matroids cannot be replaced by weak orthogonal $F$-matroids, as the class of weak orthogonal $F$-matroids is strictly larger than the class of moderately weak orthogonal $F$-matroids for some tract $F$.
This is true even if we restrict the classes of weak and moderately weak orthogonal $F$-matroids to those whose underlying orthogonal matroids are lifts of matroids.

\begin{example}\label{eg:nonidyll}
  Let $F$ be the tract $(\{1\},\{1+1, 1+1+1\})$ with the trivial involution and let $\ul{M}$ be the lift of the uniform matroid $U_{3,6}$. The set of bases of $\ul{M}$ is $\{abc d^* e^* f^* : abcdef = [6]\}$. Since $F^\times = \{1\}$, the function $\varphi: \mathcal{T}_6 \to F$ whose support is the set of bases of $\ul{M}$ is uniquely determined. Because $\ul{M}$ is the lift of a matroid, for all transversals $T_1$ and $T_2$ with $|(T_1 \symdiff T_2) \cap [6]| = 4$, at most three of $\varphi(T_1 \symdiff \skewpair{i_j}) \varphi(T_2 \symdiff \skewpair{i_j})$ with $1\leq j\leq 4$ are nonzero, where $(T_1 \symdiff T_2) \cap [6] = \{i_1 < i_2 < i_3 < i_4\}$.
  Therefore, $\varphi$ is a weak Wick $F$-function.
  Consider $T_1'=\{1,2,3,4,5^*,6^*\}$ and $T_2'=(T_1')^*$, we have $\sum_{i=1}^6 (-1)^i \varphi(T_1'\symdiff\skewpair{i}) \varphi(T_2'\symdiff\skewpair{i}) = 1+1+1+1 \notin N_F$. Hence $\varphi$ is not a moderately weak Wick $F$-function. Similarly, if we take $\cC$ to be the unique $F$-signature of $\ul{M}$, then it is readily seen that $\cC$ is a weak $F$-circuit set but not a weak orthogonal $F$-signature.
  %
\end{example}

We also have an instance showing where the class of strong $F$-matroids is strictly larger than the class of moderately weak $F$-matroids, i.e., the class of strong orthogonal $F$-signatures is strictly larger than the class of weak orthogonal $F$-signatures.

\begin{example}\label{eg:nonidyll2}
  Let $F$ be the tract $(\{1\},\{1+1,1+1+1,1+1+1+1\})$ endowed with the trivial involution and let $\ul{M}$ be the lift of $U_{4,8}$.
  Let $\cC$ be the unique $F$-signature of $\ul{M}$.
  Then for $X,Y\in\cC$ whose supports are $[5]$ and $[5]^*$, respectively, we have $\langle X,Y^* \rangle = 1+1+1+1+1 \notin N_F$ and thus~\ref{item:O} does not hold.
  However, \ref{item:O'} holds obviously by our choice of $F$.
\end{example}

By Theorem~\ref{thm:weak matroid bijection}, 
if $\mathcal{C}$ is an $F$-signature of the lift of a matroid satisfying the $3$-term orthogonality~\ref{item:Ot3}, then $\varphi_\mathcal{C}$ is a weak Wick $F$-function. However, this is false in general for orthogonal matroids, even if $F$ is a field.

\begin{example}\label{eg:O3 but not W2'}
  Consider the $K$-signature $\cC$ defined in Example~\ref{eg:why not choose O3 and Lii}, which satisfies~\ref{item:Ot3} but not~\ref{item:O'}.
  Note that $\cC_{\varphi_\cC} = \cC$ and thus by Theorem~\ref{thm:W2O-weak}, $\varphi_\cC$ is not a moderately weak Wick $K$-function.
  Since $E(\ul{M}_\cC) = [4]\cup[4]^*$, \ref{item:W2'} and~\ref{item:W2''} are equivalent for $\varphi_\cC$.
  Thus $\varphi_\cC$ is not a weak Wick function.
  
  More precisely, we can compute $\varphi_\cC$ by setting $\varphi_\cC([4])=1$ and check whether it satisfies~\ref{item:W2''}.
  By definition, it is easily seen that
  \begin{align*}
    \varphi_\cC(B) =
    \begin{cases}
      1   & \text{if $B = [4]$ or $1^*23^*4$}, \\
      -1  & \text{if $B\in \{ijk^*\ell^*:ijk\ell=[4]\} \setminus \{1^*23^*4\}$}, \\
      -x   & \text{if $B = [4]^*$}, \\
      0   & \text{otherwise}.
    \end{cases}
  \end{align*}
  Then for $T_1 = 1234^*$ and $T_2 = 1^*2^*3^*4$,
  \begin{align*}
    \sum_{i=1}^4 (-1)^i \varphi(T_1 \symdiff \skewpair{i}) \varphi(T_2 \symdiff \skewpair{i})
    = -1 -1 -1 - x \ne 0
  \end{align*}
  since $x \in K\setminus\{0,-3\}$. Therefore, $\varphi_\cC$ does not satisfies~\ref{item:W2''}.
\end{example}


In Sections~\ref{sec:OM functionality} and~\ref{sec:OM minors}, we promised to show that the minors and the pushforward operations of an orthogonal $F$-vector set are not properly defined.
Recall that for $\cW \subseteq F^E$ and $e\in E$, $\cW|e = \{\pi(X) \in F^{E \setminus \skewpair{e}} : \text{$X\in \cW$ with $X(e^*) = 0$}\}$, where $\pi : F^E \to F^{E \setminus \skewpair{e}}$ is the canonical projection.
For an orthogonal $F$-signature $\cC$ and the corresponding $F$-vector set $\cV := \cC^\perp$, it is readily seen that $\cV|e \subseteq (\cC|e)^\perp$. Example~\ref{eg:minor of orthogonal vertex set} provides an instance where $\cV|e \neq (\cC|e)^\perp$.
If $f:F \to F'$ is a tract homomorphism commuting with involutions of $F$ and $F'$, one can check that $f_*(\cV) \subseteq (f_*(\cC))^\perp$.
It might not be an equality, as Example~\ref{eg:no pushforward for orthogonal vector sets} shows.

\begin{example}\label{eg:minor of orthogonal vertex set}
    Let $\ul{M}$ be the lift of $U_{1,3}$.
    Then $\cC(\ul{M}) = \{12, 13, 23, 1^* 2^* 3^*\}$.
    Consider the following orthogonal $\U_0$-signature of $\ul{M}$:  
    \begin{align*}
        \cC := \{(1,-1,0,0,0,0), \, (1,0,1,0,0,0), \, (0,1,1,0,0,0), \, (0,0,0,1,1,-1)\},
    \end{align*}
    where the coordinates of the vectors are indexed by $1,2,3,1^*,2^*,3^*$ in order.
    Let $\cV := \cC^\perp$ be the orthogonal $\U_0$-vector set.
    Then $\cV|3 = \left\{(1, -1, 0, 0), \, (1, 0, 0, 0), \, (0, 1, 0, 0) \right\}$ and $( \cC | 3 )^\perp = \cV | 3 \cup \left\{ (1, 1, 0, 0) \right\}$, where the coordinates of vectors are indexed by $1,2,1^*,2^*$ in order.
\end{example}

\begin{example}\label{eg:no pushforward for orthogonal vector sets}
    Similarly, let $\cC = \{(1,1,0,0,0,0), \, (1,0,1,0,0,0), \, (0,1,1,0,0,0), \, (0,0,0,1,1,1)\} \subseteq (\F_2)^{[3]\cup[3]^*}$ be the orthogonal $\F_2$-signature of the lift of $U_{1,3}$, where the coordinates of each vector are indexed by $1,2,3,1^*,2^*,3^*$ in order.
    Let $\cV := \cC^\perp$.
    Then it is an orthogonal $\F_2$-vector set by Theorem~\ref{thm:OV} and $(1,1,1,0,0,0) \not\in \cV$.
    For the tract homomorphism $f:\F_2 \to \K$, it is easily checked that $(1,1,1,0,0,0) \in (f_*(\cC))^\perp \setminus f_*(\cV)$ and $\elem(f_*(\cV)) = f_*(\cC)$.
    Thus, $f_*(\cV)$ is not an orthogonal $\K$-vector set by Theorem~\ref{thm:OV}.
\end{example}

\section{Applications}\label{sec:applications}

An ordinary orthogonal matroid ${M}$ is {\em representable} (resp. {\em weakly reprsentable}) over a tract $F$ if there is a strong (resp. weak) orthogonal $F$-matroid whose underlying orthogonal matroid is ${M}$.
When $F$ is a field, the representability of orthogonal matroids was introduced using skew-symmetric matrices in~\cite{Bouchet1988repre}, and coincides with our definition by~{\cite[Theorem~2.2]{Wenzel1993}}.
Note that whenever ${M}$ is the lift of a matroid ${N}$, the orthogonal matroid ${M}$ is representable over a field $K$ if and only if the matroid ${N}$ 
admits a usual matrix representation over $K$ by~{\cite[(4.4)]{Bouchet1988repre}}. 

The following theorem will be used repeatedly in this section. 

\begin{theorem}[Baker-Jin,~Theorem~4.3~of~\cite{BJ2022}]\label{thm:BJ}
    Let $P$ be a partial field and let $\varphi : \mathcal{T}_n \to P$ be a function.
    Then $\varphi$ is a strong Wick function if and only if it is a weak Wick function.
    In particular, an orthogonal matroid is representable over $P$ if and only if it is weakly representable over $P$.
\end{theorem}


For a tract $F$ and a nonnegative integer $k$, let $N_F^{\leq k}$ be the set of elements in $N_F \subseteq \N[F^\times]$ that are formal sums of at most $k$ elements of $F^\times$.
To check whether a map $\varphi:\cT_n \to F$ is a weak Wick function, we only need the information of $N_F^{\leq 4}$ rather than $N_F$. 

One impressive result in matroid theory is that if a matroid is representable over $\F_2$ and $\F_3$, then it is representable over all fields~\cite{Tutte1958}. Geelen extended this result to orthogonal matroids.

\begin{theorem}[Geelen,~Theorem~4.13~of~\cite{Geelen1996thesis}]\label{thm:regular}
  Let $M$ be an orthogonal matroid.
  Then the following are equivalent:
  \begin{enumerate}[label=\rm(\roman*)]
      \item $M$ is representable over $\F_2$ and $\F_3$.
      \item $M$ is representable over the regular partial field $\mathbb{U}_0$.
      \item $M$ is representable over all fields. 
  \end{enumerate}
\end{theorem}

The proof in~\cite{Geelen1996thesis} involves technical matrix calculations. However, using the theory of orthogonal matroids over tracts, we are able to give a short and conceptual proof. 

\begin{proof}
If $M$ is representable over $\F_2$ and $\F_3$ via strong Wick functions $\varphi_1$ and $\varphi_2$, respectively, then by Proposition~\ref{prop:product}, $\varphi_1 \times \varphi_2$ is a strong Wick function over $\F_2 \times \F_3$ with underlying orthogonal matroid $M$.
Let $f$ be the map from the set $\F_2 \times \F_3 = \{0,(1,\pm 1)\}$ to the set $\U_0 = \{0, \pm 1\}$ given by $f(1,\pm 1) = \pm 1$ and $f(0) = 0$, then we have $f(N_{\F_2 \times F_3}^{\leq 4}) = N_{\U_0}^{\leq 4}$.
Therefore, $\varphi_0 := f\circ (\varphi_1 \times \varphi_2)$ is a weak Wick function over $\U_0$ and hence a strong Wick function by Theorem~\ref{thm:BJ}, and we have (i) implies (ii). 
For every field $F$, since there is a natural tract homomorphism $\mathbb{U}_0 \to F$ induced by the map $\Z \to F$, we have (ii) implies (iii) using Proposition~\ref{prop:pushforward Wick}. 
It is trivial that (iii) implies (i). 
\end{proof}

It is worth noting that the map $f$ defined in the above proof is not a tract homomorphism.

We say that an orthogonal matroid is {\em regular} if it satisfies one of the three equivalent conditions in Theorem~\ref{thm:regular}. We now give two more characterizations of regular orthognal matroids without a specific minor $\ul{M_4}$ on $[4]\cup[4]^*$ whose bases are
\[
    \{abc d^*, a^* b^* c^* d : abcd = [4]\}.
\]

An {\em ordered field} is a field together with a strict total order $\prec$ such that for every $x, y, z \in F$, (i) if $x \prec y$, then $x + z \prec y + z$, and (ii) if $0 \prec x$ and $0 \prec y$, then $0 \prec xy$. For instance, the real field $\R$ with the usual order is an ordered field.

\begin{theorem}\label{thm:regular with no M4}
    Let $M$ be an orthogonal matroid with no minor isomorphic to $\ul{M_4}$ and
    let $(K,\prec)$ be an ordered field.
    Then the following are equivalent:
    \begin{enumerate}[label=\rm(\roman*)]
        \item $M$ is regular.
        \item $M$ is representable over $\F_{2}$ and $K$.
        \item $M$ is representable over $\F_{2}$ and the sign hyperfield $\bS$.
    \end{enumerate}
\end{theorem}

To show Theorem~\ref{thm:regular with no M4}, we need the following lemma on orthogonal matroids with no minor isomorphic to $\ul{M_4}$.

\begin{lemma}\label{lem:M4 free}
Let $F$ be a tract and $\varphi$ a weak Wick function over $F$. If $\underline{M}_\varphi$ has no minor isomorphic to $\ul{M_4}$, then for all transversals $T_1$ and $T_2$ with $T_1 \setminus T_2 = \{i_1,i_2,i_3,i_4\}$, at least one of $\varphi(T_1 \symdiff \skewpair{i_j}) \varphi(T_2 \symdiff \skewpair{i_j})$ with $j\in[4]$ is zero. 
\end{lemma}
\begin{proof}
    Suppose for contradiction that all products are nonzero.
    Then all of the eight transversals $T_k \symdiff \skewpair{i_j}$ with $k\in[2]$ and $j\in[4]$ are bases of $\underline{M}_\varphi$.
    Let $S := T_1 \setminus \{i_1,i_2,i_3,i_4\}$.  
    Then $M|S$ is isomorphic to $\ul{M_4}$, a contradiction.
\end{proof}

\begin{proof}[Proof of Theorem~\ref{thm:regular with no M4}]
If $M$ is representable over $\F_2$ and $\bS$ via Wick functions $\varphi_1$ and $\varphi_2$, respectively, then by Proposition~\ref{prop:product}, $\varphi_1 \times \varphi_2$ is a Wick function over $\F_2 \times \bS$ with underlying orthogonal matroid $M$.
Let $g$ be the map from the set $\F_2 \times \bS = \{0,(1,\pm 1)\}$ to the set $\U_0 = \{0,\pm 1\}$ given by $g(0)= 0$ and $g(1,\pm 1) = \pm 1$.
Then $g(N_{\F_2 \times \bS}^{\leq 3}) = N_{\U_0}^{\leq 3}$.
Hence by Lemma~\ref{lem:M4 free}, $\varphi_0 := g\circ (\varphi_1 \times \varphi_2)$ is a weak Wick function over $\U_0$.
By Theorem~\ref{thm:BJ}, $\varphi_0$ is a strong Wick function, and we have (iii) implies (i).
The direction (i) implies (ii) follows trivially from Theorem~\ref{thm:regular}. Finally, let $\sigma : K \to \bS$ be such that
    \[
        \sigma(x) :=
        \begin{cases}
            1 & \text{if $x \succ 0$},\\
            0 & \text{if $x = 0$}, \\
            -1 & \text{otherwise}.
        \end{cases}
    \]
    Then $\sigma$ is a tract homomorphism and thus we have (ii) implies (iii) by Proposition~\ref{prop:pushforward Wick}. 
\end{proof}

\begin{remark}
The condition that an orthogonal matroid $M$ does not have minors isomorphic to $\ul{M_4}$ is sufficient but not necessary for the characterizations of regular orthogonal matroids in Theorem~\ref{thm:regular with no M4}. In fact, $\ul{M_4}$ itself is representable over the regular partial field $\mathbb{U}_0$ by setting $\varphi(T) = 1$ if $T$ is a basis, and $\varphi(T) = 0$ otherwise, and hence representable over all fields and the sign hyperfield $\bS$. It is still an open question whether Theorem~\ref{thm:regular with no M4} holds for all orthogonal matroids. 
\end{remark}

Duchamp~\cite[Proposition~1.5]{Duchamp1992} proved that an orthogonal matroid $M$ is isomorphic to a twisting of the lift of a matroid if and only if $M$ has no minor isomorphic to the orthogonal matroid $\ul{M_3}$ on $[3]\cup[3]^*$ whose set of bases is 
  \begin{align*}
      \cB(\ul{M_3}) = \{abc^*: abc=[3]\} \cup \{[3]^*\}.
  \end{align*}
Note that $\ul{M_3} = \ul{M_4}|4$. So in particular, if $M$ is isomorphic to the lift of a matroid, then it does not have minors isomorphic to $\ul{M_4}$. As a consequence, we have: 

\begin{corollary}[Bland-Las Vergnas,~\cite{Bland1978}]
A matroid is regular if and only if it is binary and orientable, if and only if the matroid is binary and representable over the reals.
\end{corollary} 

We also extend Whittle's theorem~\cite[Theorem~1.2]{Wh97} that a matroid is representable over both $\F_3$ and $\F_4$ if and only if it is representable over the sixth-root-of-unity partial field $R_6$ to orthogonal matroids. 

\begin{theorem}\label{thm:sixth-roots-of-unity}
    Let $M$ be an orthogonal matroid.
    Then the following are equivalent:
    \begin{enumerate}[label=\rm(\roman*)]
        \item $M$ is representable over the sixth-root-of-unity partial field $R_6$.
        \item $M$ is representable over $\F_3$ and $\F_4$.
        \item $M$ is representable over $\F_3$, $\F_{p^2}$ for all primes $p$, and $\F_q$ for all primes $q$ with $q \equiv 1 \pmod{3}$.   
    \end{enumerate}
\end{theorem}

To show Theorem~\ref{thm:sixth-roots-of-unity}, we need the following lemma on $R_6$. 

\begin{lemma}[van~Zwam,~Lemma~2.5.12~and~Table~4.1~of~\cite{vanZwam2009thesis}]\label{lem:sixth-roots-of-unity}
    Let $p$ be a prime. 
    \begin{enumerate}
        \item There is a tract homomorphism $R_6 \to \F_{p^2}$.
        \item If $p \equiv 1 \pmod{3}$, then there is a tract homomorphism $R_6 \to \F_{p}$.
        \item There is a tract isomorphism $R_6 \cong \F_{3} \times \F_{4}$.
    \end{enumerate}
\end{lemma}

\begin{proof}[Proof of Theorem~\ref{thm:sixth-roots-of-unity}]
 The proof is a straightforward application of Propositions~\ref{prop:pushforward Wick},~\ref{prop:product}, and Lemma~\ref{lem:sixth-roots-of-unity}, and is similar to the proof of Theorem~\ref{thm:regular}. In particular, the only nontrivial part that if $M$ is representable over $\F_3$ and $\F_4$ then $M$ is representable over $R_6$ is guaranteed by the tract isomorphism $R_6 \cong \F_{3} \times \F_{4}$ and Proposition~\ref{prop:product}. 
\end{proof}





\printbibliography

\end{document}